\documentclass[reqno,a4paper,12pt]{amsart}
\usepackage{amsmath,amsthm,amssymb,mathrsfs}
\usepackage[
	left = 2.9cm,
	right = 2.9cm,
	top = 2.5cm,
	bottom = 2.5cm
]{geometry}
\usepackage{graphicx,xcolor,tikz,mathtools}
\usepackage{bm,stmaryrd}
\usepackage{enumitem}
\usepackage[toc,page]{appendix}
\usepackage[
	colorlinks = true,
	citecolor = black,
	linkcolor = black,
	anchorcolor = black
        ]{hyperref}

\usepackage[utf8]{inputenc}
\newtheorem{theorem}{Theorem}[section]

\newtheorem{lemma}{Lemma}[section]
\newtheorem{proposition}{Proposition}[section]
\newtheorem{remark}{Remark}[section]

\numberwithin{equation}{section}

\newcommand{\bbr}{\mathbb{R}}
\newcommand{\bbI}{\mathbb{I}}
\newcommand{\bbn}{\mathbb{N}}

\newcommand{\bv}{\bm{v}}
\newcommand{\vf}{\bv_{f}}
\newcommand{\vs}{\bv_{s}}
\newcommand{\bu}{\bm{u}}

\newcommand{\us}{\bu_{s}}
\newcommand{\hv}{\hat{\bv}}
\newcommand{\hu}{\hat{\bu}}
\newcommand{\hvf}{\hv_{f}}
\newcommand{\hvs}{\hv_{s}}
\newcommand{\hvone}{\hv^1}
\newcommand{\hvtwo}{\hv^2}
\newcommand{\hvfone}{\hvone_f}
\newcommand{\hvsone}{\hvone_s}
\newcommand{\hvftwo}{\hvtwo_f}
\newcommand{\hvstwo}{\hvtwo_s}
\newcommand{\hO}{\widehat{\Omega}}
\newcommand{\bF}{\bm{F}}
\newcommand{\Fs}{\bm{F}_s}
\newcommand{\Fse}{{}\Fs^e}

\newcommand{\hF}{\hat{\bF}}
\newcommand{\hFf}{\hF_{f}}
\newcommand{\hFs}{\hF_{s}}
\newcommand{\hFse}{{}\hFs^e}
\newcommand{\hFsg}{{}\hFs^g}
\newcommand{\Oft}{\Omega_{f}^{t}}
\newcommand{\Ost}{\Omega_{s}^{t}}
\newcommand{\Osg}{\Omega_{s}^{g}}
\newcommand{\Gt}{\Gamma^{t}}
\newcommand{\Of}{\Omega_{f}}
\newcommand{\Os}{\Omega_{s}}

\newcommand{\Gs}{\Gamma_s}
\newcommand{\Gst}{\Gamma_s^t}
\newcommand{\OM}{\Omega}
\newcommand{\tO}{\widetilde{\OM}}
\newcommand{\hnab}{\hat{\nabla}}
\newcommand{\hJ}{\hat{J}}
\newcommand{\hJf}{\hJ_f}
\newcommand{\hJs}{\hJ_s}
\newcommand{\hJse}{\hJs^e}
\newcommand{\hJsg}{\hJs^g}
\newcommand{\hg}{\hat{g}}
\newcommand{\rf}{\rho_f}
\newcommand{\rs}{\rho_s}
\newcommand{\hr}{\hat{\rho}}
\newcommand{\hrf}{\hr_f}
\newcommand{\hrs}{\hr_s}
\newcommand{\hc}{\hat{c}}
\newcommand{\hcf}{\hc_f}
\newcommand{\hcs}{\hc_s}
\newcommand{\hcss}{\hcs^*}
\newcommand{\nuf}{\nu_f}
\newcommand{\nus}{\nu_s}
\newcommand{\mus}{\mu_s}
\newcommand{\hnu}{\hat{\nu}}
\newcommand{\hnuf}{\hnu_f}
\newcommand{\hnus}{\hnu_s}
\newcommand{\hmu}{\hat{\mu}}
\newcommand{\hmus}{\hmu_s}
\newcommand{\vp}{\varphi}
\newcommand{\cL}{\mathcal{L}}

\newcommand{\bphi}{\bar{\phi}}
\newcommand{\wo}{w_0}

\newcommand{\bsigma}{\bm{\sigma}}
\newcommand{\sigf}{\bsigma_f}
\newcommand{\sigs}{\bsigma_s}
\newcommand{\sigse}{\sigs^{e}}
\newcommand{\sigsv}{\sigs^v}
\newcommand{\hsigma}{\hat{\bsigma}}
\newcommand{\hsigf}{\hat{\bsigma}_f}
\newcommand{\hsigs}{\hat{\bsigma}_s}
\newcommand{\hsigse}{\hsigs^{e}}
\newcommand{\hsigsv}{\hsigs^v}
\newcommand{\bn}{\bm{n}}
\newcommand{\hn}{\hat{\bn}}
\newcommand{\ngt}{\bn_{\Gt}}
\newcommand{\ngst}{\bn_{\Gst}}
\newcommand{\hng}{\hn_{\Gamma}}
\newcommand{\hngs}{\hn_{\Gs}}

\newcommand{\bT}{\bm{T}}
\newcommand{\hT}{\hat{\bT}}
\newcommand{\hpi}{\hat{\pi}}
\newcommand{\pif}{\pi_f}
\newcommand{\pis}{\pi_s}
\newcommand{\hpif}{\hpi_f}
\newcommand{\hpis}{\hpi_s}
\newcommand{\hpione}{\hpi^1}
\newcommand{\hpitwo}{\hpi^2}
\newcommand{\hpisone}{\hpione_s}
\newcommand{\hpistwo}{\hpitwo_s}

\newcommand{\cD}{\mathcal{D}}
\newcommand{\Dq}{\cD_q}
\newcommand{\Dqv}{\Dq^1}
\newcommand{\Dqc}{\Dq^2}
\newcommand{\YT}{Y_T}
\newcommand{\ZT}{Z_T}
\newcommand{\bS}{\bm{S}}

\newcommand{\bK}{\bm{K}}
\newcommand{\tk}{\tilde{\bK}}
\newcommand{\kf}{\tk_f}
\newcommand{\ks}{\tk_s}
\newcommand{\bk}{\bm{k}}
\newcommand{\bH}{\bm{H}}
\newcommand{\bh}{\bm{h}}

\newcommand{\FOinv}{\inv{\hF} - \bbI}
\newcommand{\FOinvtran}{\invtr{\hF} - \bbI}

\newcommand{\FfOinv}{\inv{\hFf} - \bbI}
\newcommand{\FfOinvtran}{\invtr{\hFf} - \bbI}
\newcommand{\FsO}{\hFs - \bbI}
\newcommand{\FsOinv}{\inv{\hFs} - \bbI}
\newcommand{\FsOinvtran}{\invtr{\hFs} - \bbI}
\newcommand{\TD}{T^{\delta}}
\newcommand{\TQ}{T^{\frac{1}{q'}}}
\newcommand{\TQQ}{T^{\frac{1}{2q'}}}
\newcommand{\Dh}{\Delta_h}
\newcommand{\Df}{D_f}
\newcommand{\Ds}{D_s}
\newcommand{\hD}{\hat{D}}
\newcommand{\hDf}{\hD_f}
\newcommand{\hDs}{\hD_s}
\newcommand{\vo}{\hv^0}
\newcommand{\co}{\hc^0}

\newcommand{\tF}{\tilde{F}}
\newcommand{\tFf}{\tF_f}
\newcommand{\tFs}{\tF_s}

\newcommand{\barv}{\bar{\bv}}

\newcommand{\barvs}{\barv_s}
\newcommand{\barpi}{\bar{\pi}}

\newcommand{\barpis}{\barpi_s}
\newcommand{\tv}{\tilde{\bv}}

\newcommand{\tvs}{\tv_s}
\newcommand{\tpi}{\tilde{\pi}}

\newcommand{\tpis}{\tpi_s}
\newcommand{\bR}{\bm{R}}
\newcommand{\vr}{\varrho}
\newcommand{\nusigma}{\nu_{\Sigma}}
\newcommand{\sL}{\mathscr{L}}
\newcommand{\sN}{\mathscr{N}}
\newcommand{\sM}{\mathscr{M}}
\newcommand{\cE}{\mathcal{E}}
\newcommand{\cM}{\mathcal{M}}

\newcommand{\onehalf}{\frac{1}{2}}
\newcommand{\tr}{\mathrm{tr}}
\newcommand{\rv}[1]{\left. #1 \right\vert}
\newcommand{\tran}[1]{ {#1} ^{\top}}
\newcommand{\inv}[1]{ {#1} ^{-1}}
\newcommand{\invtr}[1]{ {#1} ^{-\top}}
\renewcommand{\d}{\mathrm{d}}
\newcommand{\pt}{\partial_t}
\newcommand{\jump}[1]{\left\llbracket #1 \right\rrbracket}
\newcommand{\abs}[1]{\left\vert #1 \right \vert}
\newcommand{\norm}[1]{\left\Vert #1 \right \Vert}
\newcommand{\seminorm}[1]{\left[ #1 \right]}
\newcommand{\inner}[2]{\left\langle #1 , #2 \right\rangle}
\newcommand{\Lq}[1]{L^{#1}}
\newcommand{\Lqa}[1]{L^{#1}_{(0)}}
\newcommand{\W}[1]{W^{#1}_{q}}
\newcommand{\WO}[1]{{_0W}^{#1}_{q}}
\newcommand{\WOO}[1]{W^{#1}_{q,0}}
\newcommand{\WA}[1]{W^{#1}_{q,(0)}}
\newcommand{\Bq}[1]{B^{#1}_{q,q}}

\newcommand{\Bqp}[1]{B^{#1}_{q,p}}
\renewcommand{\Bar}[1]{\overline{#1}}

\newcommand{\FSI}{\eqref{Fluid}--\eqref{initialc}}
\newcommand{\NFSI}{\eqref{Fluid0}--\eqref{initialv}}
\newcommand{\NLFSI}{\eqref{Fluid1}--\eqref{initialv1g}}

\DeclareMathOperator*{\Div}{\mathrm{div}}
\newcommand{\hdiv}{\widehat{\Div}}
\newcommand{\hDelta}{\widehat{\Delta}}
\allowdisplaybreaks[4]

\begin{document}

\title[Fluid-structure interaction problem for plaque growth]{On a  fluid-structure interaction problem for plaque growth}

\author{Helmut Abels}
\address{Fakult\"at f\"ur Mathematik, Universit\"at Regensburg, 93040 Regensburg, Germany}
\email{Helmut.Abels@ur.de}
\author{Yadong Liu}
\address{Fakult\"at f\"ur Mathematik, Universit\"at Regensburg, 93040 Regensburg, Germany}
\email{Yadong.Liu@ur.de}

\date{\today}


\thanks{Y. Liu was supported by the RTG 2339 ``Interfaces, Complex Structures, and Singular Limits'' of the German Science Foundation (DFG). The support is gratefully acknowledged.}

\begin{abstract}
	We study a free-boundary fluid-structure interaction problem with growth, which arises from the plaque formation in blood vessels. The fluid is described by the incompressible Navier-Stokes equation, while the structure is considered as a viscoelastic incompressible neo-Hookean material. Moreover, the growth due to the biochemical process is taken into account. Applying the maximal regularity theory to a linearization of the equations, along with a deformation mapping, we prove the well-posedness of the full nonlinear problem via the contraction mapping principle. 
\end{abstract}

\maketitle


\section{Introduction}
\subsection{The free-boundary fluid-structure interaction model}
In this paper, we consider a free-boundary fluid-structure interaction problem with growth, which is used to describe the plaque formation in a human artery. The blood is assumed to be the incompressible Navier-Stokes equation and the artery is modeled by an elastic equation with viscosity. Based on \cite{Yang2016}, where the model was proposed and simulated in a cylindrical domain, we analyze such problem in a bounded domain $ \Omega^t = \subset \bbr^{n} $, $ n \geq 2 $. See Figure \ref{domain}. 
\begin{figure}[h]
	\centering
	\begin{tikzpicture}[line width=0.75pt]
		\draw [scale = 0.9] plot[smooth cycle] 
		coordinates{(0,0) (0.5,1.5) (2,2.5) (6,3) (8,0) (6,-1.65) (4,-2) (2,-1.65)};
		\draw[shift ={(2.5 ,0.3)}, scale = 0.3] plot[smooth cycle] coordinates{(0,0) (0.5,2) (2,3.5) (7,3) (8,0) (6,-3) (4,-3) (2,-1.5)};
		\draw node at (4,0.5) {$\Oft$}; 
		\draw node at (1.5,0.6) {$\Ost$};
		\draw node at (6,0.5) {$\Gt$};
		\draw node at (2.5,3) {$\Gst$};
		\draw [line width=0.5pt, -stealth] (4.88,0.2) .. controls (5.5,0.24) .. (5.7,0.35);
		\draw [line width=0.5pt, -stealth] (2.8,2.5) .. controls (2.7,2.7) .. (2.7,2.7);
		\draw node at (3.5,-2.35) {$\OM^t$}; 
	\end{tikzpicture}
	\caption{Domain of the problem.}
	\label{domain}
\end{figure}
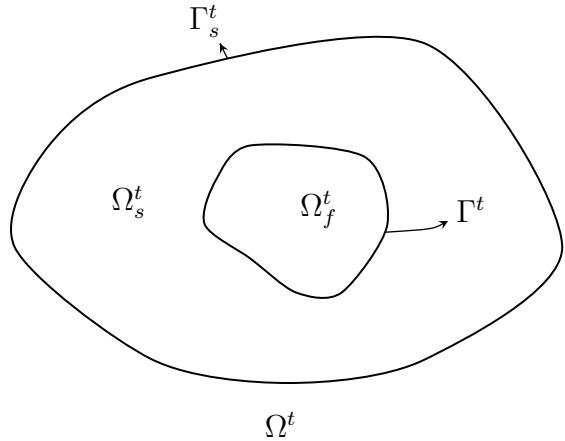
Here, $ \Omega^t = \Oft \cup \Ost \cup \Gt $, where $ \OM^t $ is divided by the interface $ \Gt $ into two disjoint parts, fluid domain $ \Oft $ and solid domain $ \Ost $. $ \Gst $ denotes the outer boundary of $ \OM^t $, which is also a free boundary. 

Before giving a precise description of the model, we introduce the setting of Lagrangian coordinate. For convenience, we define the moving domain at initial time $ t = 0 $ as $ \hO = \Of \cup \Os \cup \Gamma $, where $ \Of = \Omega_f^0 $, $ \Os = \Omega_s^0 $ and $ \Gamma = \Gamma^0 $. From the viewpoint of material deformation (see e.g. \cite{Ciarlet1988,Gurtin2010}), we set the so-called reference configuration at $ t = 0 $ and the deformed configuration at time $ t $. Moreover, we denote the spatial variable at $ t = 0 $ by the Lagrangian variable $ X $, resp., by the Eulerian variable $ x $ the spatial variable at $ t $. The velocities of deformations are $ \hv (X,t) $ and $ \bv(x,t) $ respectively. In the sequel, without special statement, the quantities or operators with a hat ``$ \hat{\cdot} $'' will indicate those in Lagrangian reference configuration. To formulate the model, we define the deformation as (See Figure \ref{deformation})
\begin{align*}
	\vp: \hO \rightarrow \Omega^t,
\end{align*}
with
\begin{align*}
	x = \vp(X,t) = X + \int_{0}^{t} \hv(X, \tau) \d \tau, \quad \forall X \in \hO,
\end{align*}
and $ \rv{x}_{t = 0} = \vp(X,0) = X $.
\begin{figure}[ht]
	\centering
	\begin{tikzpicture}[line width=0.75pt,scale = 0.9]
		\draw [shift ={(-8.45 ,0)}, scale = 0.7] plot[smooth cycle] 
		coordinates{(0,0) (0.5,2.5) (2,3.6) (6.5,3) (7.5,0) (6,-1.8) (4,-2) (2,-1.3)};
		\draw [shift ={(-6.5 ,0.3)}, scale = 0.2] plot[smooth cycle] coordinates{(0,0) (0.5,2) (2,3) (7,3.5) (9,0) (6,-2.5) (4,-2.5) (2,-1.5)};
		\draw node at (-5.5,0.5) {$\Of$}; 
		\draw node at (-7.4,0.35) {$\Os$};
		\draw node at (-5.5,-2) {$\hO$};
		\draw [-stealth] (-3,1) .. controls (-2,1.5) and (-1,1.5).. (0,1);
		\draw [scale = 0.7] plot[smooth cycle] 
		coordinates{(0,0) (0.5,1.5) (2,2.5) (6,3) (8,0) (6,-1.65) (4,-2) (2,-1.65)};
		\draw [shift ={(2 ,0.3)}, scale = 0.2] plot[smooth cycle] coordinates{(0,0) (0.5,2) (2,3.5) (7,3) (8,0) (6,-3) (4,-3) (2,-1.5)};
		\draw node at (3,0.25) {$\Oft$}; 
		\draw node at (1.45,0.3) {$\Ost$};
		\draw node at (3,-2) {$\OM^t$};
		\draw node at (-1.55,1.8) {$\vp$};
	\end{tikzpicture}
	\caption{Deformation $ \vp $ mapping from $ \hO $ into $ \OM $.}
	\label{deformation}
\end{figure}
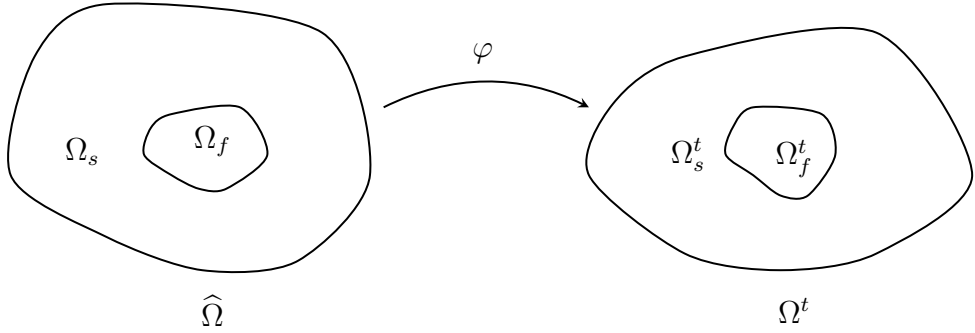

Subsequently, we denote by $ \hF $ the deformation gradient
\begin{equation}\label{DG}
	\hF = \frac{\partial}{\partial X} \vp(X,t) = \hnab \vp(X,t) = \bbI + \int_{0}^{t} \hnab \hv(X,\tau) \d \tau, \quad \forall X \in \hO,
\end{equation}
with initial deformation $ \rv{\hF}_{t = 0} = \bbI $ and by $ \hJ = \det \hF $ its determinant.
Conversely, we have the inverse deformation gradient by $ \bF = \inv{\hF} $. In the following, quantities in fluid and structure domain will be distinguished by subscript ``$ f $'' and ``$ s $'' respectively, quantities without subscript are defined both in fluid and structure domain.

Now, we are in the position to describe the whole system. For a finite time $ T > 0 $, $ 0 < t < T $, we model the motion of the fluid by the classical incompressible Navier-Stokes equations, which is
\begin{equation}
	\label{Fluid}
	\left.
		\begin{aligned}
			\rf \left(\pt + \vf \cdot \nabla \right) \vf & = \Div \sigf(\vf,\pif) \\
			\Div \vf & = 0 
		\end{aligned}
	\right\} \quad \text{in}\ \Oft,
\end{equation}
where $ \rf $ is the known constant fluid density. $ \sigf(\vf,\pif) = - \pif \bbI + \nuf ( \nabla \vf + \tran{\nabla} \vf )  $ denotes the Cauchy stress tensor, $ \pif $ is the unknown fluid pressure and $ \nuf $ represents the fluid viscosity.

The equations for the solid are written as:
\begin{equation}
	\label{Solid}
	\left.
	\begin{aligned}
		\rs \left(\pt + \vs \cdot \nabla \right) \vs & = \Div \sigs(\vs,\pis) \\
		\left(\pt + \vs \cdot \nabla \right) \rs + \rs \Div \vs & = f_s^g
	\end{aligned}
	\right\} \quad \text{in}\ \Ost,
\end{equation}
where $ \rs $ is the solid density. $ \sigs = \sigse + \sigsv $ is the stress tensor of the solid where 
\begin{align*}
	\sigse & = - \pis \bbI + \mus \left( \inv{(\Fse)} \invtr{(\Fse)} - \bbI \right), \\
	\sigsv & = \nus \left( \pt \inv{\Fs} + \pt \invtr{\Fs} \right) \invtr{\Fs},
\end{align*}
$ \pis $ is the unknown solid pressure. Moreover, $ \mus $ denotes the Lam\'e coefficient and $ \nus $ represents the solid viscosity, which are all positive constant. 
The first equation is the balance equation of linear momentum.
$ \sigse $ is given by the constitutive relation of an incompressible Neo-Hookean material, which is hyperelastic, isotropic and incompressible. This relationship was widely used to describe blood vessel wall by many investigators, see e.g. \cite{Tang2004,Yang2016}. The tensor $ \Fse $ is the inverse elastic deformation gradient under the assumption of growth and will be assigned later in Section \ref{growth}. We consider not only the elastic stress $ \sigse $, but also the viscoelastic stress $ \sigsv $, which could be deduced by linearizing the Kelvin-Voigt stress tensor, see Mielke and Roub\'{\i}\v{c}ek \cite{MR2020}. The second equation of \eqref{Solid} is due to the mass balance, where $ f_s^g $ is called growth function and represents the rate of mass growth per unit volume due to the formation of plaque, see e.g. \cite{AM2002,JC2012,Yang2016}.

\begin{remark}
	For short time existence, the Kelvin-Voigt viscous stress $ \sigsv $ we introduced brings the parabolicity to the solid equation, which dominates the regularity of solutions. Moreover, after linearization one obtains a two-phase Stokes type problem, which ensures us to get the solvabilities and regularities of fluid and solid velocities by maximal regularity theory. In a recent work \cite{BG2021}, a similar stress tensor of solid part was also considered to investigate weak solutions of the interaction between an incompressible fluid and an incompressible immersed viscous-hyperelastic solid structure.
\end{remark}
\begin{remark}
	In \cite{Yang2016,Tang2004}, some numerical simulations are carried out by considering that $ \mus $ depends on the concentration of some chemical species, and hence varies from healthy vessel to plaque area. In the case of viscoelasticity, $ \nus $ may also vary over the solid domain. However, to simplify the model for the analysis, we assume that these coefficients are constant over the solid domain.
\end{remark}

The interaction between the fluid and solid is modeled by transmission conditions on the interface $ \Gt $, which consists of the continuity of velocity and the balance of normal stresses:
\begin{align}
	\label{vjump}
	\jump{\bv} & = 0, \quad \text{on} \  \Gt, \\
	\label{sigjump}
	\jump{\bsigma} \ngt & = 0, \quad \text{on} \  \Gt,
\end{align}
where $ \ngt $ stands for the outer unit normal vertor on $ \Gt $ pointing from $ \Oft $ to $ \Ost $. For a quantity $ f $, $ \jump{f} $ denotes the jump defined on $ \Oft $ and $ \Ost $ across $ \Gt $, namely,
\begin{equation*}
	\jump{f}(x) := \lim_{\theta \rightarrow 0} {f(x + \theta \ngt(x)) - f(x - \theta \ngt(x))}, \quad \forall x \in \Gt.
\end{equation*}
Moreover, to ensure the compatibility with growth and incompressibility, the boundary condition on $ \Gst $ is assumed to be the so-called ``stress-free'' boundary condition:
\begin{equation}
	\label{bounds}
	\sigs \ngst = 0, \quad \text{on} \  \Gst,
\end{equation}
where $ \ngst $ is the unit outer normal vector on $ \Gst = \partial \OM^t $.
\begin{remark}
	\label{stress free boundary}
	The ``stress-free'' boundary condition \eqref{bounds} is set due to physical reality. Since we consider the growth of solid part and both fluid and solid part are incompressible, we can not impose some types of boundary conditions. For example, the clamped condition $ \vs = 0 $ on $ \Gst $ (correspondingly, $ \vs = \pt \us = 0 $ on $ \Gst $), will destroy the property of incompressibility.
\end{remark}

\begin{remark}
	In this work, the fluid part is supposed to be surrounded by the solid part. In fact, if the solid is immersed in fluid domain, there will be no essential difference in our framework of analysis. Specifically, the outer boundary will still be a Neumann-type boundary, which is a ``do-nothing'' outer boundary for fluid. 
\end{remark}

The initial values for velocities are prescribed as
\begin{equation}
	\label{initial}
	\vf(x,0) = \vf^0, \quad \vs(x,0) = \vs^0, \quad \text{in} \  \OM^0.
\end{equation}

\subsection{Biochemical processes}
The formation of plaque is usually due to biochemical processes in the blood flow and vessel wall. Following the descriptions in \cite{Yang2016,Yang2017}, we introduce the dynamics of monocytes in the blood flow and dynamics of macrophages and foam cells in the vessel wall, of which the concentrations are denoted by $ c_f $, $ c_s $, $ c_s^* $, respectively. Convection and diffusion happen during these biochemical processes, so the motion of monocytes in the blood is given by the transport-diffusion equation
\begin{equation}\label{CF}
	\pt c_f + \Div \left( c_f \vf \right) - D_f \Delta c_f = 0, \quad \text{in}\ \Oft,
\end{equation}
where $ D_f > 0 $ is the diffusion coefficient in the blood, which is assumed to be a constant since the fluid is incompressible and homogeneous. Analogously, the motion of macrophages in the vessel wall is described by
\begin{equation}\label{CS}
	\pt c_s + \Div \left( c_s \vs \right) - D_s \Delta c_s = - f_s^r, \quad \text{in}\ \Ost,
\end{equation}
where $ D_s > 0 $ is the diffusion coefficient in the vessel wall. To simplify the model, we assume that the solid is a homogeneous material, and thus $ D_s $ is a constant. We mention that vessel wall could be inhomogeneous, which represents different diffusion rate in healthy and diseased vessel, see e.g. \cite{Yang2016}. $ f_s^r $ is the reaction function, modeling the rate of transformation from macrophages into foam cells. Furthermore, since foam cells do not diffuse inside the solid material, they are accumulated only due to the convection with $ \vs $ and the transformation from macrophages, which results in the equation of foam cells
\begin{equation}\label{CSS}
	\pt c_s^* + \Div \left( c_s^* \vs \right) = f_s^r, \quad \text{in}\ \Ost.
\end{equation}
The reaction term $ f_s^r $ is supposed to depend on the concentration of macrophages $ c_s $ linearly, namely,
\begin{equation}\label{fsr}
	f_s^r = \beta c_s, \quad \text{in}\ \Ost,
\end{equation}
where $ \beta > 0 $ is assumed to be a constant. In reality, it is more complicated and may depend on the concentration of other chemical species. We just assume a linear relation for the sake of analysis. Then, we give another linear dependence of $ f_s^g $, which is
\begin{equation}\label{fsg}
	f_s^g = \gamma f_s^r = \gamma \beta c_s, \quad \text{in}\ \Ost,
\end{equation}
with a positive constant $ \gamma $. \eqref{fsg} indicates the plaque growth as mentioned in \eqref{Solid}, resulting from the accumulation of foam cells.

To close the system, we still need to model the penetration of monocytes from the blood into the vessel wall, for which the transmission conditions are given on $ \Gt $ by the jump condition
\begin{align}
	\jump{D \nabla c} \cdot \ngt & = 0, \quad \text{on} \  \Gt, 
	\label{TC1} \\
	\zeta \jump{c} - D_s \nabla c_s \cdot \ngt & = 0, \quad \text{on} \  \Gt,
	\label{TC2}
\end{align}
where $ \zeta $ denotes the permeability of the interface $ \Gt $ with respect to the monocytes, which should depend on the hemodynamical stress $ \sigf \cdot \ngt $, however, is supposed to be a constant for simplicity. Moreover, the condition on the outer boundary $ \Gst $ is given by
\begin{equation}\label{boundsc}
	D_s \nabla c_s \cdot \ngst = 0, \quad \text{on} \  \Gst.
\end{equation}
At initial state, the vessel is supposed to be healthy. Hence, there is no foam cell in vessel, i.e.,
\begin{equation}
	c_s^*(x,0) = 0, \quad \text{in} \  \Ost.
\end{equation}
The initial values for concentrations of monocytes and macrophages are given by
\begin{equation}
	\label{initialc}
	c_f(x,0) = c_f^0, \quad c_s(x,0) = c_s^0, \quad \text{in} \  \OM^0.
\end{equation}

\subsection{Description of growth} \label{growth}
Now, we give the description of growth. Normally, prescribing the rate of growth function $ f_s^g $ is not enough to capture the full effect of the tissue growth. Specifically, the real deformation and corresponding deformation gradient $ \hFs $ are induced by both growth and mechanics. Hence, the deformation gradient $ \hFs $ is not enough to capture all responses in the system, for example the deformation gradient in $ \sigse $. Thus, simply transforming the system to Lagrangian coordinates, such as in \cite{CS2006}, we can not solve the whole fluid-structure interaction problem with growth.

As in \cite{Yang2016}, Yang et al. took the idea of deformation gradient decomposition based on the theory of multiple natural configurations. 
In this formulation, one needs a new configuration, which is usually called \textit{natural configuration}, so that one can decompose the whole process into a pure growth and a pure elastic one, see Figure \ref{Growth}.
For more details, readers are referred to \cite{AM2002,JC2012,RHM1994,Yang2016,Yang2017}. 
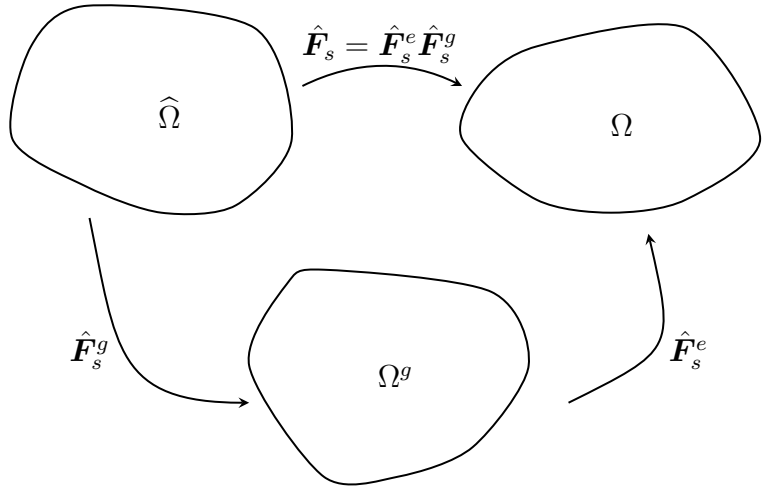
\begin{figure}[htbp]
	\centering
	\begin{tikzpicture}[line width=0.75pt,scale = 0.7]
	\draw [shift ={(-8.45,0)}, scale = 0.7] plot[smooth cycle] coordinates{(0,0) (0.5,2.5) (2,3.6) (6.5,3) (7.5,0) (6,-1.8) (4,-2) (2,-1.3)};
	\draw node at (-5.5,0.5) {$\hO$}; 
	\draw [-stealth] (-3,1) .. controls (-2,1.5) and (-1,1.5).. (0,1);
	\draw node at (-1.55,1.8) {$\hFs = \hFse \hFsg$};
	\draw [scale = 0.7] plot[smooth cycle] coordinates{(0,0) (0.5,1.5) (2,2.5) (6,3) (8,0) (6,-1.65) (4,-2) (2,-1.65)};
	\draw node at (3,0.25) {$\OM$}; 
	\draw [shift ={(-4,-5)}, scale = 0.7] plot[smooth cycle] 
	coordinates{(0,1) (1,3.2) (2,3.6) (6.5,3) (7.5,1) (6,-1.2) (4,-2) (2,-2)};
	\draw node at (-1.25,-4.5) {$\OM^g$}; 
	\draw [-stealth] (-7,-1.5) .. controls (-6.5,-4) and (-6.5,-5).. (-4,-5);
	\draw node at (-7,-4) {$\hFsg$};
	\draw [-stealth] (2,-5) .. controls (4,-4) and (4,-4).. (3.5,-1.8);
	\draw node at (4.25,-4) {$\hFse$};
	\end{tikzpicture}
	\caption{Decomposition of deformation gradient.}
	\label{Growth}
\end{figure}
In this article, we assume the decomposition of the deformation gradient $ \hFs $ as
\begin{equation*}
	\hFs = \hFse \hFsg, \quad \text{in}\ \Os,
\end{equation*}
where $ \hFsg $ is the so-called growth tensor and $ \hFse $ represents the purely elastic tensor. The associated determinants are 
\begin{equation*}
	\hJsg = \det \hFsg, \quad \hJse = \det \hFse, \quad \text{in}\ \Os,
\end{equation*}
respectively. Then
\begin{equation*}
	\hJs = \hJsg \hJse.
\end{equation*}

Growth may happen in different ways. In applications, two assumptions were most commonly applied: \textit{constant-density}, which stands for adding new material with the same density; \textit{constant-volume}, by which the total mass is added and density varies. Since constant-density growth is usually coupled with the assumption of an incompressible tissue, see e.g. \cite{JC2012,RHM1994}, we take this kind of growth into consideration in this work. Then the second equation of \eqref{Solid} reads as
\begin{equation*}
	\rs \Div \vs = f_s^g.
\end{equation*}
Moreover, we assume that plaques grows isotropically:
\begin{equation*}
	\hFsg = \hg \bbI, \quad \text{in}\ \Os,
\end{equation*}
where $ \hg = \hg(X,t) $ is the metric of growth, a scalar function depending on the concentration of macrophages. Hence, 
\begin{equation*}
	\hFse = \frac{1}{\hg} \hFs, \quad
	\hJsg = \hg^n,
\end{equation*}
where $ n $ is the dimension of space. As \cite{AM2002} mentioned, $ \hg $ describes the deformation state of the material, either growing or resorbing, as
\begin{align*}
	0 < \hg < 1 & \Rightarrow \text{resorption}, \\
	\hg > 1 & \Rightarrow \text{growth}.
\end{align*}
From \cite{AM2002,JC2012,Yang2016}, under the assumption of constant-density growth, we deduce that
\begin{equation}
	\label{gLagragian}
	\pt \hg = \frac{\gamma \beta}{n \hrs} \hcs \hg, \quad \text{in}\ \Os.
\end{equation}
This equation shows the specific dependence on $ \hcs $ of $ \hg $. At initial state, $ \hg \bbI $ is supposed to be the identity, i.e., 
\begin{equation*}
	\hg(X,0) = 1, \quad \text{in}\ \Os,
\end{equation*}
without growth or resorption of the material.

\subsection{Literature}
During last decades, fluid-structure interaction problems attracted much attention from mathematicians due to its strong applications in various areas, e.g., biomechanics, blood flow dynamics, aeroelasticity and hydroelasticity. Studies can be divided into two types depending on the dimensions of the fluid and the solid. They are for example 3d-3d coupled and 3d-2d coupled systems, where the solid is contained in the fluid and one part of fluid's boundary respectively.

In the case of 3d-3d model, which is exactly our consideration, let us recall some existence results of strong solutions. Well-posedness of such model was firstly established by Coutand and Shkoller \cite{CS2005}, where they investigated the interaction problem between the Navier-Stokes equation and a linear Kirchhoff elastic material. The results were extended to the quasilinear elastodynamics case by them, where they regularized the hyperbolic elastic equation by a particular parabolic artificial viscosity and then obtained the existence of strong solutions together with the a priori estimates in \cite{CS2006}. Thereafter, Ignatova, Kukavica, Lasiecka and Tuffaha \cite{IKLT2014,IKLT2017} investigated the coupled system of incompressible Navier-Stokes equation and a wave equation from different aspects. More specifically, In \cite{IKLT2014}, static damping and velocity internal damping were added in the wave equation and boundary friction was considered, by which exponential decay was obtained. Later, the boundary friction was removed in \cite{IKLT2017} by introducing the tangential and time-tangential energy estimates. The coupling of the Navier–Stokes equations and the Lam\'{e} system was analyzed by 
Kukavica and Tuffaha \cite{KT2012} with initial regularity $ (v_0, \xi_1) \in H^3(\Of) \times H^2(\Os) $, while Raymond and Vanninathan \cite{RV2014} further proved the existence and uniqueness of local strong solutions with a weaker initial regularity $ (v_0, \xi_1) \in H^{3/2 + \varepsilon}(\Of) \times H^{1 + \varepsilon}(\Os) $, $ \varepsilon > 0 $ arbitrarily small, with periodic boundary conditions. Lately, Boulakia, Guerrero and Takahashi \cite{BGT2019} showed a similar result for the Navier–Stokes-Lam\'{e} system in a smooth domain with reduced demand of the initial regularity.

There are also other variants of free-boundary fluid-structure interactions models. For compressible fluid coupled with elastic bodies, we refer to \cite{BG2010}, where Boulakia and Guerrero addressed the local in time existence and the uniqueness of regular solutions with the initial data $ (\rho_0, u_0, w_0, w_1) \in H^3(\Of) \times H^4(\Of) \times H^3(\Os) \times H^2(\Os) $. This results was later improved by Kukavica and Tuffaha \cite{KT2012compressible} with a weaker initial regularity $ (\rho_0, u_0, w_1) \in H^3(\Of) \times H^{3/2 + r}(\Of) \times H^{3/2 + r}(\Os) $, $ r > 0 $.
More recently, Shen, Wang and Yang \cite{SWY2021} consider the magnetohydrodynamics (MHD)-structure interaction system, where the fluid is described by the incompressible viscous non-resistive MHD equation and the structure is modeled by the wave equation with superconductor material. They solved the existence of local strong solutions with penalization and regularization techniques.

For the 3d-2d/2d-1d systems where where the structure is seen as one part of the fluid's boundary, we just mention several works on the existence and uniqueness of strong solutions to be concise. The mostly investigated case is the fluid-beam/plate systems where the beam/plate equation was imposed with different mechanical mechanism (rigidity, stretching, friction, rotation, etc.), readers are refer to \cite{Beirao da Veiga2004,DS2020,GH2016,GHL2019,Lequeurre2011,Lequeurre2013,MT2021NARWA,Mitra2020} and references therein. Moreover, the fluid-structure interaction problems with nonlinear shells were studied in \cite{CCS2007,CS2010,MRR2020}. It has to be mentioned that in the recent works \cite{DS2020,MT2021NARWA}, a maximal regularity framework, which requires lower initial regularity and less compatibility conditions compared to the energy method, was employed.

\subsection{Mathematical strategy and features}
The new difficulties arise from the plaque formation in the blood vessels, along with the interaction between the fluid and the solid separated by a free interface, the reaction and the diffusion of different cells and the growth of the vessel wall. Numerical computations were carried out in recent years \cite{FRW2016,Yang2016,Yang2017} to simulate the plaque formation and test the effects of different parameters.
To our best knowledge, this is the first work concerning the existence of the strong solutions to the fluid-structure interaction problems with growth. 
Unlike most of the literature above, which are associated with Hilbert spaces ($ L^2 $-setting) and energy methods, we establish our local strong solutions under the framework of maximal $ L^q $-regularity for more general dimension. The method is based on the Banach fixed-point theorem, for which we rewrite the free boundary problem established with Eulerian coordinates in Lagrangian reference configuration, linearize the system at the initial configuration, construct a contraction mapping in a fixed ball and show the local existence and uniqueness of strong solutions. Throughout the proof, we point out the following features.
\begin{itemize}
	\item[i)] We adapt the maximal $ L^q $-regularity theory to solve our problem. Hence, there will be no ``regularity loss'' from initial data to the solution spaces and only few compatibility conditions are needed.
	\item[ii)] The growth is considered to be of the constant-density type. Then under the assumption of  isotropy, the growth will be indicated by the metric function $ \hg $. An ordinary differential equation of $ \hg $ provides the regularity of $ \hg $ needed for the solid velocity and the concentration of macrophages.
	\item[iii)] The Kelvin-Voigt viscous stress $ \sigsv $ we introduced brings the parabolicity to the solid equation. For the linearization, we can use a two-phase Stokes type problem for the fluid-structure interaction problem. This makes sure that we can get the solvabilities and regularities of fluid and solid velocities by maximal regularity theory.
	\item[iv)] The transformed two-phase Stokes problem is endowed with a stress free (Neumann-type) outer boundary condition due to Remark \ref{stress free boundary}. One of our aims is to obtain the solvability of such system. To this end, reduction and truncation arguments are applied. More specifically, we firstly reduce the inhomogeneous linear system to a source and initial value homogeneous problem (except the boundary terms), in order to obtain the pressure regularities. Then by choosing a cutoff function (see \eqref{cutoff function}) which is supported in a subset $ U \subseteq \hO $ and imposing an artificial vanishing Dirichlet boundary on $ \Gs = \partial \hO $, one obtains the solvability of the linear system since the two-phase Stokes problem with Dirichlet boundary is solved in Appendix \ref{twophaseD}. 
\end{itemize}

\subsection{Outline of the paper}
In Section \ref{section 2} we briefly introduce some notations and function spaces along with several preliminary results. Transformation from the deformed configuration to the reference one is shown in the last subsection, as well as the main theorem for the transformed system. 
Section \ref{section 3} is devoted to the analysis of the underlying linear problems, where three separated parts of analysis are proceeded. The main results of this section are the maximal $ L^q $-regularities for these linear problems. The first one is the two-phase Stokes problems with Neumann boundary condition, to which reduction and truncation (localization) arguments are applied. The second problem consists of two reaction-diffusion systems with Neumann boundary condition due to the decoupling of the transmission problem, while the last one is an ordinary differential equations for growth and foam cells. 
In Section \ref{section 4}, we firstly give some estimates related to the deformation gradient, which are of much importance when proving that the constructed nonlinear terms are well-defined and endowed with the property of contraction in the next subsection. Then the full nonlinear system is shown to be well-posed locally in time via Banach fixed-point theorem. Moreover, the cell concentrations are showed to be always nonnegative, provided that the initial data is nonnegative.
Additionally, we introduce some maximal $ L^q $-regularity results of several linear systems in Appendix \ref{results: linear} and establish a uniform extension of the Sobolev-Slobodeckij spaces in Appendix \ref{appendix:extension}.

\section{General settings and main results}
\label{section 2}
\subsection{Mathematical notations}
For matrices $ A, B \in \bbr^{n \times n}  $, let $ A : B = \tr(\tran{B} A) $ and corresponding induced modulus of $ A $ as $ \abs{A} = \sqrt{A : A} $. The set of invertible matrices in $ \bbr^{n \times n}  $ is $ GL(n,\bbr) $. For a differentiable $ A : \bbr_+ \rightarrow GL(n,\bbr) $, we have two useful formulas as
\begin{align}
	\label{DtdA}
	& \frac{\d }{\d t} \det A  = \tr\left( \inv{A} \frac{\d }{\d t} A \right) \det A \\
	\label{DtA}
	& \frac{\d }{\d t} \inv{A}  = - \inv{A} \left(\frac{\d }{\d t} A\right) \inv{A},
\end{align}
which can be found in \cite{EGK2017,Gurtin2010}.
Furthermore, for a vector function $ \bu $ and a tensor matrix $ \bT $, we give an identity which will be used later (see e.g. \cite[(3.20)]{Gurtin2010}):
\begin{equation}
	\label{divTu}
	\Div \left( \tran{\bT} \bu \right) = \bT : \nabla \bu + \bu \cdot \Div \bT.
\end{equation}

For metric spaces $ X $, $ B_X(0,r) $ represents the open ball with radius $ r > 0 $ around $ x \in X $. For normed spaces $ X, Y $ over $ \mathbb{K} = \bbr $ or $ \mathbb{C} $, the set of bounded, linear operators $ T : X \rightarrow Y $ is denoted by $ \cL(X,Y) $ and in particular, $ \cL(X) = \cL(X,X) $.

As usual, the letter $ C $ in the paper represents generic positive constant which may change its value from line to line or even in the same line, unless we give a special declaration.
 \subsection{Function spaces}
If $ M \subseteq \bbr^d $, $ d \in \bbn_+ $ is measurable, $ \Lq{q}(M) $, $ 1 \leq q \leq \infty $ denotes the usual Lebesgue space and $ \norm{\cdot}_{\Lq{q}(M)} $ its norm, as well as the mean value zero Lebesgue space
\begin{equation*}
	\Lqa{q}(M) := \left\{ f \in \Lq{q}(M) : \int_M f \d \mu = 0 \right\},
\end{equation*}
with $ \abs{M} < \infty $. Moreover, $ \Lq{q}(M; X) $ denotes its vector-valued variant of strongly measurable $ q $-integrable functions/essentially bounded functions, where $ X $ is a Banach space. If $ M = (a, b) $, we write for simplicity $ \Lq{q}(a, b) $ and $ \Lq{q}(a, b; X) $. By simple computation, we have 
\begin{equation}
	\label{LI}
	\norm{f}_{\Lq{q}(a,b)} \leq \abs{a - b}^{\frac{1}{q}} \norm{f}_{\Lq{\infty}(a,b)}.
\end{equation}

Let $ \OM \subseteq \bbr^n $ be a open and nonempty domain, $ \W{m}(\OM) $ denotes the usual Sobolev space with $ m \in \bbn $ and $ \Lq{q}(\OM) = \W{0}(\OM) $. Moreover, we set 
\begin{gather*}
	W^{m}_{q,0}(\OM) = \overline{C_0^\infty (\OM)}^{\W{m}(\OM)}, \quad
	\W{-m}(\OM) := [ W^m_{q',0}(\OM) ]', \\
	\WA{m}(\OM) = \W{m}(\OM) \cap \Lqa{q}(\OM), \quad 
	\WA{-m}(\OM) := [ W^m_{q',(0)}(\OM) ]',
\end{gather*}
where $ q' $ is the conjugate exponent to $ q $ satisfying $ \frac{1}{q} + \frac{1}{q'} = 1 $. 

For $ k,k' \in \bbn $ with $ k < k' $, we consider the standard definition of the Besov spaces by real interpolation of Sobolev spaces (see Lunardi \cite{Lunardi2018})
\begin{equation*}
	\Bqp{s}(\OM) = \left( \W{k}(\OM), \W{k'}(\OM) \right)_{\theta, p},
\end{equation*}
where $ s = (1 - \theta)k + \theta k',\ \theta \in (0, 1) $. In the special case $ q = p $, we also have Sobolev-Slobodeckij spaces
\begin{equation*}
	\W{s}(\OM) = \Bq{s}(\OM) = \left( \W{k}(\OM), \W{k'}(\OM) \right)_{\theta, q},
\end{equation*}
which is endowed with norm $ \norm{\cdot}_{\W{s}(\OM)} = \norm{\cdot}_{\Lq{q}(\OM)} + \seminorm{\cdot}_{\W{s}(\OM)} $, where
\begin{equation*}
	\seminorm{f}_{\W{s}(\OM)}^q = \int_{\OM} \int_{\OM} \left( \frac{\abs{f(x) - f(y)}}{\abs{x-y}^s} \right)^q \frac{\d x \d y}{\abs{x - y}^n}.
\end{equation*}
The multiplication property of such space is given in the next lemma.
\begin{lemma}[Multiplication]
	\label{algebra}
	Let $ \OM $ be a bounded Lipschitz domain. For $ f,g \in \W{s}(\OM) $ and $ sq > n $ with $ s > 0 $, we have the multiplication property, which is 
	\begin{equation*}
		\norm{fg}_{\W{s}(\OM)} \leq M_q \norm{f}_{\W{s}(\OM)} \norm{g}_{\W{s}(\OM)},
	\end{equation*}
	where $ M_q $ is a constant depending on $ q $.
\end{lemma}
\begin{proof}
	For the case $ s \in \bbn_+ $, we refer to \cite[Theorem 1]{Valent1985}.
	For the other cases, since $ \W{s} = B_{q,q}^{s} $ for every $ s \in \bbr_+ \backslash \bbn$, then \cite[Theorem 6.6]{Johnsen1995} implies this.
\end{proof}

Next, for an interval $ I \subset \bbr $ and a Banach space $ X $, we recall the definition of vector-valued Sobolev-Slobodeckij space as 
\begin{equation*}
	\W{s}(I; X): = \left\{ f \in \Lq{q}(I; X): \norm{f}_{\W{s}(I; X)} < \infty \right\},
\end{equation*}
whose the norm is $ \norm{\cdot}_{\W{s}(I; X)} = \norm{\cdot}_{\Lq{q}(I; X)} + \seminorm{\cdot}_{\W{s}(I; X)} $ with
\begin{equation*}
	\seminorm{f}_{\W{s}(I; X)}^q = \int_{I} \int_{I} \left( \frac{\norm{f(t) - f(\tau)}_X}{\abs{t-\tau}^s} \right)^q \frac{\d t \d \tau}{\abs{t - \tau}}.
\end{equation*}
Then we define $ \WO{s}(0,T; X) $ with $ 0 < T \leq \infty $ to be a vector-valued space having a vanishing trace at $ t = 0 $, i.e., 
\begin{equation*}
	\WO{s}(0,T; X) := \left\{ u \in \W{s}(0,T; X) : \rv{u}_{t = 0} = 0 \right\}.
\end{equation*}
In addition, we introduce one embedding result from Simon \cite[Corollary 17]{Simon1990}.
\begin{lemma}
	\label{embedding: Ws Wr}
	Suppose $ 0 < r \leq s < 1 $ and $ 1 \leq p \leq \infty $. Then
	\begin{equation*}
		\W{s}(I; X) \hookrightarrow \W{r}(I; X)
	\end{equation*} 
	and, for all $ f \in \W{s}(I; X) $,
	\begin{equation*}
		\seminorm{f}_{\W{r}(I; X)} \leq 
		\left\{
			\begin{aligned}
				& \abs{I}^{s - r} \seminorm{f}_{\W{s}(I; X)} && \text{for bounded}\ I, \\
				& \seminorm{f}_{\W{s}(I; X)} + \frac{4}{r} \norm{f}_{\Lq{q}(I; X)} && \text{for all}\ I.
			\end{aligned}
		\right.
	\end{equation*}
\end{lemma}

For $ r, s \geq 0 $, the anisotropic Sobolev-Slobodeckij spaces $ \W{r,s} $ is defined as
\begin{equation}
	\label{WSR}
	\W{r,s}(\OM \times I) := \Lq{q}\left( I; \W{r}(\OM) \right) \cap \W{s}\left( I; \Lq{q}(\OM) \right).
\end{equation}
Based on the trace method interpolation at time zero \cite[Section 3.4.6]{PS2016} and \cite[Chapter III, Theorem 4.10.2]{Amann1995}, we give some useful embeddings, which will be employed later.
\begin{lemma}\label{embedding}
	Let $ X_1, X_0 $ be two Banach spaces and $ X_1 \hookrightarrow X_0 $. Define $ X_T = \Lq{q}(0,T; X_1) \cap \W{1}(0,T; X_0) $ for all $ 1 < q < \infty $ and $  0 < T < \infty $, then
	\begin{equation*}
		X_T \hookrightarrow C \left( [0,T]; X_\gamma \right),
	\end{equation*}
	where 
	\begin{equation*}
		X_\gamma = (X_0, X_1)_{1- \frac{1}{q}, q} = \left\{ \rv{u}_{t = 0}: u \in X_T \right\}
	\end{equation*} 
	is the trace space. Moreover, if $ X_T $ is endowed with the norm
	\begin{equation*}
		\norm{u}_{X_T} := \norm{u}_{\Lq{q}(0,T; X_1)} + \norm{u}_{\W{1}([0,T]; X_0)} + \norm{\rv{u}_{t = 0}}_{X_\gamma},
	\end{equation*}
	then there is some $ C > 0 $ independent of $ T $ such that for $ T \in [0, \infty) $ and $ u \in X_T $,
	\begin{equation*}
		\norm{u}_{C \left( 0,T; X_\gamma \right)} 
		\leq C \norm{u}_{X_T}.
	\end{equation*}
	
	In particular, if $ \OM \subset \bbr^n $, $ n \geq 2 $, is a bounded domain, $ n < q < \infty $, and if $ X_1 = \W{2}(\OM) $, $ X_0 = \Lq{q}(\OM) $, then $ X_\gamma = \W{2 - \frac{2}{q}}(\OM) $ and 
	\begin{equation}
		\label{embedding: W21 W1}
		\W{2,1}(\OM \times (0,T)) 
		\hookrightarrow C ( [0,T]; \W{2 - \frac{2}{q}}(\OM) ) 
		\hookrightarrow C ( [0,T]; \W{1}(\OM) ),
	\end{equation}
	together with 
	\begin{equation*}
		\norm{u}_{C ( [0,T]; \W{1}(\OM) )} 
		\leq C ( \norm{u}_{\W{2,1}(\OM \times (0,T))} + \norm{u_0}_{\W{2 - \frac{2}{q}}} ),
	\end{equation*}
	\begin{equation*}
		\norm{u - v}_{C ( [0,T]; \W{1}(\OM) )} 
		\leq C \norm{u - v}_{\W{2,1}(\OM \times (0,T))},
	\end{equation*}
	for $ u,v \in \W{2,1}(\OM \times (0,T)) $ with $ \rv{u}_{t = 0} = \rv{v}_{t = 0} = u_0 $.
\end{lemma}

\begin{lemma}\label{embedding: W alpha}
	Let $ \Sigma $ be a compact sufficiently smooth hypersurface. For $ 1 < q < \infty $, $ \frac{1}{q} < \alpha \leq 1 $ and $  0 < T < \infty $, define $ X_T := \Lq{q}(0,T; \W{2 \alpha}(\Sigma)) \cap \W{\alpha}(0,T; \Lq{q}(\Sigma)) $, then
	\begin{equation*}
		X_T \hookrightarrow C \left( [0,T]; X_\gamma \right),
	\end{equation*}
	where 
	\begin{equation*}
		X_\gamma = \left\{ \rv{u}_{t = 0}: u \in X_T \right\} = \W{2 \alpha - \frac{2}{q}}(\Sigma).
	\end{equation*} 
	Moreover, if $ X_T $ is endowed with the norm
	\begin{equation*}
		\norm{u}_{X_T} := \norm{u}_{\Lq{q}(0,T; X_1)} + \norm{u}_{\W{\alpha}(0,T; X_0)} + \norm{\rv{u}_{t = 0}}_{X_\gamma},
	\end{equation*}
	then there is some $ C > 0 $ independent of $ T $ such that for all $ u \in X_T $,
	\begin{equation*}
		\norm{u}_{C \left( [0,T]; X_\gamma \right)} 
		\leq C \norm{u}_{X_T}.
	\end{equation*}
\end{lemma}

\subsection{An equivalent system in Lagrangian reference configuration} \label{equivalent system}
In this section, we transform the free-boundary fluid-structure problem with growth from deformed configuration (Eulerian) to a fixed reference configuration (Lagrangian) and state the main result. For quantities in different configurations, we define
\begin{equation}
	\label{equalofv}
	\begin{aligned}
	& \hv(X,t) = \bv(x,t), \ \hpi(X,t) = \pi(x,t), \ \hsigma(X,t) = \bsigma(x,t), \\
	& \hr(X,t) = \rho(x,t), \ \hmu(X,t) = \mu(x,t), \ \hnu(X,t) = \nu(x,t),
	\end{aligned}
\end{equation}
for all $ x = \vp(X,t), \ X \in \hO $ and $ t \geq 0 $. Then one can easily deduce the derivatives between quantities in different configurations as
\begin{align}
	\label{ptu}
	& \pt \hat{\bm{u}}(X,t) = \left( \pt + \bv(x,t) \cdot \nabla \right) \bm{u}(x,t), \\
	\label{grad}
	& \nabla \phi = \invtr{\hF} \hnab \hat{\phi}, \quad \nabla \bm{u} = \inv{\hF} \hnab \hat{\bm{u}}, \\
	\label{div}
	& \Div \bm{u} = \tr ( \nabla \bu ) = \tr ( \inv{\hF} \hnab \hat{\bm{u}} ) = \invtr{\hF} : \hnab \hu, 	
\end{align}
where $ \phi / \hat{\phi} $ is any scalar function in $ \OM / \hO $ and $ \bm{u} / \hat{\bm{u}} $ is any vector-valued function in $ \OM / \hO $.
From \cite{Ciarlet1988},  we know that the Piola transform establishes a correspondence between tensor field defined in deformed and reference configurations, which is
\begin{equation}
	\label{Piola}
	\hT(X,t) = \hJ(X,t) \bsigma(x,t) \invtr{\hF}(X,t), \quad \text{for all}\ x = \vp(X,t), \ X \in \hO,
\end{equation}
where $ \hT $ is the first Piola–Kirchhoff stress tensor. Moreover, the following property of the Piola transformation will be useful:
\begin{lemma}[{\cite[Theorem 1.7-1]{Ciarlet1988}}]
	\label{Piolaproperty}
	For a stress tensor $ \bsigma(x,t) $ in the deformed configuration $ \OM $, and the corresponding first Piola–Kirchhoff stress tensor $ \hT(X,t) $ in reference configuration $ \hO $, we have:
	\begin{align*}
		& \hdiv \hT(X,t) = \hJ(X,t) \Div \bsigma(x,t), \quad \text{for all}\ x = \vp(X,t), \ X \in \hO, \\
		& \hT(X,t) \hn \d \hat{a} = \bsigma(x,t) \bn \d a, \quad \text{for all}\ x = \vp(X,t), \ X \in \hO.
	\end{align*}
\end{lemma}

For the fluid part, it follows from \eqref{DtdA} that
\begin{align*}
	\pt \hJf = \tr \left( \inv{\hF} \pt \hF \right) \hJf 
	= \tr \left( \inv{\hF} \hnab \hv \right) \hJf 
	= \Div \bv \hJf 
	= 0,
\end{align*}
which implies
\begin{equation}
	\label{Jf=1}
	\hJf = \rv{\hJf}_{t = 0} = \det \bbI = 1, \quad \text{in}\ \Of.
\end{equation}
For the solid part, since the deformation from natural configuration $ \Osg $ to the deformed configuration $ \Ost $ conserves mass, incompressibility yields $ \hJse = 1 $ and hence,
\begin{equation*}
	\hJs = \hJsg = \hg^n, \quad \text{in}\ \Os.
\end{equation*}

Now combining formulas \eqref{gLagragian}, \eqref{equalofv}--\eqref{Jf=1} and Lemma \ref{Piolaproperty}, we rewrite the fluid–structure interaction problem \FSI\  in the reference configuration $ \hO $.
\begin{align}
	\label{Fluid0}
	\left.
		\begin{aligned}
			\hrf \pt \hvf - \hdiv \left( \hsigf \invtr{\hFf} \right) & = 0 \\
			\invtr{\hFf} : \hnab \hvf & = 0 \\
			\pt \hcf - \hDf \hdiv \left( \inv{\hFf} \invtr{\hFf} \hnab \hcf \right) & = 0
		\end{aligned}
	\right\} & \quad \text{in}\  \Of \times (0, T), \\
	\label{Solid0}
	\left.
		\begin{aligned}
			\hrs \pt \hvs - \inv{\hJs} \hdiv \left( \hJs \hsigs \invtr{\hFs} \right) & = 0 \\
			\invtr{\hFs} : \hnab \hvs - \frac{\gamma \beta}{\hrs} \hcs & = 0 \\
			\pt \hcs - \hDs \inv{\hJs} \hdiv \left( \hJs \inv{\hFs} \invtr{\hFs} \hnab \hcs \right) + \beta \hcs \left( 1 + \frac{\gamma}{\hrs} \hcs \right) & = 0 \\
			\pt \hcss - \beta \hcs + \frac{\gamma \beta}{\hrs} \hcs \hcss = 0, \quad
			\pt \hg - \frac{\gamma \beta}{n \hrs} \hcs \hg & = 0 
		\end{aligned}
	\right\} & \quad \text{in}\  \Os \times (0, T), \\
	\label{Gamma}
	\left.
		\begin{aligned}
			\jump{\hv} = 0, \quad
			\jump{\hsigma \invtr{\hF}} \hng = 0, \quad
			\jump{\hD \inv{\hF} \invtr{\hF} \hnab \hc} \hng & = 0 \\
			\zeta \jump{\hc} - \hDs \inv{\hFs} \invtr{\hFs} \hnab \hcs \cdot \hng & = 0
		\end{aligned}
	\right\} & \quad \text{on}\  \Gamma \times (0, T), \\
	\label{Gammas}
	\hsigs \invtr{\hFs} \hngs = 0, \quad
	\hDs \inv{\hFs} \invtr{\hFs} \hnab \hcs \cdot \hngs = 0 & \quad \text{on}\  \Gs \times (0, T), \\
	\rv{\hv}_{t = 0} = \vo, \quad \rv{\hc}_{t = 0} = \co & \quad \text{in}\  \hO, \\
	\label{initialv}
	\rv{\hcss}_{t = 0} = 0, \quad \rv{\hg}_{t = 0} = 1 & \quad \text{in}\  \Os,
\end{align}
where the corresponding stress tensors are
\begin{align*}
	& \hsigf = - \hpif \bbI + \hnuf \left( \inv{\hFf} \hnab \hvf + \tran{\hnab} \hvf \invtr{\hFf} \right), \quad \hsigs = \hsigse + \hsigsv, \\
	& \hsigse = - \hpis \bbI + \hmus \left( \hFse \tran{\hFse} - \bbI \right) = - \hpis \bbI + \hmus \left( \frac{1}{(\hg)^2} \hFs \tran{\hFs} - \bbI \right), \\
	& \hsigsv = \hnus \left( \hnab \hvs + \tran{\hnab} \hvs \right) \tran{\hFs}.
\end{align*}

For the maximal $ L^q $-regularity setting, we assume 
$$ 
	\vo \in \Bq{1-1/q}(\hO)^n \cap \Bq{2\left(1-1/q\right)}(\tO)^n, \quad
	\co \in \Bq{2\left(1-1/q\right)}(\tO),
$$ 
that is, 
$$ 
	\vo \in \W{1-1/q}(\hO)^n \cap \W{2\left(1-1/q\right)}(\tO)^n =: \Dqv, \quad 
	\co \in \W{2\left(1-1/q\right)}(\tO) =: \Dqc,
$$ 
where we define $ \tO = \Of \cup \Os $. $ \Dq := \Dqv \times \Dqc $ will be the initial space for velocities and concentrations. Moreover, we introduce the compatibility conditions for $ q > n + 2 $, which were also used in e.g. Abels \cite{Abels2005}, Pr\"{u}ss and Simonett \cite{PS2016}, Shibata and Shimizu \cite{SS2008}, Shimizu \cite{Shimizu2008}:
\begin{equation}
	\label{compatibility: v}
	\begin{gathered}
		\hdiv \vo = 0, \quad
		\rv{\jump{\vo}}_\Gamma = 0, \quad
		\rv{\jump{\left( \hnu \left( \hnab \vo + \tran{\hnab} \vo \right) \hng\right)_{\tau}} }_{\Gamma} = 0, \\
		\rv{\left( \hnu \left( \hnab \vo + \tran{\hnab} \vo \right) \hngs \right)_{\tau}}_{\Gs} = 0,
	\end{gathered}
\end{equation}
and
\begin{equation}
	\label{compatibility: c}
	\rv{\left( \zeta \jump{\co} - \hDs \hnab \co_s \cdot \hng \right)}_\Gamma = 0, \quad 
	\rv{\jump{\hD \hnab \co} \cdot \hng}_\Gamma = 0, \quad 
	\rv{\hDs \hnab \co_s \cdot \hngs}_{\Gs} = 0,
\end{equation}
where $ (\cdot)_{\tau} $ denotes the tangential part on the surface, namely, $ (\cdot)_{\tau} = (\bbI - \hn \otimes \hn) \cdot $. 
Besides this, we define the solution space for $ (\hv, \hpi, \hc, \hcss, \hg) $ as $ \YT = \YT^1 \times \YT^2 \times \YT^3 \times \YT^4 \times \YT^4 $, where
\begin{align*}
	& \YT^1 = \Lq{q}\left(0,T; \W{2}(\tO) \cap \W{1}(\hO)\right)^n \cap \W{1}\left(0,T; \Lq{q}(\hO)\right)^n, \\
	& \YT^2 = \left\{ 
		\begin{aligned}
			\hpi \in \Lq{q}\left(0, T; \W{1}(\hO) \right) : & \jump{\hpi} \in \W{1-\frac{1}{q},\onehalf(1-\frac{1}{q})}(\Gamma \times (0,T)) \\
			& \rv{\hpi}_{\Gs} \in \W{1-\frac{1}{q},\onehalf(1-\frac{1}{q})}(\Gs \times (0,T))
		\end{aligned} 
	\right\}, \\
	& \YT^3 = \Lq{q}\left(0,T; \W{2}(\tO) \right) \cap \W{1}\left(0,T; \Lq{q}(\hO)\right), \\
	& \YT^4 = \W{1}\left(0,T; \W{1}(\Os)\right),
\end{align*}
equipped with norms
\begin{align*}
	\norm{\hv}_{\YT^1} & = \norm{\hv}_{\Lq{q}\left(0,T; \W{2}(\tO) \cap \WOO{1}(\hO)\right)^n} + \norm{\hv}_{\W{1}\left(0,T; \Lq{q}(\hO)\right)^n}, \\
	\norm{\hpi}_{\YT^2} & = \norm{\hpi}_{\Lq{q}\left(0, T; \W{1}(\hO) \right)} + \norm{\jump{\hpi}}_{\W{1-\frac{1}{q},\onehalf(1-\frac{1}{q})}(\Gamma \times (0,T))} \\
	& \qquad \qquad \qquad \qquad \qquad + \norm{\rv{\hpi}_{\Gs}}_{\W{1-\frac{1}{q},\onehalf(1-\frac{1}{q})}(\Gs \times (0,T))}, \\
	\norm{\hc}_{\YT^3} & = \norm{\hc}_{\Lq{q}\left(0,T; \W{2}(\tO) \right)} + \norm{\hc}_{\W{1}\left(0,T; \Lq{q}(\hO)\right)}, \\
	\norm{\hcss}_{\YT^4} & = \norm{\hcss}_{\W{1}\left(0,T; \W{1}(\Os) \right)}, \quad
	\norm{\hg}_{\YT^4} = \norm{\hg}_{\W{1}\left(0,T; \W{1}(\Os) \right)}.
\end{align*}
Moreover, we set $ \YT^v := \YT^1 \times \YT^2 $.

\begin{remark}
	\label{space of pi}
	These spaces are constructed from the problem and the maximal regularity theory, endowed with the natural norms. In particular, $ \jump{\hpi} $ and $ \rv{\hpi}_{\Gs} $ are determined by the regularities of the Neumann trace of $ \hv $ on $ \Gamma $ and $ \Gs $ respectively. Hence, we add the norm of $ \norm{\jump{\hpi}}_{\W{1-1/q,(1-1/q)/2}(\Gamma \times (0,T))} $ and $ \norm{\rv{\hpi}_{\Gs}}_{\W{1-1/q,(1-1/q)/2}(\Gs \times (0,T))} $ in $ \YT^2 $-norm correspondingly. One can easily verify that all spaces are Banach spaces.
\end{remark}

Now the main theorem is given as follows.
\begin{theorem}[{Main theorem}]
	\label{main}
	Let $ q > n + 2 $. Assume that $ \Gt $ is a hypersurface of class $ C^3 $, $ (\vo, \co) \in \Dq $ such that the compatibility conditions \eqref{compatibility: v} and \eqref{compatibility: c} hold, then there is a positive $ T_0 = T_0(\norm{(\vo, \co)}_{\Dq}) < \infty $ such that there exists a unique strong solution $ (\hv, \hpi, \hc, \hcss, \hg) \in Y_{T_0} $ to system \NFSI. Moreover, $ \hc \geq 0 $ and $ \hcss, \hg > 0 $, if $ \co \geq 0 $.
\end{theorem}

\begin{remark}
	In this work, the boundary of domain is supposed to be $ C^3 $. We remark here that if the boundary is not smooth enough, for example, $ C^{0,1} $ Lipschitz domain, it will encounter the contact line problems with a contact angle. As far as we know, it is still an open problem. The authors considered the similar model with a ninety degree contact angle in \cite{AL2021} recently.
\end{remark}

The proof of Theorem \ref{main} relies on the Banach fixed-point theorem. To this end, we need to linearize the nonlinear system \NFSI. Since we consider a nonzero initial reference configuration, a standard perturbation method is applied to \NFSI, for which we rearrange the system at the initial deformation and move all perturbed terms to the right-hand side, namely,
\begin{align}
	\label{Fluid1}
	\left.
		\begin{aligned}
			\hrf \pt \hvf - \hdiv \bS( \hvf, \hpif ) & = \bK_f \\
			\hdiv \hvf & = G_f
		\end{aligned}
	\right\} \quad & \text{in}\ \Of \times (0, T), \\
	\label{Fluid1s}
	\left.
		\begin{aligned}
			\hrs \pt \hvs - \hdiv \bS( \hvs, \hpis ) & = \bar{\bK}_s + \bK_s^g =: \bK_s \\
			\hdiv \hvs - \frac{\gamma \beta}{\hrs} \hcs & = G_s
		\end{aligned}
	\right\} \quad & \text{in}\ \Os \times (0, T), \\
	\label{Gamma1}
	\jump{\hv} = 0, \quad
	\jump{\bS( \hv, \hpi )} \hng = \bH^1 \quad & \text{on}\ \Gamma \times (0, T), \\
	\label{Gammas1}
	\bS( \hvs, \hpis ) \hngs = \bH^2 \quad & \text{on}\ \Gs \times (0, T), \\
	\label{initialv1}
	\rv{\hv}_{t = 0} = \vo \quad & \text{in}\  \hO,
\end{align}
\begin{align}
	\label{Fluid1cf}
	\pt \hcf - \hDf \hDelta \hcf = F_f^1 \quad & \text{in}\ \Of \times (0, T), \\
	\label{Fluid1cs}
	\pt \hcs - \hDf \hDelta \hcs = \bar{F}_s^1+ F_s^g =: F_s^1 \quad & \text{in}\ \Os \times (0, T), \\
	\label{Gamma1c}
	\left.
		\begin{aligned}
			\hDf \hnab \hcf \cdot \hng & = \hDs \nabla \hcs \cdot \hng + \bar{F}_f^2 =: F_f^2 \\
			\hDs \hnab \hcs \cdot \hng & = \zeta \jump{\hc} + \bar{F}_s^2 =: F_s^2
		\end{aligned}
	\right\} \quad & \text{on}\ \Gamma \times (0, T), \\
	\label{Gammas1c}
	\hDs \hnab \hcs \cdot \hngs = F^3 \quad & \text{on}\ \Gs \times (0, T), \\
	\label{initialv1c}
	\rv{\hc}_{t = 0} = \co \quad & \text{in}\  \tO,
\end{align}
\begin{align}
	\label{Solidcss}
	\pt \hcss - \beta \hcs = F^4 \quad & \text{in}\  \Os \times (0, T), \\
	\label{initialv1cs}
	\rv{\hcss}_{t = 0} = 0 \quad & \text{in}\  \Os,
\end{align}
\begin{align}
	\label{Solidg}
	\pt \hg - \frac{\gamma \beta}{n \hrs} \hcs = F^5 \quad & \text{in}\  \Os \times (0, T), \\
	\label{initialv1g}
	\rv{\hg}_{t = 0} = 1 \quad & \text{in}\  \Os,
\end{align}
where $ \bS( \hv, \hpi ) = - \hpi \bbI + \hnu \left( \hnab \hv + \tran{\hnab}\hv \right) $ in $ \tO $ and
	\begin{align}
		& \bK_f = \hdiv \kf, \quad \bar{\bK}_s = \hdiv \ks, \quad \bK_s^g = - \left( \hsigs \invtr{\hFs} \right) \frac{n \hnab \hg}{\hg}, \nonumber \\
		& G = - \left( \FOinvtran \right) : \hnab \hv, \quad 
		\bH^1 = - \jump{\tk} \cdot \hng, \quad \bH^2 = - \ks \cdot \hngs, \nonumber \\
		& F_f^1 = \hdiv \tFf, \quad \bar{F}_s^1 = \hdiv \tFs, \label{Nonlinear} \\
		& F_s^g = - \beta \hcs \left( 1 + \frac{\gamma}{\hrs} \hcs \right) - \frac{n \hnab \hg}{\hg} \cdot \left( \hDs \inv{\hFs} \invtr{\hFs} \hnab \hcs \right), \nonumber \\
		& \bar{F}_f^2 = - \jump{\tF} \cdot \hng, \quad \bar{F}_s^2 = - \tFs \cdot \hng, \quad F^3 = -\tFs \cdot \hngs, \nonumber \\
		& F^4 = - \frac{\gamma \beta}{\hrs} \hcs \hcss, \quad
		F^5 = - \frac{\gamma \beta}{n \hrs} \hcs \left( \hg - 1 \right), \nonumber 
	\end{align}
with
\begin{align*}
	\kf & = - \hpif \left( \FfOinvtran \right) 
	+ \nuf \left( \inv{\hFf} \hnab \hvf + \tran{\hnab} \hvf \invtr{\hFf} \right) \left( \FfOinvtran \right) \\
	& \qquad + \nuf \left( \left(\FfOinv\right) \hnab \hvf + \tran{\hnab} \hvf  \left(\FfOinvtran\right)\right), \\
	\ks & = - \hpis \left(\FsOinvtran\right) + \mus \left( \frac{1}{\hg^2} \left( \FsO \right) + \left( \frac{1}{\hg^2} - 1 \right) \bbI - \left( \FsOinvtran \right) \right), \\
	\tF & = \hD \left( \inv{\hF} \invtr{\hF} - \bbI \right) \hnab \hc.
\end{align*}
Then we analyze system \NLFSI, which is exactly \NFSI.
\begin{remark}
	It follows from the Piola identity, which can be found in \cite[Page 39]{Ciarlet1988}, that
	\begin{equation*}
		\hdiv \left( \hJ \invtr{\hF} \right) = 0.
	\end{equation*}
	Then from \eqref{divTu},
	\begin{equation*}
		\hJ \invtr{\hF} : \hnab \hv = \hdiv \left( \hJ \inv{\hF} \hv \right).
	\end{equation*}
	Hence, $ G $ possesses the form
	\begin{equation}
		\label{G : form}
		\begin{aligned}
		G_f & = - \hdiv \left( \left( \FfOinv \right) \hvf \right), \quad
		G_s = - \hdiv \left( \left( \FsOinv \right) \hvs \right)
		+ \hvs \cdot \hdiv \invtr{\hFs}.
		\end{aligned}
	\end{equation}
\end{remark}
\begin{remark}
	In generic, the system \eqref{Fluid1cf}--\eqref{initialv1c} for concentrations of monocytes and macrophages can be considered as a transmission problem in $ \Of $ and $ \Os $ with a common boundary $ \Gamma $. However, if we use the concentration and stress jump condition as boundary condition on $ \Gamma $, we will meet the regularity problem due to the high order term $ \Ds \hnab \hcs \cdot \hng $ in \eqref{Gamma1c}$ _2 $. More precisely, in our further perturbation argument, all perturbated or unrelated terms will be removed to right-hand side of the equation and the regularities of both sides should coincide with each other. The point is that in such argument, the right-hand side of \eqref{Gamma1c}$ _2 $ contains $ \Ds \hnab \hcs \cdot \hng $, which leads to a lower regularity, provided the same regularity of $ \hc $ on the both side. 
	
	Therefore, to avoid such awkward situation, we rewrite the transmission conditions as two Neumann type boundary conditions. Then the transmission problem can be decoupled into two separated parabolic system, which are both imposed with Neumann boundary and defined in $ \Of $ and $ \Os $ respectively. This is why we treat the boundary conditions on $ \Gamma $ as the form shown in \eqref{Gamma1c}. 
\end{remark}

Consequently, given data $ (\bK, G, \bH^1, \bH^2, F^1, F^2, F^3, F^4, F^5) $ with suitable regularities, existence and uniqueness of $ (\hv, \hpi, \hc, \hcss, \hg) $ in associated spaces will be obtained by the well-poesdness of linear systems in the next section. 

\section{Analysis of the linear systems}
\label{section 3}
As seen in \NLFSI, the linearized system can be seen as a two-phase type Stokes problem \eqref{Fluid1}--\eqref{initialv1}, two separated reaction-diffusion systems \eqref{Fluid1cf}--\eqref{initialv1c} and two ordinary differential equations \eqref{Solidcss}--\eqref{initialv1g} (equation for foam cells and growth, respectively). In this section, thanks to the maximal $ L^q $-regularity theory, we establish the existence for strong solutions to these systems with prescribed initial data and source terms in appropriate spaces.

\subsection{Two-phase Stokes problems with Neumann boundary condition}

Observing that $ \rv{(\bK, G, \bH^1, \bH^2)}_{t = 0} =0 $, one replaces $ (\bK, G, \bH^1, \bH^2) $ in \eqref{Fluid1}--\eqref{initialv1} by known functions $ (\bk, g, \bh^1, \bh^2) $ with $ \rv{(\bk, g, \bh^1, \bh^2)}_{t = 0} = 0 $ in \eqref{Fluid1s}. Then we get the problem addressed in this subsection.
\begin{equation} \label{twophase: linear}
	\begin{aligned}
		\hr \pt \hv - \hdiv \bS( \hv, \hpi ) & = \bk && \quad \text{in}\  \tO \times (0, T), \\
		\hdiv \hv & = g && \quad \text{in}\  \tO \times (0, T), \\
		\jump{\hv} & = 0 && \quad \text{on}\  \Gamma \times (0, T), \\
		\jump{\bS( \hv, \hpi )} \hng & = \bh^1 && \quad \text{on}\  \Gamma \times (0, T), \\
		\bS( \hvs, \hpis ) \hngs & = \bh^2 && \quad \text{on}\  \Gs \times (0, T), \\
		\rv{\hv}_{t = 0} & = \vo && \quad \text{in}\  \hO.
	\end{aligned}
\end{equation}
Now, we will prove the following theorem, namely, existence of unique solution to a two-phase Stokes problem with outer Neumann boundary condition.
\begin{theorem}
	\label{twophase: theorem}
	Let $ q > n + 2 $, $ T > 0 $, $ \hO $ a bounded domain with $ \Gs \in C^3$, $ \Gamma $ a closed hypersurface of class $ C^3 $. Assume that $ (\bk, g , \bh^1, \bh^2) $ are known functions contained in $ \ZT^v $ with initial value zero and $ \vo \in \Dqv $ with compatibility conditions
	\begin{gather*}
		\hdiv \vo = \rv{g}_{t = 0}, \quad 
		\rv{\jump{\vo}}_{\Gamma} = 0, \quad 
		\rv{\jump{\left( \hnu \left( \hnab \vo + \tran{\hnab} \vo \right) \hng\right)_{\tau}} }_{\Gamma} = 0, \\
		\rv{\left( \hnu \left( \hnab \vo + \tran{\hnab} \vo \right) \hngs \right)_{\tau}}_{\Gs} = 0.
	\end{gather*}
	Then the Stokes problem \eqref{twophase: linear} admits a unique strong solution $ (\hv, \hpi) $ in $ \YT^v $. Moreover, there exist a time $ T_0 > 0 $ and a constant $ C = C(T_0) > 0 $ such that for $ 0 < T \leq T_0 $,
	\begin{equation}
		\label{thm3.1:estimates}
		\norm{(\hv, \hpi)}_{\YT^v} \leq C \norm{(\bk, g , \bh^1, \bh^2, \vo)}_{\ZT^v \times \Dqv},
	\end{equation}
	where $ \ZT^v := \ZT^1 \times \ZT^2 \times \ZT^3 \times \ZT^4 $ with
	\begin{align}
		\label{ZT1}
		& \ZT^1 := \Lq{q}\left( 0,T; \Lq{q}(\tO) \right)^n, \\
		\label{ZT2}
		& \ZT^2 := \left\{
			\begin{aligned}
				& g \in \Lq{q}\left( 0,T; \W{1}(\tO) \right) \cap \W{1}\left( 0,T; \W{-1}(\hO) \right): \\
				& \qquad \qquad \qquad \qquad \tr_{\Gamma}(g) \in \W{1 - \frac{1}{q},\onehalf\left(1 - \frac{1}{q}\right)}(\Gamma \times (0,T)), \\
				& \qquad \qquad \qquad \qquad \tr_{\Gs}(g) \in \W{1 - \frac{1}{q},\onehalf\left(1 - \frac{1}{q}\right)}(\Gs \times (0,T))
			\end{aligned}
		\right\}, \\
		\label{ZT3}
		& \ZT^3 := \W{1 - \frac{1}{q},\onehalf\left(1 - \frac{1}{q}\right)}(\Gamma \times (0,T))^n, \quad \ZT^4 := \W{1 - \frac{1}{q},\onehalf\left(1 - \frac{1}{q}\right)}(\Gs \times (0,T))^n,
	\end{align}
	endowed with norms
	\begin{align*}
		& \norm{\bk}_{\ZT^1} = \norm{\bk}_{\Lq{q}\left( 0,T; \Lq{q}(\tO) \right)^n}, \\
		& \norm{g}_{\ZT^2} = \norm{g}_{\Lq{q}\left( 0,T; \W{1}(\tO) \right)} + \norm{g}_{\W{1}\left( 0,T; \W{-1}(\hO) \right)} \\
		& \qquad \qquad \quad + \norm{\tr_{\Gamma}(g)}_{\W{1 - \frac{1}{q},\onehalf\left(1 - \frac{1}{q}\right)}(\Gamma \times (0,T))}
		+ \norm{\tr_{\Gs}(g)}_{\W{1 - \frac{1}{q},\onehalf\left(1 - \frac{1}{q}\right)}(\Gs \times (0,T))}, \\
		& \norm{\bh^1}_{\ZT^3} = \norm{\bh}_{\W{1 - \frac{1}{q},\onehalf\left(1 - \frac{1}{q}\right)}(\Gamma \times (0,T))^n}, \quad
		\norm{\bh^2}_{\ZT^4} = \norm{\bh}_{\W{1 - \frac{1}{q},\onehalf\left(1 - \frac{1}{q}\right)}(\Gs \times (0,T))^n}. 
	\end{align*}
\end{theorem}

\subsubsection{Reductions}
To simplify the proof of Theorem \ref{twophase: theorem}, we reduce \eqref{twophase: linear} to the case $ (\bk, g , \vo) = 0 $. 

First of all, we define $ \barv $ as the solution of the parabolic transmission problem
\begin{equation}\label{parabolic: transmission}
	\begin{aligned}
		\hrf \pt \barv - \hdiv \bS(\barv, 0) & = \bk  && \quad \text{in}\  \tO \times (0, T), \\
		\jump{\hv} & = 0 && \quad \text{on}\  \Gamma \times (0, T), \\
		\jump{\bS(\barv, 0)} \hng & = 0 && \quad \text{on}\  \Gamma \times (0, T), \\
		\bS(\barvs, 0) \hngs & = 0 && \quad \text{on}\  \Gs \times (0, T), \\
		\rv{\barv}_{t = 0} & = \vo  && \quad \text{in}\  \tO,
	\end{aligned}
\end{equation}
with $ \bk \in \Lq{q}(\tO \times (0,T)) $ and $ \vo \in \Dqv $. Since the Lopatinskii-Shapiro conditions are satisfied, \eqref{parabolic: transmission} is uniquely solvable in $ \W{2,1}(\hO \times (0,T)) $, thanks to \cite[Theorem 6.5.1]{PS2016}.

Now, we are in the position to reduce $ g $. To this end, we introduce a elliptic transmission problem with Dirichlet boundary
\begin{equation}
	\label{elliptic: phi in reduction}
	\begin{aligned}
		\hDelta \phi & = g - \hdiv \barv =: \tilde{g} && \quad \text{in}\ \tO, \\
		\jump{\hr \phi} & = 0 && \quad 
		\text{on}\ \Gamma, \\
		\jump{\hnab \phi} \cdot \hng & = 0 && \quad 
		\text{on}\ \Gamma, \\
		\hrs \phi_s & = 0 && \quad \text{on}\ \Gs,
	\end{aligned}
\end{equation}
with $ \tilde{g} \in \Lq{q}(\tO) $. Then \eqref{elliptic: phi in reduction} is uniquely solvable by Proposition \ref{transmission laplace: strong proposition}. In addition, with the regularity of $ g $ and $ \bv $, the solution satisfies $ \hnab \phi \in \YT^1 $. Employing the decomposition
\begin{equation}
	\label{decomposition}
	(\hv, \hpi) = (\barv + \hnab \phi + \tv, - \hr \pt \phi + \hnu \hDelta \phi + \tpi),
\end{equation}
we know that $ (\tv, \tpi) $ solves system \eqref{twophase: linear} with $ (\bk, g , \vo) = 0 $ and modified nonvanishing data $ (\bh^1, \bh^2) $  (not to be relabeled) in the right regularity classes having a vanishing trace at $ t = 0 $. Thus, we will focus on the reduced system in the case $ (\bk, g , \vo) = 0 $.

\begin{remark}
	From the decomposition \eqref{decomposition}, regularity of $ \hpi $ given in $ \YT^2 $ indicates that $ \pt \phi $ and $ \hDelta \phi $ must be contained in $ \YT^2 $. Since $ \hnab \phi \in \YT^1 = \Lq{q}(0,T; \W{2}(\tO) \cap \W{1}(\hO))^n \cap \W{1}(0,T; \Lq{q}(\hO))^n $, it is clear that $ \pt \phi, \hDelta \phi  \in \Lq{q}(0,T; \W{1}(\tO)) $. Moreover:
	\begin{itemize}
		\item[i)] Vanishing Dirichlet boundary conditions of $ \phi $ on $ \Gamma $ and $ \Gs $ lead to $ \rv{\jump{\pt \phi}}_{\Gamma} = \rv{\pt \phi}_{\Gs} = 0 $, which naturally satisfy the boundary regularity $ \W{1 - 1/q, (1 - 1/q)/2}(\Gamma \times (0,T)) $ and $ \W{1 - 1/q, (1 - 1/q)/2}(\Gs \times (0,T)) $. Then $ \pt \phi \in \YT^2 $. 
		\item[ii)] For $ \hDelta \phi = \tilde{g} = g - \hdiv \barv $, the boundary regularity for $ \hdiv \barv $ is not a problem due to the zero Neumann boundary of $ \barv $. Thus, to ensure the validation of the regularity for $ \hpi $, we add trace regularities on $ \Gamma $ and $ \Gs $ for $ g $ in $ \ZT^2 $. Namely,
		\begin{equation*}
			\tr_{\Gamma}(g) \in \W{1 - 1/q, (1 - 1/q)/2}(\Gamma \times (0,T)), \quad 
			\tr_{\Gs}(g) \in \W{1 - 1/q, (1 - 1/q)/2}(\Gs \times (0,T)).
		\end{equation*}
		Consequently, $ \hDelta \phi \in \YT^2 $.
	\end{itemize}
\end{remark}

\subsubsection{Proof of Theorem \ref{twophase: theorem}}
As stated in the last section, we analyze the reduced system of \eqref{twophase: linear} with $ (\bk, g , \vo) = 0 $. Due to the outer Neumann boundary condition, the proof is proceeded by a truncation (localization) argument, based on the results given in Appendix \ref{results: linear}. More precisely, with suitable cutoff function, we decompose the system into a two-phase Stokes problem with Dirichlet boundary conditions and a one-phase nonstationary Stokes problem, which are uniquely solvable as in Section \ref{twophaseD} and Abels \cite[Theorem 1.1]{Abels2010} respectively.
\begin{proof}[{Proof of Theorem \ref{twophase: theorem}}]
	\textbf{\textit{Step 1.}} The first step is finding $ (\hvone, \hpione) $ to solve
	\begin{equation}\label{twophase: g}
		\begin{aligned}
			\hr \pt \hvone - \hdiv \bS( \hvone, \hpione ) & = 0 && \quad \text{in}\  \tO \times (0, T), \\
			\hdiv \hvone & = 0 && \quad \text{in}\  \tO \times (0, T), \\
			\jump{\hvone} & = 0 && \quad \text{on}\  \Gamma \times (0, T), \\
			\jump{\bS( \hvone, \hpione )} \hng & = \bh^1 && \quad \text{on}\  \Gamma \times (0, T), \\
			\hvone & = 0 && \quad \text{on}\  \Gs \times (0, T), \\
			\rv{\hvone}_{t = 0} & = 0 && \quad \text{in}\  \hO,
		\end{aligned}
	\end{equation}
	where $ \bh^1 \in \ZT^3 $ with $ \rv{\bh^1}_{t = 0} = 0 $. Since $ \rv{\hvone}_{t = 0} = 0 $, compatibility conditions \eqref{twophaseD: compatibility} holds and then \eqref{twophase: g} admits a unique solution $ (\hvone, \hpi) $ in $ \YT^v $, thanks to Proposition \ref{twophase: Dirichlet boundary}. In addition, we have the estimate
	\begin{equation}\label{hvone}
		\norm{(\hvone, \hpione)}_{\YT^v}
		\leq C \norm{\bh^1}_{\ZT^3},
	\end{equation}
	for some $ C > 0 $ independent of $ \hvone, \hpione, \bh^1 $.
	\\\textbf{\textit{Step 2.}} Now, we construct $ (\hvstwo, \hpistwo) $ to solve the Stokes problem with Neumann boundary condition, which reads
	\begin{equation}\label{stokes: g}
		\begin{aligned}
			\hrs \pt \hvstwo - \hdiv \bS( \hvstwo, \hpistwo ) & = 0 && \quad \text{in}\  \Os \times (0, T), \\
			\hdiv \hvstwo & = 0 && \quad \text{in}\  \Os \times (0, T), \\
			\bS( \hvstwo, \hpistwo ) \hng & = 0 && \quad \text{on}\  \Gamma \times (0, T), \\
			\bS( \hvstwo, \hpistwo ) \hngs & = \bh^2 && \quad \text{on}\  \Gs \times (0, T), \\
			\rv{\hvtwo}_{t = 0} & = 0 && \quad \text{in}\  \Os,
		\end{aligned}
	\end{equation} 
	where $ \bh^2 \in \ZT^4 $ with $ \rv{\bh^2}_{t = 0} = 0 $. 
	Thanks to Theorem 1.1 in Abels \cite{Abels2010} with $ \Gamma_1 = \emptyset $, \eqref{stokes: g} admits a unique solution $ (\hvtwo, \hpitwo) $ in $ \W{2,1}(\tO) \times \Lq{q}(0,T; \WOO{1}(\tO)) $. Due to $ \rv{\hvstwo}_{t = 0} = 0 $, all the compatibility conditions are satisfied. Moreover,
	\begin{equation}\label{hvtwo}
		\norm{(\hvtwo, \hpitwo)}_{\W{2,1}(\tO) \times \Lq{q}(0,T; \WOO{1}(\tO))}
		\leq C \norm{\bh^2}_{\ZT^4},
	\end{equation} 
	for some $ C > 0 $ independent of $ \hvtwo, \hpitwo, \bh^2 $.
	\\\textbf{\textit{Step 3.}} Finally, we combine the regularity results above by truncation. Specifically, let $ \psi(x) \in C_0^\infty (\hO) $ be a cutoff function over $ \hO $ such that 
	\begin{equation}
		\label{cutoff function}
		\psi(x) = 
		\left\{
			\begin{aligned}
				& 1, \quad \text{in a neighborhood of}\  \Of,\\
				& 0, \quad \text{in a neighborhood of}\  \Gs.
			\end{aligned}
		\right.
	\end{equation}
	We define 
	\begin{align*}
		\tv := \psi \hvone + (1 - \psi) \hvtwo, \quad
		\tpi := \psi \hpione + (1 - \psi) \hpitwo.
	\end{align*}
	Then $ (\tv, \tpi) \in \YT^v $ solves 
	\begin{equation}\label{twophase: R1R2}
		\begin{aligned}
			\hr \pt \tv - \hdiv \bS( \tv, \tpi ) & = \bR^1 && \quad \text{in}\  \tO \times (0, T), \\
			\hdiv \tv & = R^2 && \quad \text{in}\  \tO \times (0, T), \\
			\jump{\tv} & = 0 && \quad \text{on}\  \Gamma \times (0, T), \\
			\jump{\bS( \tv, \tpi ) } \hng & = \bh^1 && \quad \text{on}\  \Gamma \times (0, T), \\
			\bS( \tvs, \tpis ) \hngs & = \bh^2 && \quad \text{on}\  \Gs \times (0, T), \\
			\rv{\tv}_{t = 0} & = 0 && \quad \text{in}\  \hO,
		\end{aligned}
	\end{equation}
	where $ \bR^1 $ and $ R^2 $ vanish in $ \Of $, while in $ \Os $,
	\begin{align*}
		\bR^1 
			& = - \bS(\hvsone - \hvstwo, \hpisone - \hpistwo) \hnab \psi \\
			& \quad \qquad - 2 \hnus \left( \hDelta \psi \left( \hvsone - \hvstwo \right) + \left( \hnab \hvsone - \hnab \hvstwo \right) \hnab \psi + \hnab^2 \psi \left( \hvsone - \hvstwo \right) \right), \\
		R^2 & = \hnab \psi \cdot \left( \hvsone - \hvstwo \right).
	\end{align*}
	Since the embedding
	\begin{equation*}
		\WO{2,1}\left( \Os \times (0,T) \right) \hookrightarrow \WO{\onehalf} \left( 0,T; \W{1}(\Os) \right)
	\end{equation*}
	holds, we know $ \hv^i \in \WO{\onehalf} ( 0,T; \W{1}(\Os) ) $, $ i = 1,2 $. For the reduced system, Proposition 8.2.1 and 7.3.5 in Pr\"{u}ss and Simonett \cite{PS2016} imply that $ \hpione $ and $ \hpistwo $ enjoys extra time regularities $ \hpione \in \WO{\alpha}(0,T; \Lq{q}(\hO))  $ and $ \hpistwo \in \WO{\alpha}(0,T; \Lq{q}(\Os)) $ respectively for $ 0 < \alpha < \onehalf ( 1 - \frac{1}{q} ) $. Hence
	\begin{equation*}
		\bR^1 \in \WO{\alpha} \left( 0,T; \Lq{q}(\Os) \right) \cap \Lq{q} \left( 0,T; \W{1}(\Os) \right),
	\end{equation*}
	for some fixed $ 0 < \alpha < \onehalf ( 1 - \frac{1}{q} ) $.
	
	To complete the proof, we still need to prove that the right-hand side terms of \eqref{twophase: R1R2} can be in fact substituted by the right-hand side terms of \eqref{twophase: linear} in appropriate spaces. 
	Since the regularity of $ \hvs^i $ and $ \hpis^i $, $ i = 1,2 $, are not enough to control $ \bR^1 $ and $ R^2 $ for small times, we are going to remove the inhomogeneities $ \bR^1 $ and  $ R^2 $. For $ \bR^1 $, we construct a $ \bphi $ solving the problem
	\begin{equation}
		\label{elliptic: phi bar}
		\begin{aligned}
			\bphi_f & = 0 && \quad \text{in}\ \Of, \\
			\hDelta \bphi_s & = \hdiv \bR^1 && \quad \text{in}\ \Os,\\
			\bphi_s & = 0 && \quad \text{on}\ \Gamma, \\
			\bphi_s & = 0 && \quad \text{on}\ \Gs.
		\end{aligned}
	\end{equation}
	Then we obtain $ \rv{\hnab \bphi}_{t = 0} = \rv{\bR^1}_{t = 0} = 0 $. By elliptic theory and regularity of $ \bR^1 $, \eqref{elliptic: phi bar} admits a unique solution $ \bphi $ satisfying $ \WO{\alpha} ( 0,T; \W{1}(\Os) ) \cap \Lq{q} ( 0,T; \W{2}(\Os) ) $.
	For $ R^2 $, we find a $ \phi $ solving the elliptic transmission problem
	\begin{equation}
		\label{elliptic: phi}
		\begin{aligned}
			\hDelta \phi_f & = 0 && \quad \text{in}\ \Of, \\
			\hDelta \phi_s & = R^2 && \quad \text{in}\ \Os, \\
			\jump{\hr \phi} & = 0 && \quad 
			\text{on}\ \Gamma, \\
			\jump{\hnab \phi} \cdot \hng & = 0 && \quad 
			\text{on}\ \Gamma, \\
			\hrs \phi_s & = 0 && \quad \text{on}\ \Gs.
		\end{aligned}
	\end{equation}
	Then we have $ \rv{\phi}_{t = 0} = 0 $.
	Since $ \hvsone - \hvstwo \in \WO{2,1}(\Os \times (0,T))^n $, $ R^2 \in \WO{2,1}(\Os \times (0,T)) \hookrightarrow \WO{\onehalf}(0,T; \W{1}(\Os)) $. Together with Proposition \ref{transmission laplace: strong proposition}, one concludes that \eqref{elliptic: phi} admits a solution such that $ \hnab \phi $ is unique, with regularity
	\begin{equation*}
		\hnab \phi \in E_0 := \WO{1}( 0,T; \W{1}(\tO) )^n
		\cap \WO{\frac{1}{4}}( 0,T; \W{2}(\tO) )^n.
	\end{equation*}
	For its traces on $ \Gamma $ and $ \Gs $, we have
	\begin{gather*}
		\jump{\hnab \phi} \in E_1 := \WO{1}( 0,T; \W{1 - \frac{1}{q}}(\Gamma) )^n
		\cap \WO{\frac{1}{4}}( 0,T; \W{2 - \frac{1}{q}}(\Gamma) )^n, \\
		\hnab \phi_s \in E_1^s := \WO{1}( 0,T; \W{1 - \frac{1}{q}}(\Gs) )^n
		\cap \WO{\frac{1}{4}}( 0,T; \W{2 - \frac{1}{q}}(\Gs) )^n.
	\end{gather*}
	Besides, $ \hnab^2 \phi \in  $
	\begin{gather*}
		\jump{\hnu \hnab^2 \phi} \in E_2 := \WO{1 - \frac{1}{2q}}( 0,T; \Lq{q}(\Gamma) )^{n \times n}
		\cap \WO{\frac{1}{4}}( 0,T; \W{1 - \frac{1}{q}}(\Gamma) )^{n \times n}, \\
		\hnus \hnab^2 \phi_s \in E_2^s := \WO{1 - \frac{1}{2q}}( 0,T; \Lq{q}(\Gs) )^{n \times n}
		\cap \WO{\frac{1}{4}}( 0,T; \W{1 - \frac{1}{q}}(\Gs) )^{n \times n}.
	\end{gather*}
	Moreover, the estimate holds for a constant $ C $, independent of $ 0 < T < T_0 $,
	\begin{align*}
		& \norm{\hnab \phi}_{E_0}
		+ \norm{\jump{\hnab \phi}}_{E_1}
		+ \norm{\hnab \phi_s}_{E_1^s} \\
		& \qquad \qquad + \norm{\jump{\hnu \hnab^2 \phi}}_{E_2}
		+ \norm{\hnus \hnab^2 \phi}_{E_2^s} 
		\leq C \norm{\hvsone - \hvstwo}_{\W{2,1}(\Os \times (0,T))^n}.
	\end{align*}
	
	Finally, define 
	\begin{align*}
		\bv^\sharp := \tv - \hnab \phi, \quad 
		\pi^\sharp := \tpi + \hr \pt \phi - \bar{\phi} - 2 \hnu \hDelta \phi.
	\end{align*}
	Since $ \rv{\jump{\hr \phi}}_{\Gamma}, \rv{\hrs \phi}_{\Gs} = 0 $, $ \rv{\jump{\hr \pt \phi}}_{\Gamma}, \rv{\hrs \pt \phi}_{\Gs} = 0 $. Then $ (\bv^\sharp, \pi^\sharp) $ solves
	\begin{equation}\label{twophase: NewR1h1h2}
		\begin{aligned}
			\hr \pt \bv^\sharp - \hdiv \bS( \bv^\sharp, \pi^\sharp ) & = {\bR}^1 - \hnab \bphi =: \bR^0 && \quad \text{in}\  \tO \times (0, T), \\
                        \hdiv \bv^\sharp & =0 && \quad \text{in} \ \tO \times (0,T),\\
			\jump{\bv^\sharp} & = \bR' && \quad \text{on}\  \Gamma \times (0, T), \\
			\jump{\bS( \bv^\sharp, \pi^\sharp ) } \hng & = \bh^1 + \bR^3 && \quad \text{on}\  \Gamma \times (0, T), \\
			\bS( \vs^\sharp, \pis^\sharp ) \hngs & = \bh^2 + \bR^4 && \quad \text{on}\  \Gs \times (0, T), \\
			\rv{\bv^\sharp}_{t = 0} & = 0 && \quad \text{in}\  \hO,
		\end{aligned}
	\end{equation}
	where
	\begin{gather*}
		 \hdiv \bR^0 = 0, \quad 
		\bR' = - \jump{\hnab \phi}\\
		 \bR^3 = \jump{2 \hnu \hnab^2 \phi} \hng -  \jump{2 \hnus \hDelta \phi} \hng  \quad
		\bR^4 = 2 \hnus \hnab^2 \phi_s \hngs - 2 \hnus \hDelta \phi_s \hngs.
	\end{gather*}
%
	$ \bR^0 $ can be seen as a Helmholtz projection of $ \bR^1 $ and
	\begin{equation*}
		\bR^0 \in \WO{\alpha} ( 0,T; \Lq{q}(\Os) )^n
		\cap \Lq{q} ( 0,T; \W{2 \alpha}(\Os) )^n, \quad \text{for all}\ 0 < \alpha < \onehalf - \frac{1}{2q}.
	\end{equation*}
	By Lemma \ref{embedding: W alpha},
	\begin{equation*}
		\bR^0 \in C ( [0,T]; \W{2\alpha - \frac{2}{q}}(\Os) )^n \hookrightarrow C ( [0,T]; \Lq{q}(\Os) )^n
	\end{equation*}
	holds for $ \frac{1}{q} < \alpha < \onehalf - \frac{1}{2q} $.
	Hence, for $ \rv{\bR^0}_{t = 0} = \rv{( \hnab \bphi - \bR^1 )}_{t = 0} = 0 $,
	\begin{align*}
		\norm{\bR^0}_{\ZT^1} & 
	 		\leq C T^{\frac{1}{q}} \norm{\bR^0}_{C ( [0,T]; \Lq{q}(\Os) )^n} 
	 		\leq C T^{\frac{1}{q}} \norm{\bR^0}_{\WO{\alpha} ( 0,T; \Lq{q}(\Os) )^n
				\cap \Lq{q} ( 0,T; \W{2 \alpha}(\Os) )^n} \\
			& \qquad \qquad \leq C T^{\frac{1}{q}} \left( \max_{i = 1,2} \norm{(\hv^i, \hpi^i)}_{\YT^v} \right) 
			\leq C T^{\frac{1}{q}} \norm{(\bk, g, \bh^1, \bh^2, \vo)}_{\ZT^v \times \Dqv},
	\end{align*}
	for $ 0 < T < T_0 $. According to Appendix \ref{results: linear}, the regularity space of $ \bR' $ is defined as $ \ZT' := \W{2 - \frac{1}{q}, 1 - \frac{1}{2q}} (\Gamma \times (0,T)) $. Then with Lemma \ref{embedding: Ws Wr} and $ \W{s}(0,T; X) \hookrightarrow C([0,T]; X) $ for $ sq > 1 $,
	\begin{align*}
		\norm{\bR'}_{\ZT'} 
			& \leq C \left( 
			\norm{\jump{\hnab \phi}}_{\Lq{q}( 0,T; \W{2 - \frac{1}{q}}(\Gamma) )^n} \right. \\
			& \left. \qquad \qquad 
			+ \norm{\jump{\hnab \phi}}_{\Lq{q}( 0,T; \Lq{q}(\Gamma) )^n} 
			+ \seminorm{\jump{\hnab \phi}}_{\W{ 1 - \frac{1}{2q} } ( 0,T; \Lq{q}(\Gamma) )^n } 
			\right) \\
			& \quad \leq C T^\frac{1}{q} 
			\norm{\jump{\hnab \phi}}_{\WO{\frac{1}{4}}( 0,T; \W{2 - \frac{1}{q}}(\Gamma) )^n} 
			+ C T^\frac{1}{2q} \seminorm{\jump{\hnab \phi}}_{\WO{1} ( 0,T; \W{1 - \frac{1}{q}}(\Gamma) )^n }  \\
			& \quad \leq C T^\frac{1}{2q} \left( \max_{i = 1,2} \norm{(\hv^i, \hpi^i)}_{\YT^v} \right) 
			\leq C T^{\frac{1}{2q}} \norm{(\bk, g, \bh^1, \bh^2, \vo)}_{\ZT^v \times \Dqv}. 
	\end{align*}
	Since $ \Gamma $ and $ \Gs $ are of class $ C^3 $, $ \hng $ and $ \hngs $ are contained in $ C^2 $. Then we obtain
	\begin{align*}
		\norm{\bR^3}_{\ZT^3} 
			& \leq C \left( 
				\norm{\jump{\hnab^2 \phi}}_{\Lq{q}( 0,T; \W{1 - \frac{1}{q}}(\Gamma) )^{n \times n}} \right. \\
				& \left. \qquad \qquad + \norm{\jump{\hnab^2 \phi}}_{\Lq{q}( 0,T; \Lq{q}(\Gamma) )^{n \times n}} 
				+ \seminorm{\jump{\hnab^2 \phi}}_{\W{\onehalf \left( 1 - \frac{1}{q} \right)} ( 0,T; \Lq{q}(\Gamma) )^{n \times n} } \right) \\
			& \leq C T^\frac{1}{q}  
				\norm{\jump{\hnab^2 \phi}}_{\WO{\frac{1}{4}}( 0,T; \W{1 - \frac{1}{q}}(\Gamma) )^{n \times n}} 
				+ C T^{\onehalf} \seminorm{\jump{\hnab^2 \phi}}_{\WO{1 - \frac{1}{2q}} ( 0,T; \Lq{q}(\Gamma) )^{n \times n} }  \\
			& \leq C T^{\frac{1}{q}} \left( \max_{i = 1,2} \norm{(\hv^i, \hpi^i)}_{\YT^v} \right) 
			\leq C T^{\frac{1}{q}} \norm{(\bk, g, \bh^1, \bh^2, \vo)}_{\ZT^v \times \Dqv},
	\end{align*}
	with the help of Lemma \ref{embedding: Ws Wr}. Similarly,
	\begin{align*}
		\norm{\bR^4}_{\ZT^4} \leq C T^{\frac{1}{q}} \left( \max_{i = 1,2} \norm{(\hv^i, \hpi^i)}_{\YT^v} \right) 
		\leq C T^{\frac{1}{q}} \norm{(\bk, g, \bh^1, \bh^2, \vo)}_{\ZT^v \times \Dqv}.
	\end{align*}
	Taking $ T_0 $ sufficiently small such that $ C T_0^{\frac{1}{2q}} \leq \onehalf $, we have 
	\begin{equation*}
		\norm{\bR^0(y)}_{\ZT^1} + \norm{\bR'(y)}_{\ZT'} + \norm{\bR^3(y)}_{\ZT^3} + \norm{\bR^4(y)}_{\ZT^4}
		\leq \onehalf \norm{y}_{\ZT^v \times \Dqv},
	\end{equation*}
	for $ y = (\bk, g, \bh^1, \bh^2, \vo)^\top $. By a Neumann series argument, 
	\begin{equation*}
		\Phi : \tilde{y} \mapsto \tilde{y} + (\bR^0, 0, \bR', \bR^3, \bR^4, 0)^\top(\tilde{y})
	\end{equation*}
	is invertible for $ \tilde{y} = (\bk, g, 0, \bh^1, \bh^2, \vo)^\top $. Consequently, replacing $ \tilde{y} $ by $ \Phi^{-1}(\tilde{y}) $ in \eqref{twophase: R1R2} yields the solvability of \eqref{twophase: NewR1h1h2} for $ 0 < T < T_0 \leq 1/(2C)^{2q} $. 
	Solving \eqref{twophase: linear} iteratively on $ [0, T_0] $, $ [T_0, 2 T_0] $, $ \dots $, with initial values $ \vo $, $ \rv{\vo}_{t = T_0} $, $ \dots $, one obtains the solvability for any $ T_0 > 0 $.
	Additionally, estimate \eqref{thm3.1:estimates} is a result of \eqref{hvone} and \eqref{hvtwo}.
	This completes the proof.
\end{proof}

\begin{remark}
	Since $ \hc $ is contained in $ \YT^3 = \Lq{q}(0,T; \W{2}(\tO) ) \cap \W{1}(0,T; \Lq{q}(\hO)) $, which will be given in Section \ref{parabolic: neumann}, one can easily verify that 
	$ \hc \in \ZT^2 $. Hence, we replace $ g $ in \eqref{twophase: linear} by $ g + \frac{\gamma \beta}{\hrs} \hcs $ with the same existence and regularity results to the original linear system. To be more precise, we find $ (\barv, \barpi) \in \YT^v $ to solve
	\begin{equation*}
		\begin{aligned}
			\hr \pt \barv - \hdiv \bS( \barv, \barpi ) & = 0 && \quad \text{in}\  \tO \times (0, T), \\
			\hdiv \barv & = \frac{\gamma \beta}{\hrs} \hcs && \quad \text{in}\  \tO \times (0, T), \\
			\jump{\barv} & = 0 && \quad \text{on}\  \Gamma \times (0, T), \\
			\jump{\bS( \barv, \barpi )} \cdot \hng & = 0 && \quad \text{on}\  \Gamma \times (0, T), \\
			\bS( \barvs, \barpis ) \cdot \hngs & = 0 && \quad \text{on}\  \Gs \times (0, T), \\
			\rv{\barv}_{t = 0} & = 0 && \quad \text{in}\  \hO,
		\end{aligned}
	\end{equation*}
	with $ \hc \in \ZT^2 $, thanks to Theorem \ref{twophase: theorem}. Then $ (\hv + \barv, \hpi + \barpi) $ solves the original linear system.
\end{remark}

\subsection{Heat equations with Neumann boundary condition} \label{parabolic: neumann}
From \eqref{Fluid1cf}--\eqref{initialv1c}, we have two decoupled systems with given functions $ (f^1, f^2, f^3) $, that is,
\begin{equation}\label{parabolic: cf}
	\left\{
		\begin{aligned}
			\pt \hcf - \hDf \hDelta \hcf & = f_f^1 \quad && \text{in}\  \Of \times (0, T), \\
			\hDf \hnab \hcf \cdot \hng & = f_f^2 \quad && \text{on}\  \Gamma \times (0, T), \\
			\rv{\hcf}_{t = 0} & = \co_f \quad && \text{in}\  \Of,
		\end{aligned}
	\right.
\end{equation}
and
\begin{equation}\label{parabolic: cs}
	\left\{
		\begin{aligned}
			\pt \hcs - \hDs \hDelta \hcs & = f_s^1 \quad && \text{in}\  \Os \times (0, T), \\
			\hDs \hnab \hcs \cdot \hng & = f_s^2 \quad && \text{on}\  \Gamma \times (0, T), \\
			\hDs \hnab \hcs \cdot \hng & = f^3 \quad && \text{on}\  \Gs \times (0, T), \\
			\rv{\hcs}_{t = 0} & = \co_s \quad && \text{in}\  \Os,
		\end{aligned}
	\right.
\end{equation}
According to the maximal $ L^q $-regularity results we introduced in Appendix \ref{parabolic}, we immediately have following theorem.
\begin{theorem}
	\label{parabolic: theorem}
	Let $ q > n + 2 $, $ \hO $ a bounded domain with $ \Gs \in C^3$, $ \Gamma $ a closed hypersurface of class $ C^3 $. Assume that $ (f^1, f^2 , f^3) $ are known functions contained in $ \ZT^c $ and $ \co \in \Dqc $ with compatibility conditions
	\begin{equation*}
		\rv{\hDf \hnab \co_f \cdot \hng}_\Gamma = \rv{f_f^2}_{t = 0}, \quad 
		\rv{\hDs \hnab \co_s \cdot \hng}_\Gamma = \rv{f_s^2}_{t = 0}, \quad 
		\rv{\hDs \hnab \co_s \cdot \hngs}_{\Gs} = \rv{f^3}_{t = 0}.
	\end{equation*}
	Then the parabolic problems \eqref{parabolic: cf} and \eqref{parabolic: cs} admit unique strong solutions $ \hcf $ and $ \hcs $ in $ \YT^3 $ respectively. Moreover, there exist a constant $ C > 0 $ and a time $ T_0 > 0 $ such that for $ 0 < T < T_0 $,
	\begin{equation}
		\label{thm3.2:estimates}
		\norm{\hc}_{\YT^3} \leq C \norm{(f^1, f^2, f^3, \co)}_{\ZT^c \times \Dqc},
	\end{equation}
	where $ \ZT^c := \ZT^5 \times \ZT^6 \times \ZT^7 $ with
	\begin{gather*}
		\ZT^5 := \Lq{q}( 0,T; \Lq{q}(\tO) ), \\
		\ZT^6 := \W{1 - \frac{1}{q}, \onehalf \left( 1 - \frac{1}{q} \right)} \left( \Gamma \times (0,T) \right), \quad
		\ZT^7 := \W{1 - \frac{1}{q}, \onehalf \left( 1 - \frac{1}{q} \right)} \left( \Gs \times (0,T) \right).
	\end{gather*}
\end{theorem}

\subsection{Ordinary differential equations for foam cells and growth}
Given functions $ (f^4, f^5) $ in $ \Os $, we have
\begin{align}
	\label{foam cells}
	\begin{aligned}
		\pt \hcss - \beta \hcs = f^4, \quad & \text{in}\  \Os \times (0, T), \\
		\rv{\hcss}_{t = 0} = 0, \quad & \text{in}\  \Os,
	\end{aligned} \\
	\label{growth metric}
	\begin{aligned}
		\pt \hg - \frac{\gamma \beta}{n \hrs} \hcs = f^5, \quad & \text{in}\  \Os \times (0, T), \\
		\rv{\hg}_{t = 0} = 1, \quad & \text{in}\  \Os.
	\end{aligned}
\end{align}
Since \eqref{foam cells} and \eqref{growth metric} are linear ordinary differential equations, it follows easily from $ \hcs \in \YT^3 $ in Theorem \ref{parabolic: theorem} that system \eqref{foam cells} and \eqref{growth metric} admits a unique solution $ \hcss $ and $ \hg $, respectively, both in $ \ZT^8 := \Lq{q}(0,T; \W{1}(\Os)) $. Moreover, there exists a constant $ C $ independent of $ T $ such that
\begin{equation}
	\label{estimate: foam cells and growth}
	\norm{\hcss}_{\YT^4} + \norm{\hg}_{\YT^4}
	\leq C \norm{(f^4, f^5, \hc)}_{\ZT^8 \times \ZT^8 \times \YT^3}.
\end{equation}

\section{Local in time existence}
\label{section 4}
This section is intended to prove Theorem \ref{main}. 

\subsection{Some key estimates}
Before showing Theorem \ref{main}, let us give some useful estimates with regard to the deformation gradient $ \inv{\hF} $and vector-valued Sobolev-Slobodeckij space $ \W{\onehalf - \varepsilon}(0,T; \Lq{q}(\OM)) $.
\begin{lemma}[Estimates on deformation gradient]
	\label{F}
	Let $ q > n $, $ n \geq 2 $. $ \hF(\hv) $ is a deformation gradient defined in \eqref{DG} corresponding to a function $ \hv \in \YT^1 $. Then for every $ R > 0 $, there are a constant $ C = C(R) > 0 $ and a finite time $ 0 < T_R < 1 $ depending on $ R $ such that for all $ 0 < T < T_R $, $ \inv{\hF} $ exists and
	\begin{enumerate}[label = (\arabic*)]
		\item 
		$ \displaystyle \norm{\inv{\hF}}_{\Lq{\infty}\left( 0,T; \W{1}(\tO) \right)^{n \times n}} \leq C $, \quad
		$ \norm{\pt \inv{\hF}}_{\Lq{q} \left( 0,T; \W{1}(\tO) \right)^{n \times n}} 
		\leq C \norm{\hv}_{\YT^1} $;
		\item $ \displaystyle \norm{\FOinv}_{\Lq{\infty}(0,T; \W{1}(\tO))^{n \times n}} \leq C \TQ \norm{\hv}_{\YT^1} $;
		\item $\displaystyle \sup_{0 \leq t \leq T} \left( \int_{0}^{t} \frac{\norm{\Dh\left( \FOinv \right)(\cdot, t)}_{\W{1}(\tO)^{n \times n}}^q}{h^{1+\frac{q}{2q'}}} \d h \right)^{\frac{1}{q}} \leq C \TQQ \norm{\hv}_{\YT^1}$;
		\item $ \displaystyle \seminorm{ \FOinv }_{\W{\onehalf\left(1-\frac{1}{q}\right)} \left(0,T; \W{1}(\tO) \right)^{n \times n}} \leq C T^{\frac{1}{q} + \frac{1}{2q'}} \norm{\hv}_{\YT^1} $,
	\end{enumerate}
	for all $ \norm{\hv}_{\YT^1} \leq R $, where $ \Dh f(t) := f(t) - f(t - h) $ is a difference of the time shift for a function $ f $. 
	Moreover, for another $ \hu \in \YT^1 $ with $ \norm{\hu}_{\YT^1} \leq R $ and $ \rv{\hv}_{t = 0} = \rv{\hu}_{t = 0} $, we have 
	\begin{enumerate}[start = 5, label = (\arabic*)]
		\item $ \displaystyle \norm{ \inv{\hF}(\hu) - \inv{\hF}(\hv) }_{\Lq{\infty}\left(0,T; \W{1}(\tO)\right)^{n \times n}} \leq C \TQ \norm{\hu - \hv}_{\YT^1} $; \\
		$ \norm{\pt \inv{\hF}(\hu) - \pt \inv{\hF}(\hv)}_{\Lq{q}\left(0,T; \Lq{\infty}(\tO)\right)^{n \times n}} \leq C T^{\frac{1}{q} - \frac{1}{r}} \norm{\hu - \hv}_{\YT^1} $;
		\item $ \displaystyle \sup_{0 \leq t \leq T} \left( \int_{0}^{t} \frac{\norm{ \Dh\left( \inv{\hF}(\hu) - \inv{\hF}(\hv) \right)(\cdot, t) }_{\W{1}(\tO)^{n \times n}}^q}{h^{1+\frac{q}{2q'}}} \d h \right)^{\frac{1}{q}}
		\leq C \TQQ \norm{\hu - \hv}_{\YT^1} $;
		\item $ \displaystyle \seminorm{ \inv{\hF}(\hu) - \inv{\hF}(\hv) }_{\W{\onehalf\left(1-\frac{1}{q}\right)} \left(0,T; \W{1}(\tO) \right)^{n \times n}} \leq C T^{\frac{1}{q} + \frac{1}{2q'}} \norm{\hu - \hv}_{\YT^1} $,
	\end{enumerate}
	where $ r = \frac{q^2}{n} $.
\end{lemma}
\begin{proof}
	
	
	Recalling from \eqref{DG} the definition of $ \hF $ that
	\begin{equation*}
		\hF = \bbI + \int_{0}^{t} \hnab \hv(X,\tau) \d \tau, \quad \forall X \in \hO.
	\end{equation*}
	Then we have 
	\begin{align*}
		\sup_{0 \leq t \leq T} \norm{\hF - \bbI}_{\W{1}(\tO)^{n \times n}} 
		= \sup_{0 \leq t \leq T} \norm{\int_0^t \hnab \hv(X, \tau) \d \tau}_{\W{1}(\tO)^{n \times n}}
		\leq C T^{\frac{1}{q'}} R,
	\end{align*}
	for all $ \norm{\hv}_{\YT^1} \leq R $. Choosing $ T_R > T $ small such that $ C T_R^{\frac{1}{q'}} R \leq \frac{1}{2M_q} $, we know
	\begin{align*}
		\sup_{0 \leq t \leq T} \norm{\hF - \bbI}_{\W{1}(\tO)^{n \times n}} \leq \frac{1}{2M_q},
	\end{align*}
	where $ M_q $ is the constant of multiplication of $ \W{1}(\tO) $, see Lemma \ref{algebra}. According to the Neumann series (see \cite[Section 5.7]{Alt2016}), $ \inv{\hF} $ does exist and
	\begin{align*}
		\inv{\hF}
		= \left( \hF - \bbI + \bbI \right)^{-1}
		= \left( \bbI - \left( \bbI - \hF \right) \right)^{-1}
		= \sum_{k = 0}^{\infty} \left( \bbI - \hF \right)^k.
	\end{align*}
	Then from Lemma \ref{algebra}, one obtains
	\begin{align*}
		& \sup_{0 \leq t \leq T} \norm{\inv{\hF}}_{\W{1}(\tO)^{n \times n}} 
		\leq \sup_{0 \leq t \leq T} \sum_{k = 0}^{\infty} \norm{\left( \bbI - \hF \right)^k}_{\W{1}(\tO)^{n \times n}} \\
		& \qquad \qquad \leq \frac{1}{M_q} \sum_{k = 0}^{\infty} \left( M_q \sup_{0 \leq t \leq T} \norm{ \bbI - \hF }_{\W{1}(\tO)^{n \times n}} \right)^k 
		\leq \frac{1}{M_q} \sum_{k = 0}^{\infty} \left( \onehalf \right)^k = \frac{2}{M_q},
	\end{align*}
%
%
%
	Consequently, if follows from \eqref{DtA} and Lemma \ref{algebra} that
	\begin{align*}
		& \norm{\pt \inv{\hF}}_{\Lq{q} \left( 0,T; \W{1}(\tO) \right)^{n \times n}} \\
		& \qquad \leq M_q^2 \norm{\inv{\hF}}_{\Lq{\infty}\left( 0,T; \W{q}(\tO) \right)^{n \times n}}^2 \norm{ \hnab \hv }_{\Lq{q}\left( 0,T: \W{1}(\tO) \right)^{n \times n}} 
		\leq C \norm{\hv}_{\YT^1},
	\end{align*}
	for all $ 0 < T < T_R $ and
	\begin{align*}
		\norm{\FOinv}_{\Lq{\infty}\left(0,T; \W{1}(\tO)\right)^{n \times n}}
		\leq \int_{0}^{T} \norm{ \pt \inv{\hF} (\cdot, \tau) }_{\W{1}(\tO)^{n \times n}} \d \tau \leq C \TQ \norm{\hv}_{\YT^1},
	\end{align*}
	where $ C = C(R) $ depends on $ R $. These estimates prove first two statements.
	
	For the third and fourth statements, we have
	\begin{equation*}
		\begin{aligned}
			\norm{\Dh\left( \FOinv \right) (\cdot, t)}_{\W{1}(\tO)^{n \times n}} 
			\leq \int_{t - h}^{t} \norm{\pt \inv{\hF} (\cdot, \tau)}_{\W{1}(\tO)^{n \times n}} \d \tau \leq C h^{\frac{1}{q'}} \norm{\hv}_{\YT^1},
		\end{aligned}
	\end{equation*}
	which can be used to deduce
	\begin{align*}
		& \sup_{0 \leq t \leq T} \left( \int_{0}^{t} \frac{\norm{\Dh\left( \FOinv \right)(\cdot, t)}_{\W{1}(\tO)^{n \times n}}^q}{h^{1+\frac{q}{2q'}}} \d h \right)^{\frac{1}{q}} \\
		& \qquad \leq C \sup_{0 \leq t \leq T} \left( \int_{0}^{t} h^{-1+\frac{q}{2q'}} \d h \right)^{\frac{1}{q}} \norm{\hv}_{\YT^1} 
		= C 2 q' \sup_{0 \leq t \leq T} t^{\frac{1}{2q'}} \norm{\hv}_{\YT^1}
		\leq C \TQQ \norm{\hv}_{\YT^1},
	\end{align*}
	and therefore from \eqref{LI} and the definition of Sobolev-Slobodeckij space,
	\begin{align*}
		\seminorm{ \FOinv }_{\W{\onehalf\left(1-\frac{1}{q}\right)} \left(0,T; \W{1}(\tO) \right)^{n \times n}} 
		\leq C T^{\frac{1}{q} + \frac{1}{2q'}} \norm{\hv}_{\YT^1}.
	\end{align*}
	
	For the rest statements, we notice from \eqref{DG} that
	\begin{equation*}
		\hF(\hu) - \hF(\hv) = \int_{0}^{t} \left( \hnab \hu - \hnab \hv \right)(X,\tau) \d \tau.
	\end{equation*}
	Then for all $ 0 < T < T_R $,
	\begin{equation*}
		\sup_{0 \leq t \leq T} \norm{\hF(\hu) - \hF(\hv)}_{\W{1}(\tO)^{n \times n}} \leq C \TQ \norm{\hu - \hv}_{\YT^1}.
	\end{equation*}
	Since
	\begin{align*}
		\inv{\hF}(\hu) - \inv{\hF}(\hv) 
		= - \inv{\hF}(\hu) \left( \hF(\hu) - \hF(\hv) \right) \inv{\hF}(\hv),
	\end{align*}
	it follows from the multiplication property of $ \W{1}(\tO) $ again that for all $ 0 < T < T_R $,
	\begin{align*}
		& \sup_{0 \leq t \leq T} \norm{\inv{\hF}(\hu) - \inv{\hF}(\hv)}_{\W{1}(\tO)^{n \times n}} \\
		& \qquad \leq M_q^2 \sup_{0 \leq t \leq T} \norm{\inv{\hF}(\hu)}_{\W{1}(\tO)^{n \times n}} \norm{\inv{\hF}(\hv)}_{\W{1}(\tO)^{n \times n}} \norm{\hF(\hu) - \hF(\hv)}_{\W{1}(\tO)^{n \times n}} \\
		& \qquad \qquad \leq C \TQ \norm{\hu - \hv}_{\YT^1}.
	\end{align*}
	Moreover,
	\begin{align}
		& \pt \inv{\hF}(\hu) - \pt \inv{\hF}(\hv)  = - \pt \inv{\hF}(\hu) \left( \hF(\hu) - \hF(\hv) \right) \inv{\hF}(\hv) \label{ptF: difference} \\
		& \qquad - \inv{\hF}(\hu) \pt \left( \hF(\hu) - \hF(\hv) \right) \inv{\hF}(\hv)  - \inv{\hF}(\hu) \left( \hF(\hu) - \hF(\hv) \right) \pt \inv{\hF}(\hv) \nonumber.
	\end{align}
	Hence
	\begin{align*}
		& \norm{\pt \inv{\hF}(\hu) - \pt \inv{\hF}(\hv)}_{\Lq{q}\left( 0,T; \Lq{\infty}(\tO) \right)^{n \times n}} \\
		& \qquad \leq \norm{\pt \inv{\hF}(\hu) \left( \hF(\hu) - \hF(\hv) \right) \inv{\hF}(\hv)}_{\Lq{q}\left( 0,T; \Lq{\infty}(\tO) \right)^{n \times n}} \\
		& \qquad \quad + \norm{\inv{\hF}(\hu) \left( \hnab \hu - \hnab \hv \right) \inv{\hF}(\hv)}_{\Lq{q}\left( 0,T; \Lq{\infty}(\tO) \right)^{n \times n}} \\
		& \qquad \quad + \norm{\inv{\hF}(\hu) \left( \hF(\hu) - \hF(\hv) \right) \pt \inv{\hF}(\hv)}_{\Lq{q}\left( 0,T; \Lq{\infty}(\tO) \right)^{n \times n}}
		=: F_1 + F_2 + F_3.
	\end{align*}
	From the embedding \eqref{embedding: W21 W1}, 
	we know that for $ \hv \in \YT^1 $,
	\begin{equation*}
		\sup_{0 \leq t \leq T} \norm{\hnab \hv}_{\Lq{q}(\tO)^{n \times n}}
		\leq C \left( \norm{\hv}_{\YT^1} + \norm{\rv{\hv}_{t = 0}}_{\W{1}(\tO)} \right).
	\end{equation*}
	The Gagliardo–Nirenberg inequality tells us
	\begin{align*}
		\norm{\hnab \hv}_{\Lq{\infty}(\tO)^{n \times n}}
		\leq C \norm{\hnab \hv}_{\Lq{q}(\tO)^{n \times n}}^{1 - \frac{n}{q}} \norm{\hnab \hv}_{\W{1}(\tO)^{n \times n}}^{\frac{n}{q}}.
	\end{align*}
	For $ r = \frac{q^2}{n} > q $, we obtain
	\begin{align*}
		\norm{\hnab \hv}_{\Lq{r} \left( 0,T; \Lq{\infty}(\tO) \right)^{n \times n}}
		& \leq C \norm{\hnab \hv}_{\Lq{\infty}\left( 0,T; \Lq{q}(\tO) \right)^{n \times n}}^{1 - \frac{n}{q}} \norm{\hnab \hv}_{\Lq{q} \left( 0,T; \W{1}(\tO) \right)^{n \times n}}^{\frac{n}{q}} \leq C(R).
	\end{align*}
	Then, 
	\begin{align*}
		\norm{\hnab \hv}_{\Lq{q} \left( 0,T; \Lq{\infty}(\tO) \right)^{n \times n}}
		\leq T^{\frac{1}{q} - \frac{1}{r}} \norm{\hnab \hv}_{\Lq{r} \left( 0,T; \Lq{\infty}(\tO) \right)^{n \times n}}
		\leq C(R) T^{\frac{1}{q} - \frac{1}{r}},
	\end{align*}
	and also, for $ \hu \in \YT^1 $, $ \norm{\hu}_{\YT^1} \leq R $,
	\begin{equation}
		\label{hv: difference}
		\norm{\hnab \hv - \hnab \hu}_{\Lq{q} \left( 0,T; \Lq{\infty}(\tO) \right)^{n \times n}}
		\leq C(R) T^{\frac{1}{q} - \frac{1}{r}} \norm{\hv - \hu}_{\YT^1}.
	\end{equation}
	Consequently, with $ \W{1}(\tO) \hookrightarrow \Lq{\infty}(\tO) $ for $ q > n $,
	\begin{align*}
		F_1 
		& \leq \norm{\pt \inv{\hF}(\hu)}_{\Lq{q}\left( 0,T; \Lq{\infty}(\tO) \right)^{n \times n}} \\
		& \qquad \qquad \times 
		\norm{\hF(\hu) - \hF(\hv)}_{\Lq{\infty}\left( 0,T; \Lq{\infty}(\tO) \right)^{n \times n}} \norm{\inv{\hF}(\hv)}_{\Lq{\infty}\left( 0,T; \Lq{\infty}(\tO) \right)^{n \times n}} \\
		& \quad \leq \norm{\hnab \hu}_{\Lq{q} \left( 0,T; \Lq{\infty}(\tO) \right)^{n \times n}}
		\norm{\inv{\hF}(\hu)}_{\Lq{\infty}\left( 0,T; \Lq{\infty}(\tO) \right)^{n \times n}}^2 \\
		& \qquad \qquad \times
		\norm{\hF(\hu) - \hF(\hv)}_{\Lq{\infty}\left( 0,T; \Lq{\infty}(\tO) \right)^{n \times n}} 
		\norm{\inv{\hF}(\hv)}_{\Lq{\infty}\left( 0,T; \Lq{\infty}(\tO) \right)^{n \times n}} \\
		& \quad \leq C T^{\frac{1}{q} - \frac{1}{r}} \norm{\hu - \hv}_{\YT^1}.
	\end{align*}
	Similarly,
	\begin{align*}
		F_2 \leq C T^{\frac{1}{q} - \frac{1}{r}} \norm{\hu - \hv}_{\YT^1}, \quad F_3 \leq C T^{\frac{1}{q} - \frac{1}{r}} \norm{\hu - \hv}_{\YT^1}.
	\end{align*}
	Thus,
	\begin{align*}
		\norm{\pt \inv{\hF}(\hu) - \pt \inv{\hF}(\hv)}_{\Lq{q}\left( 0,T; \Lq{\infty}(\tO) \right)^{n \times n}} 
		\leq C T^{\frac{1}{q} - \frac{1}{r}} \norm{\hu - \hv}_{\YT^1}.
	\end{align*}
	Moreover, we can also conclude from \eqref{ptF: difference} that 
	\begin{align*}
		\norm{\pt \inv{\hF}(\hu) - \pt \inv{\hF}(\hv)}_{\Lq{q}\left( 0,T; \W{1}(\tO) \right)^{n \times n}} 
		\leq C \norm{\hu - \hv}_{\YT^1}.
	\end{align*}
	Using that $ \left( \inv{\hF_0}(\hu) - \inv{\hF_0}(\hv) \right) = 0 $,
	\begin{equation*}
		\begin{aligned}
			& \norm{\Dh\left( \inv{\hF}(\hu) - \inv{\hF}(\hv) \right) (\cdot, t)}_{\W{1}(\tO)^{n \times n}} \\
			& \qquad \leq \int_{t - h}^{t} \norm{\pt \left( \inv{\hF}(\hu) - \inv{\hF}(\hv) \right) (\cdot, \tau)}_{\W{1}(\tO)^{n \times n}} \d \tau \leq C h^{\frac{1}{q'}} \norm{\hu - \hv}_{\YT^1}.
		\end{aligned}
	\end{equation*}
	Therefore, for all $ 0 < T < T_R $,
	\begin{align*}
		& \sup_{0 \leq t \leq T} \left( \int_{0}^{h} \frac{\norm{ \Dh\left( \inv{\hF}(\hu) - \inv{\hF}(\hv) \right)(\cdot, t) }_{\W{1}(\tO)^{n \times n}}^q}{h^{1+\frac{q}{2q'}}} \d h \right)^{\frac{1}{q}} \\
		& \qquad \leq C \sup_{0 \leq t \leq T} t^{\frac{1}{2q'}} \norm{\hu - \hv}_{\YT^1}
		= C \TQQ \norm{\hu - \hv}_{\YT^1}.
	\end{align*}
	Again with the help of \eqref{LI} and the definition of Sobolev-Slobodeckij space, one obtains the last statement.	
	This completes the proof.
\end{proof}

\begin{lemma}
	\label{Fw}
	Under the assumption of Lemma \ref{F}, there exist a constant $ C = C(R) > 0 $ and a finite time $ T_R > 0 $ depending on $ R $ such that for all $ 0 < T < T_R $ and for two arbitrary functions $ f(X,t) \in \Lq{q} ( 0,T; \W{1}(\tO) ) $ and $ \bm{f} \in \Lq{q} ( 0,T; \W{2}(\tO) )^n $,
	\begin{enumerate}[label = (\arabic*)]
		\item $ \displaystyle \norm{\left( \inv{\hF}(\hv) - \bbI \right) f}_{\Lq{q} \left( 0,T; \W{1}(\tO) \right)^n} 
		\leq C \TQ \norm{f}_{\Lq{q} \left( 0,T; \W{1}(\tO) \right)} \norm{\hv}_{\YT^1} $; \\
		$ \displaystyle \norm{\left( \inv{\hF}(\hv) - \bbI \right) \left( \hnab \bm{f} \right)}_{\Lq{q} \left( 0,T; \W{1}(\tO) \right)^{n \times n}} 
		\leq C \TQ \norm{\bm{f}}_{\Lq{q} \left( 0,T; \W{2}(\tO) \right)^n} \norm{\hv}_{\YT^1} $;
		\item $ \displaystyle \norm{\left( \inv{\hF}(\hu) - \inv{\hF}(\hv) \right) f}_{\Lq{q} \left( 0,T; \W{1}(\tO) \right)^n} \leq C \TQ \norm{f}_{\Lq{q} \left( 0,T; \W{1}(\tO) \right)} \norm{\hu - \hv}_{\YT^1} $; \\
		$ \displaystyle \norm{\left( \inv{\hF}(\hu) - \inv{\hF}(\hv) \right) \left( \hnab \bm{f} \right)}_{\Lq{q} \left( 0,T; \W{1}(\tO) \right)^{n \times n}} \\
		\leq C \TQ \norm{\bm{f}}_{\Lq{q} \left( 0,T; \W{2}(\tO) \right)^n} \norm{\hu - \hv}_{\YT^1} $;
		\item $ \displaystyle \norm{\left( \inv{\hF}(\hu) - \inv{\hF}(\hv) \right) \left( \hnab \bm{f} \inv{\hF}(\hu) \right)}_{\Lq{q} \left( 0,T; \W{1}(\tO) \right)^{n \times n}} \\
		\leq C \TQ \norm{\bm{f}}_{\Lq{q} \left( 0,T; \W{2}(\tO) \right)^n} \norm{\hu - \hv}_{\YT^1} $.
	\end{enumerate}
\end{lemma}
\begin{proof}
	The key point to deduce these estimates is to use the multiplication property of $ \W{1}(\tO) $ with $ q > n $, which was given in Lemma \ref{algebra}. Then Lemma \ref{F} implies these results.
\end{proof}

\begin{lemma}
	\label{gradu}
	Let $ 1 < q < \infty $, $ T_0 > 0 $ and $ \OM \subseteq \bbr^n $, $ n \geq 2 $, be a bounded domain with $ C^{1,1} $ boundary. Then 
	\begin{equation*}
		\seminorm{\hnab \hv}_{\W{\onehalf - \varepsilon} \left( 0,T; \Lq{q}(\OM) \right)^{n \times n}} 
		\leq C T_0^{\varepsilon} \seminorm{\hv}_{\W{2,1}(\OM \times (0,T))^n},
	\end{equation*}
	for every $ \hv \in \W{2,1}(\OM \times (0,T))^n $, $ \varepsilon \in (0, \onehalf) $ and $ 0 < T < T_0 $. Here $ C $ depends on $ \varepsilon $.
\end{lemma}
\begin{proof}
	The lemma can be easily proved by using the argument in \cite[Lemma 4.2]{Abels2005}, where a layer-like domain with $ C^{1,1} $ boundary is considered. Besides, it can be seen as a corollary of Lemma \ref{embedding: Ws Wr}.
\end{proof}

\subsection{Proof of Theorem \ref{main}}
In this subsection, we prove Theorem \ref{main} by applying the strategy of a fixed-point procedure. 

For the proof, we set $ w = (\hv, \hpi, \hc, \hcss, \hg) $, $ \wo := (\vo, \co, 0, 1) $ and reformulate the initial and boundary value problem \NLFSI\  as an abstract equation: 
\begin{equation}
	\label{abstract}
	\sL (w) = \sN(w, \wo), \quad \textrm{for all}\ w \in \YT,\ (\vo, \co) \in \Dq,
\end{equation}
where
\begin{equation*}
	\sL (w) := 
	\left( 
		\begin{gathered}
		\pt \hv - \hdiv \bS( \hv, \hpi ) \\
		\hdiv \left( \hv \right) - \frac{\gamma \beta}{\hrs} \hcs \\
		\jump{\bS( \hv, \hpi )} \cdot \hng \\
		\bS( \hvs, \hpis ) \cdot \hngs \\
		\pt \hc - \hD \hDelta \hc \\
		\hD \hnab \hc \cdot \hng \\
		\hDs \hnab \hcs \cdot \hngs \\
		\pt \hcss - \beta \hcs \\
		\pt \hg - \frac{\gamma \beta}{n \hrs} \hcs \\
		\rv{(\hv, \hc, \hcss, \hg)^\top}_{t = 0}
		\end{gathered} 
	\right), \quad
	\sN (w, \wo) := 
	\left( 
		\begin{gathered}
			\bK(w) \\
			G(w) \\
			\bH^1(w) \\
			\bH^2(w) \\
			F^1(w) \\
			F^2(w) \\
			F^3(w) \\
			F^4(w) \\
			F^5(w) \\
			\wo
		\end{gathered} 
	\right).
\end{equation*}
In the sequel, we focus on \eqref{abstract}. For $ \sL $, we have the following proposition.
\begin{proposition}
	\label{isomorphism}
	Let $ \sL $ be defined as in \eqref{abstract}. Then $ \sL $ is an isomorphism from $ \YT $ to $ \ZT \times \Dq $.
\end{proposition}
\begin{proof}
	As $ \sL \in \cL (\YT, \ZT \times \Dq) $, it suffices to show that $ \sL $ is bijective, thanks to the bounded inverse theorem.
	
	\textit{Injective.} Take any $ w^1, w^2 \in \YT $. Then, from \eqref{thm3.1:estimates}, \eqref{thm3.2:estimates} and \eqref{estimate: foam cells and growth},
	\begin{equation*}
		\norm{\sL(w^1) - \sL(w^2)}_{\ZT \times \Dq} \leq C \norm{w^1 - w^2}_{\YT},
	\end{equation*}
	which implies the injection of $ \sL $.
	
	\textit{Surjective.} The existence of \eqref{twophase: linear}, \eqref{parabolic: cf}--\eqref{parabolic: cs} and \eqref{foam cells}--\eqref{growth metric} immediately yields the surjection of $ \sL $.
\end{proof}

To employ the contraction mapping principle to \eqref{abstract}, we then investigate the dependence and contraction of $ (\bK, G, \bH^1, \bH^2, F^1, F^2, F^3, F^4, F^5) $ on $ (\hv, \hpi, \hc, \hcss, \hg) $. To this end, we define
\begin{equation*}
	\sM(w) := \left( \bK(w), G(w), \bH^1(w), \bH^2(w), F^1(w), F^2(w), F^3(w), F^4(w), F^5(w) \right)^\top,
\end{equation*}
where the elements are given by \eqref{Nonlinear}. Then it is still needed to show that $ \sM(w) : \YT \rightarrow \ZT $ is well-defined in terms of $ (\hv, \hpi, \hc, \hcss, \hg) \in \YT $ and to verify $ \sM(w) $ possesses the contraction property. 

\begin{proposition}
	\label{M}
	Let $ q > n $ and $ R > 0 $. Assume $ w = (\hv, \hpi, \hc, \hcss, \hg) \in \YT $ with $ \rv{\hg}_{t = 0} = 1 $ and $ \norm{w}_{\YT} \leq R $, then there exist a constant $ C = C(R) > 0 $, a finite time $ T_R > 0 $ depending on $ R $ and $ \delta > 0 $ such that for $ 0 < T < T_R $, $ \sM(w) : \YT \rightarrow \ZT $ is well-defined and bounded along with the estimates:
	\begin{equation}
		\label{Mw}
		\norm{\sM(w)}_{\ZT}
		\leq C(R) \TD \left( \norm{w}_{\YT} + 1 \right).
	\end{equation}
	Moreover, for $ w^1 = (\hv^1, \hpi^1, \hc^1, {{}\hcss}^1, \hg^1), w^2 = (\hv^2, \hpi^2, \hc^2, {{}\hcss}^2, \hg^2) \in \YT $ with $ w^1 \neq w^2 $, $ \rv{\hc^i}_{t = 0} = \co $, $ \rv{\hcss}_{t = 0} = 0 $, $ \rv{\hg^i}_{t = 0} = 1 $ and $ \norm{w^i}_{\YT} \leq R \ (i = 1, 2) $, there exist a constant $ C = C(R) > 0 $, a finite time $ T_R > 0 $ depending on $ R $ and $ \delta > 0 $ such that for $ 0 < T < T_R $,
	\begin{equation}
		\label{MMw}
		\norm{\sM(w^1) - \sM(w^2)}_{\ZT} \leq C(R) \TD \norm{w^1 - w^2}_{\YT}.
	\end{equation}
\end{proposition}
\begin{proof}
	Firstly, we prove the second part. To this end, for $ \norm{w^i}_{\YT} \leq R $, $ i = 1, 2 $ we estimate the following terms respectively
	\begin{gather*}
		\norm{\bK(w^1) - \bK(w^2)}_{\ZT^1}, \quad
		\norm{G(w^1) - G(w^2)}_{\ZT^2}, \\
		\norm{\bH^j(w^1) - \bH^j(w^2)}_{\ZT^{j+2}}, \quad
		\norm{F^k(w^1) - F^k(w^2)}_{\ZT^{k+4}},\quad
		\norm{F^5(w^1) - F^5(w^2)}_{\ZT^{8}},
	\end{gather*} 
	where $ j \in \{ 1,2 \} $, $ k \in \{ 1,2,3,4 \} $. If $ 0 < T \leq 1 $, we have $ T^s < T^{s'} $ for $  s > s' > 0 $. In the sequel, we set a universal constant $ \delta = \min\{ \frac{1}{2q'}, {\frac{1}{q} - \frac{1}{r}} \} $, where $ q' = \frac{q}{q - 1} $, $ r = \frac{q^2}{n} $.
	
	\textbf{Estimate of $ \norm{\bK(w^1) - \bK(w^2)}_{\ZT^1} $.}
	For $ \bK_f = \hdiv \kf $ from \eqref{Nonlinear}, with the help of Lemma \ref{algebra}, \ref{F} and \ref{Fw}, we derive that
	\begin{align*}
		& \norm{\kf(w^1) - \kf(w^2)}_{\Lq{q} \left( 0,T; \W{1}(\Of) \right)^{n \times n}} \\
		& \leq \norm{\hpif^1 \left( \invtr{\hFf}(\hvf^1) - \invtr{\hFf}(\hvf^2) \right)
			+ \left( \hpif^1 - \hpif^2 \right) \left( \invtr{\hFf}(\hvf^2) - \bbI \right)}_{\Lq{q} \left( 0,T; \W{1}(\Of) \right)^{n \times n}} \\
		& \quad + \nuf \norm{\left( \inv{\hFf}(\hvf^1) \hnab \hvf^1 + \tran{\hnab} \hvf^1 \invtr{\hFf}(\hvf^1) \right) 
		\left( \invtr{\hFf}(\hvf^1) - \invtr{\hFf}(\hvf^2) \right)}_{\Lq{q} \left( 0,T; \W{1}(\Of) \right)^{n \times n}} \\
		& \quad + 2 \nuf \norm{\left( \left( \inv{\hFf}(\hvf^1) - \inv{\hFf}(\hvf^2) \right) \hnab \hvf^1 + \inv{\hFf}(\hvf^2) \left( \hnab \hvf^1 - \hnab \hvf^2 \right) \right) \right. \\
			& \left. \qquad \qquad \qquad \qquad \times \left( \invtr{\hFf}(\hvf^2) - \bbI \right)}_{\Lq{q} \left( 0,T; \W{1}(\Of) \right)^{n \times n}} \\
		& \quad + 2 \nuf \norm{\left( \inv{\hFf}(\hvf^1) - \inv{\hFf}(\hvf^2) \hnab \hvf^1 \right)
			\left( \inv{\hFf}(\hvf^2) - \bbI \right) \left( \hnab \hvf^1 - \hnab \hvf^2 \right)}_{\Lq{q} \left( 0,T; \W{1}(\Of) \right)^{n \times n}} \\
		& \leq C \TQ \left( \norm{\hpif^1}_{\YT^2} \norm{\hvf^1 - \hvf^2}_{\YT^1} 
		+ \norm{\hpif^1 - \hpif^2}_{\YT^2} \norm{\hvf^2}_{\YT^1} \right) 
		+ C \TQ \norm{\hvf^1}_{\YT^1} \norm{\hvf^1 - \hvf^2}_{\YT^1} \\
		& \quad + C T^{\frac{2}{q'}} \norm{\hvf^1}_{\YT^1} \norm{\hvf^1 - \hvf^2}_{\YT^1} \norm{\hvf^2}_{\YT^1} + C \TQ \norm{\hvf^1 - \hvf^2}_{\YT^1} \norm{\hvf^2}_{\YT^1} \\
		& \quad + C \TQ \left( \norm{\hvf^1}_{\YT^1} \norm{\hvf^1 - \hvf^2}_{\YT^1} + \norm{\hvf^1 - \hvf^2}_{\YT^1} \norm{\hvf^2}_{\YT^1} \right) 
		\leq C(R) \TD \norm{w^1 - w^2}_{\YT}. 
	\end{align*}
	\\ Let $ \hg \in \W{1}(0,T; \W{1}(\Os)) $ with $ \rv{\hg}_{t = 0} = 1 $. Now we claim that there exists a time $ T_R > 0 $ such that for $ 0 < T < T_R $, $ \hg \geq \onehalf > 0 $. Let $ \hg $ be such function with $ \norm{\hg}_{\W{1}(0,T; \W{1}(\Os))} \leq R $ for some $ R > 0 $. Then for $ 0 < t < T $,
	\begin{equation*}
		\norm{g(t) - 1}_{\Lq{\infty}(\Os)}
		\leq C \norm{g(t) - g(0)}_{\W{1}(\Os)}
		\leq C T^{\frac{1}{q'}} R
		\leq \onehalf,
	\end{equation*}
	where we choose $ T_R > 0 $ small enough such that $ T_R \leq \frac{1}{2 C R} $.
	Hence,
	\begin{equation*}
		\hg \geq \onehalf > 0.
	\end{equation*}
	For $ \bK_s = \hdiv \ks + \bar{\bK}_s^g $, the first part can be estimated similarly that
	\begin{align*}
		\norm{\ks(w^1) - \ks(w^2)}_{\Lq{q} \left( 0,T; \W{1}(\Os) \right)^{n \times n}}
		\leq C(R) \TD \norm{w^1 - w^2}_{\YT}.
	\end{align*}
	The second part follows from \eqref{LI}, Lemma \ref{F} and \ref{Fw} that
	\begin{align*}
		& \norm{\bar{\bK}_s^g(w^1) - \bar{\bK}_s^g(w^2)}_{\Lq{q} \left( 0,T; \Lq{q}(\Os) \right)^{n \times n}} \\
		& \leq \norm{\left( \hsigs(\hvs^1, \hpis^1, \hg^1) \invtr{\hFs}(\hvs^1) - \hsigs(\hvs^2, \hpis^2, \hg^2) \invtr{\hFs}(\hvs^2) \right) \frac{n \hnab \hg^1}{\hg^1}}_{\Lq{q} \left( 0,T; \Lq{q}(\Os) \right)^{n \times n}} \\
		& \qquad + \norm{\hsigs(\hvs^2, \hpis^2, \hg^2) \invtr{\hFs}(\hvs^2) \left( \frac{n \hnab \hg^1}{\hg^1} - \frac{n \hnab \hg^2}{\hg^2} \right)}_{\Lq{q} \left( 0,T; \Lq{q}(\Os) \right)^{n \times n}} 
		=: N_1 + N_2.
	\end{align*}
	From the definition of $ \hsigs $ and $ \hg \geq 1/2 $,
	\begin{align*}
		N_1 \leq C \norm{\hnab \hg^1}_{\Lq{\infty} \left( 0,T; \Lq{q}(\Os) \right)^n} N_1^1 \leq C(R) \TQ N_1^1 ,
	\end{align*}
	where
	\begin{align*}
		N_1^1 & := \norm{\hpis^1 \left( \invtr{\hFs}(\hvs^1) - \invtr{\hFs}(\hvs^2) \right)}_{\Lq{q} \left( 0,T; \Lq{\infty}(\Os) \right)^{n \times n}} \\
		& \ 
		+ \norm{\left( \hpis^1 - \hpis^2 \right) \invtr{\hFs}(\hvs^2)}_{\Lq{q} \left( 0,T; \Lq{\infty}(\Os) \right)^{n \times n}} 
		+ \hnus \norm{\hnab \hvs^1 - \hnab \hvs^2}_{\Lq{q} \left( 0,T; \Lq{\infty}(\Os) \right)^{n \times n}} \\
		& \quad + \hmus \left( \norm{\frac{1}{(\hg^1)^2} \left( \hFs(\hvs^1) - \hFs(\hvs^2) \right)}_{\Lq{q} \left( 0,T; \Lq{\infty}(\Os) \right)^{n \times n}} \right. \\
		& \left. \qquad \qquad \qquad 
		+ \norm{\left(\frac{1}{(\hg^1)^2} - \frac{1}{(\hg^2)^2} \right)\hFs(\hvs^2)}_{\Lq{q} \left( 0,T; \Lq{\infty}(\Os) \right)^{n \times n}} \right. \\
		& \left. \qquad \qquad \qquad \qquad \qquad
		+ \norm{\invtr{\hFs}(\hvs^1) - \invtr{\hFs}(\hvs^2)}_{\Lq{q} \left( 0,T; \Lq{\infty}(\Os) \right)^{n \times n}} \right) \\
		& \leq C \TQ \norm{\hpis^1}_{\YT^2} \norm{\hvs^1 - \hvs^2}_{\YT^1} 
		+ C \norm{\hpis^1 - \hpis^2}_{\YT^2} 
		+ C \norm{\hvs^1 - \hvs^2}_{\YT^1} \\
		& \quad + \hmus \left( C \TQ \norm{\hvs^1 - \hvs^2}_{\YT^1}
		+ C \TQ \norm{\hg^1 - \hg^2}_{\YT^4} \norm{\hvs^2}_{\YT^1} \left(\norm{\hg^1}_{\YT^4} + \norm{\hg^2}_{\YT^4}\right) \right. \\
		& \left. \qquad + C \TQ \norm{\hvs^1 - \hvs^2}_{\YT^1} \right) 
		\leq C(R) \norm{w^1 - w^2}_{\YT}.
	\end{align*}
	Then we get
	\begin{equation*}
		N_1 + N_2 \leq C(R) \TQ \norm{w^1 - w^2}_{\YT}.
	\end{equation*}
	Consequently,
	\begin{equation}
		\label{Kw2}
		\norm{\bK(w^1) - \bK(w^2)}_{\ZT^1} 
		\leq C(R) \TD \norm{w^1 - w^2}_{\YT}.
	\end{equation}
	
	\textbf{Estimate of $ \norm{G(w^1) - G(w^2)}_{\ZT^2} $.}
	From the definition of $ \ZT^2 $ given by \eqref{ZT2}, we need to verify that $ G(w^1) - G(w^2) $ is contained both in $ \Lq{q}( 0,T; \W{1}(\tO) ) $ and $ \W{1}( 0,T; \W{-1}(\hO) ) $, as well as the trace regularity 
	\begin{gather*}
		\tr_{\Gamma}( G(w^1) - G(w^2) ) \in \W{1 - \frac{1}{q},\onehalf(1 - \frac{1}{q})}(\Gamma \times (0,T)), \\
		\tr_{\Gs}( G(w^1) - G(w^2) ) \in \W{1 - \frac{1}{q},\onehalf(1 - \frac{1}{q})}(\Gs \times (0,T)).
	\end{gather*}
	
	The first regularity follows easily from \eqref{Nonlinear}, Lemma \ref{F} and \ref{Fw} that
	\begin{align*}
		& \norm{G(w^1) - G(w^2)}_{\Lq{q}\left( 0,T; \W{1}(\tO) \right)} \\
		& \quad \leq C \TQ \norm{\hv^1}_{\YT^1} \norm{\hv^1 - \hv^2}_{\YT^1} + C \TQ \norm{\hv^2}_{\YT^1} \norm{\hv^1 - \hv^2}_{\YT^1} 
		\leq C \TD R \norm{w^1 - w^2}_{\YT}.
	\end{align*}
	From approximation argument in \cite[Page 15]{AM2018}, we know that weak derivative with respect to time does exist for $ G $. Hence, substituting $ G $ by the form \eqref{G : form}, using integration by parts, we have
	\begin{align*}
		& \inner{\pt G(\cdot, t)}{\phi}_{\W{-1} \times W_{q',0}^1} 
		= \frac{\d}{\d t} \inner{G(\cdot, t)}{\phi}_{\W{-1} \times W_{q',0}^1} \\
		& \qquad \quad = \frac{\d}{\d t} \left( 
		\inner{\left( \FOinv \right) \hv}{\hnab \phi}_{\Lq{q} \times \Lq{q'}}
		- \inner{\hvs \cdot \hdiv \invtr{\hFs}}{\phi}_{\Lq{q} \times \Lq{q'}}
		\right) \\
		& \qquad \qquad \qquad = \int_{\hO} \left( \left( \pt \inv{\hF} \right) \hv + \left( \FOinv \right) \pt \hv \right) \cdot \hnab \phi \d X \\
		& \qquad \qquad \qquad \qquad \quad + \int_{\Os} \left( \pt \hvs \cdot \hdiv \invtr{\hFs} + \hvs \cdot \hdiv \pt \invtr{\hFs} \right) \cdot \phi \d X \\
		& \qquad \qquad \qquad = 
		\int_{\Of} \left( \pt \inv{\hFf} \right) \hvf \cdot \hnab \phi \d X
		+ \int_{\hO} \left( \left( \FOinv \right) \pt \hv \right) \cdot \hnab \phi \d X \\
		& \qquad \qquad \qquad \qquad \quad + \int_{\Os} \left( \pt \hvs \cdot \hdiv \invtr{\hFs} + \pt \invtr{\hFs} : \hnab \hvs \right) \cdot \phi \d X,
	\end{align*}
	for every $ \phi \in W_{q',0}^1(\hO) $, where $ \inner{\cdot}{\cdot}_{X \times X'} $ denotes the duality product between a pair of dual space $ X $ and $ X' $.
	Then according to \eqref{DtA}, the Sobolev embedding $ \W{1}(\hO) \hookrightarrow C^{0,1 - n/q}(\hO) \hookrightarrow \Lq{\infty}(\hO) $ and Lemma \ref{F}, one obtains
	\begin{align*}
		& \norm{\pt G(w^1) - \pt G(w^2)}_{\Lq{q}\left( 0,T; \W{-1}(\hO) \right)} \\
		& \leq \norm{\left( \pt \inv{\hFf}(\hv^1) - \pt \inv{\hFf}(\hvf^2) \right) \hvf^1 
			+ \pt \inv{\hFf}(\hvf^2) \left( \hvf^1 - \hvf^2 \right)}_{\Lq{q}\left( 0,T; \Lq{q}(\hO) \right)^n} \\
		& \quad + \norm{\left( \inv{\hF}(\hv^1) - \inv{\hF}(\hv^2) \right) \pt \hv^1
			+ \left( \inv{\hF}(\hv^2) - \bbI \right) \left( \pt \hv^1 - \pt \hv^2 \right)}_{\Lq{q}\left( 0,T; \Lq{q}(\hO) \right)^n} \\
		& \quad + \norm{\pt \hvs^1 \cdot \left( 
			\hdiv \invtr{\hFs}(\hv^1) - \hdiv \invtr{\hFs}(\hvs^2) \right)}_{\Lq{q}\left( 0,T; \Lq{q}(\Os) \right)} \\
		& \qquad \qquad \qquad + \norm{\left( \pt \hvs^1 - \pt \hvs^2 \right) \cdot \hdiv \invtr{\hFs}(\hvs^2)}_{\Lq{q}\left( 0,T; \Lq{q}(\Os) \right)} \\
		& \quad + \norm{\left( \pt \invtr{\hFs}(\hvs^1) - \pt \invtr{\hFs}(\hvs^2) \right) : \hnab \hvs^1}_{\Lq{q}\left( 0,T; \Lq{q}(\Os) \right)} \\
			& \qquad \qquad \qquad + \norm{\pt \invtr{\hFs}(\hvs^2) : \left( \hnab \hvs^1 - \hnab \hvs^2 \right)}_{\Lq{q}\left( 0,T; \Lq{q}(\Os) \right)} \\
		& \leq C T^{\frac{1}{q} - \frac{1}{r}} \norm{\hv^1 - \hv^2}_{\YT^1} \left( 1 + \TQ \norm{\hv^1}_{\YT^1} \right)
		+ C \TQ \norm{\hv^2}_{\YT^1} \norm{\hv^1 - \hv^2}_{\YT^1} \\
		& \quad + C \TQ \norm{\hv^1 - \hv^2}_{\YT^1} \norm{\hv^1}_{\YT^1}
		+ C \TQ \norm{\hv^2}_{\YT^1} \norm{\hv^1 - \hv^2}_{\YT^1} \\
		& \quad + C \TQ \norm{\hv^1}_{\YT^1} \norm{\hv^1 - \hv^2}_{\YT^1} 
		+ C \TQ \norm{\hv^1 - \hv^2}_{\YT^1} \norm{\hv^2}_{\YT^1} \\
		& \quad + C T^{\frac{1}{q} - \frac{1}{r}} \norm{\hv^1 - \hv^2}_{\YT^1} \norm{\hv^1}_{\YT^1} 
		+ C \TQ \norm{\hv^2}_{\YT^1} \norm{\hv^1 - \hv^2}_{\YT^1} \\
		& \leq C(R) \TD \norm{w^1 - w^2}_{\YT}.
	\end{align*}
	
	Then we are in the position to prove $ \tr_{\Gamma}( G(w^1) - G(w^2) ) \in \W{1 - 1/q, (1 - 1/q)/2}(\Gamma \times (0,T)) $. Recalling the definition of such mixed space \eqref{WSR}, we first write the explicit norm.
	\begin{align*}
		& \norm{\tr_{\Gamma}\left( G(w^1) - G(w^2) \right)}_{\W{1 - \frac{1}{q},\onehalf\left(1 - \frac{1}{q}\right)}(\Gamma \times (0,T))} \\
		& \quad = \norm{\left( \invtr{\hF}(\hv^1) - \invtr{\hF}(\hv^2) \right) : \hnab \hv^1}_{\Lq{q} \big( 0,T; \W{1-\frac{1}{q}}(\Gamma) \big)} \\
		& \qquad \qquad + \norm{\left( \invtr{\hF}(\hv^2) - \bbI \right): \left( \hnab \hv^1 - \hnab \hv^2 \right)}_{\Lq{q} \big( 0,T; \W{1-\frac{1}{q}}(\Gamma) \big)} \\
		& \qquad + \norm{\left( \invtr{\hF}(\hv^1) - \invtr{\hF}(\hv^2) \right) : \hnab \hv^1}_{ \W{\onehalf \left( 1-\frac{1}{q} \right)} \left( 0,T; \Lq{q} (\Gamma) \right)} \\
		& \qquad \qquad + \norm{\left( \invtr{\hF}(\hv^2) - \bbI \right): \left( \hnab \hv^1 - \hnab \hv^2 \right)}_{ \W{\onehalf \left( 1-\frac{1}{q} \right)} \left( 0,T; \Lq{q} (\Gamma) \right)} 
		=: \sum_{i = 1}^4 I_i.
	\end{align*}
	According to the trace theorem from $ \W{1}(\tO) $ into $ \W{1-\frac{1}{q}}(\Gamma) $, Lemma \ref{algebra}, \ref{F} and \ref{Fw}, 
	\begin{align*}
		& I_1 
		\leq C \norm{\left( \invtr{\hF}(\hv^1) - \invtr{\hF}(\hv^2) \right) : \hnab \hv^1}_{\Lq{q} \left( 0,T; \W{1}(\tO) \right)} 
		\leq C(R) \TD \norm{w^1 - w^2}_{\YT}, \\
		& I_2
		\leq C \norm{\left( \invtr{\hF}(\hv^2) - \bbI \right): \left( \hnab \hv^1 - \hnab \hv^2 \right)}_{\Lq{q} \left( 0,T; \W{1}(\tO) \right)} 
		\leq C(R) \TD \norm{w^1 - w^2}_{\YT}.
	\end{align*}
	It follows from the definition of vector valued Sobolev-Slobodeckij space, Lemma \ref{F} and \ref{gradu} that
	\begin{align*}
		I_3 
		& \leq \left( \int_{0}^{T} \int_{0}^{t} \frac{\norm{ \Dh \left( \invtr{\hF}(\hv^1) - \invtr{\hF}(\hv^2) \right)(t) : \hnab \hv^1(t-h) }_{\Lq{q}(\Gamma)}^q}{h^{1 + \frac{q}{2}\left( 1 - \frac{1}{q} \right)}} \d h \d t\right)^{\frac{1}{q}} \\
		& \qquad + \left( \int_{0}^{T} \int_{0}^{t} \frac{\norm{ \left( \invtr{\hF}(\hv^1) - \invtr{\hF}(\hv^2) \right)(t) : \Dh \left( \hnab \hv^1 \right)(t) }_{\Lq{q}(\Gamma)}^q}{h^{1 + \frac{q}{2}\left( 1 - \frac{1}{q} \right)}} \d h \d t\right)^{\frac{1}{q}} \\
		& \leq \sup_{0 \leq t \leq T} \left( \int_{0}^{t} \frac{\norm{ \Dh \left( \invtr{\hF}(\hv^1) - \invtr{\hF}(\hv^2) \right) }_{\Lq{\infty}(\Gamma)^{n \times n}}^q}{h^{1 + \frac{q}{2}\left( 1 - \frac{1}{q} \right)}} \d h \right)^\frac{1}{q} \norm{\hnab \hv^1}_{\Lq{q}(0,T; \Lq{q}(\Gamma))^{n \times n}} \\
		& \qquad + \sup_{0 \leq t \leq T} \norm{ \invtr{\hF}(\hv^1) - \invtr{\hF}(\hv^2) }_{\W{1}(\tO)^{n \times n}} \seminorm{\hnab \hv^1}_{\W{\onehalf \left( 1-\frac{1}{q} \right)} \left( 0,T; \Lq{q} (\Gamma) \right)^{n \times n}} \\
		& \leq C \left( \TQQ \norm{\hv^1}_{\YT^1} \norm{\hv^1 - \hv^2}_{\YT^1} + \TQ \norm{\hv^1 - \hv^2}_{\YT^1} \norm{\hv^1}_{\YT^1} \right) \\
		& \leq C(R) \TD \norm{w^1 - w^2}_{\YT},
	\end{align*}
	where we used the property of $ \Dh $ that $ \Dh(fg)(t) = \Dh f(t) g(t-h) + f(t) \Dh g(t) $. Similarly,
	\begin{align*}
		I_4 \leq C(R) \TD \norm{w^1 - w^2}_{\YT},
	\end{align*}
	Collecting $ I_i $, $ i = 1,\dots, 4 $, we get
	\begin{equation*}
		\norm{\tr_{\Gamma}\left( G(w^1) - G(w^2) \right)}_{\W{1 - \frac{1}{q},\onehalf\left(1 - \frac{1}{q}\right)}(\Gamma \times (0,T))}
		\leq C(R) \TD \norm{w^1 - w^2}_{\YT}.
	\end{equation*}
	Since the trace regularities for $ G $ on $ \Gamma $ and $ \Gs $ are same, one also obtains
	\begin{equation*}
		\norm{\tr_{\Gs}\left( G(w^1) - G(w^2) \right)}_{\W{1 - \frac{1}{q},\onehalf\left(1 - \frac{1}{q}\right)}(\Gs \times (0,T))}
		\leq C(R) \TD \norm{w^1 - w^2}_{\YT}.
	\end{equation*}
	Then 
	\begin{equation*}
		\norm{G(w^1) - G(w^2)}_{\ZT^2}
		\leq C(R) \TD \norm{w^1 - w^2}_{\YT},
	\end{equation*}
	
	\textbf{Estimate of $ \norm{\bH^1(w^1) - \bH^1(w^2)}_{\ZT^3} $, $ \norm{\bH^2(w^1) - \bH^2(w^2)}_{\ZT^4} $.}
	Since $ \Gamma $ is of class $ C^3 $, $ \hng \in C^2 (\partial \Of) $. Then by similar estimates as $ \tr_{\Gamma}( G(w^1) - G(w^2) ) $, the norm of $ \bH^1(w^1) - \bH^1(w^2) $ in $ \ZT^3 $ reads
	\begin{align*}
		& \norm{\bH^1(w^1) - \bH^1(w^2)}_{\ZT^3} = \norm{\jump{\tk(w^1) - \tk(w^2)} \hng}_{ \W{\onehalf \left( 1-\frac{1}{q} \right)} \left( 0,T; \Lq{q} (\Gamma) \right)^n} \\
		& \qquad\qquad\qquad\qquad\qquad\qquad + \norm{\jump{\tk(w^1) - \tk(w^2)} \hng}_{\Lq{q} ( 0,T; \W{1-\frac{1}{q}}(\Gamma) )^n} \\
		& \qquad \leq C \norm{\kf(w^1) - \kf(w^2)
			+ \ks(w^1) - \ks(w^2)}_{\W{\onehalf \left( 1-\frac{1}{q} \right)} \left( 0,T; \Lq{q} (\Gamma) \right)^{n \times n}} \\
		& \qquad \qquad + C \norm{\left( \tk(w^1) - \tk(w^2) \right)}_{\Lq{q}\left( 0,T; \W{1}(\tO) \right)^{n \times n}} 
		\leq  C(R) \TD \norm{w^1 - w^2}_{\YT}.
	\end{align*}
	As the similar situation, we can easily derive
	\begin{equation*}
		\norm{\bH^2(w^1) - \bH^2(w^2)}_{\ZT^4}
		\leq C \TD \left( 1 + R \right)^2 \norm{w^1 - w^2}_{\YT}.
	\end{equation*}
	
	\textbf{Estimate of $ \norm{F^1(w^1) - F^1(w^2)}_{\ZT^5} $.}
	For $ F_f^1 = \hdiv \tFf $, we have
	\begin{align*}
		& \norm{F_f^1(w^1) - F_f^1(w^2)}_{\ZT^5} \\
		& \quad \leq \norm{\tFf(w^1) - \tFf(w^2)}_{\Lq{q}( 0,T; \W{1}(\Of))^n} \\
		& \quad \leq \hDf \norm{\left( \inv{\hFf}(\hvfone) \invtr{\hFf}(\hvfone) - \inv{\hFf}(\hvftwo) \invtr{\hFf}(\hvftwo) \right) \hnab \hcf^1}_{\Lq{q}( 0,T; \W{1}(\Of))^n} \\
		& \quad \qquad + \hDf \norm{\left( \inv{\hFf}(\hvftwo) \invtr{\hFf}(\hvftwo) - \bbI \right) \left( \hnab \hcf^1 - \hnab \hcf^2 \right)}_{\Lq{q}( 0,T; \W{1}(\Of))^n} =: \mathfrak{F}_1 + \mathfrak{F}_2.
	\end{align*}
	Lemma \ref{F} and multiplication of $ \W{1}(\Of) $ in Lemma \ref{algebra} imply that
	\begin{align*}
		\mathfrak{F}_1 & \leq C \left( 
			\norm{\inv{\hFf}(\hvfone) \left( \invtr{\hFf}(\hvfone) - \invtr{\hFf}(\hvftwo) \right)}_{\Lq{\infty}( 0,T; \W{1}(\Of))^{n \times n}} \right. \\
			& \left. \qquad + \norm{\left( \inv{\hFf}(\hvfone) - \inv{\hFf}(\hvftwo) \right) \invtr{\hFf}(\hvftwo)}_{\Lq{\infty}( 0,T; \W{1}(\Of))^{n \times n}}
		\right) \norm{\hnab \hcf^1}_{\Lq{q}( 0,T; \W{1}(\Of))^n} \\
		& \leq C(R) \TQ \norm{w^1 - w^2}_{\YT},
	\end{align*}
	and
	\begin{align*}
		\mathfrak{F}_2 & \leq C \left( 
			\norm{\inv{\hFf}(\hvftwo) \left( \invtr{\hFf}(\hvftwo) - \bbI \right)}_{\Lq{\infty}( 0,T; \W{1}(\Of))^{n \times n}} \right. \\
			& \left. \qquad + \norm{\inv{\hFf}(\hvftwo) - \bbI}_{\Lq{\infty}( 0,T; \W{1}(\Of))^{n \times n}}
		\right) \norm{\hnab \hcf^1 - \hnab \hcf^2}_{\Lq{q}( 0,T; \W{1}(\Of))^n} \\
		& \leq C(R) \TQ \norm{w^1 - w^2}_{\YT}.
	\end{align*}
	Then
	\begin{equation*}
		\norm{F_f^1(w^1) - F_f^1(w^2)}_{\ZT^5}
		\leq C(R) \TQ \norm{w^1 - w^2}_{\YT}.
	\end{equation*}
	For $ F_s^1 = \bar{F}_s^1 + F_s^g = \hdiv \tFs + F_s^g $, it can be deduced similarly as $ F_f^1 $ that 
	\begin{align*}
		\norm{\bar{F}_s^1(w^1) - \bar{F}_s^1(w^2)}_{\ZT^5}
		\leq C(R) \TQ \norm{w^1 - w^2}_{\YT}.
	\end{align*}
	Moreover,
	\begin{align*}
		& \norm{F_s^g(w^1) - F_s^g(w^2)}_{\ZT^5} \\
		& \leq \beta \norm{\left( \hcs^1 - \hcs^2 \right) \left( 1 + \frac{\gamma}{\hrs} \hcs^1 \right)}_{\Lq{q}(\Os \times (0,T))} 
		+ \beta \norm{\hcs^2 \left( \hcs^1 - \hcs^2 \right)}_{\Lq{q}(\Os \times (0,T))} \\
		& \quad + n \norm{\frac{\hnab \hg^1}{\hg^1} \left( \tFs(w^1) - \tFs(w^2) + \left( \hnab \hcs^1 - \hnab \hcs^2 \right)\right)}_{\Lq{q}(\Os \times (0,T))} \\
		& \quad + n \norm{\left( \frac{\hnab \hg^1}{\hg^1} - \frac{\hnab \hg^2}{\hg^2} \right) \inv{\hFs}(\hvs^2) \invtr{\hFs}(\hvs^2) \hnab \hcs^2}_{\Lq{q}(\Os \times (0,T))} 
		=: \sum_{i = 1}^4 \mathfrak{F}_i^g.
	\end{align*}
	Apparently, with $ \rv{\hc^i}_{t = 0} = \co $, $ i = 1, 2 $,
	\begin{align*}
		\mathfrak{F}_1^g + \mathfrak{F}_2^g
		& \leq C \norm{\hcs^1 - \hcs^2}_{\Lq{\infty}(0,T; \Lq{q}(\Os))} \norm{1 + \frac{\gamma}{\hrs} \hcs^1}_{\Lq{q}(0,T; \Lq{\infty}(\Os))} \\
		& \quad + C \norm{\hcs^1 - \hcs^2}_{\Lq{\infty}(0,T; \Lq{q}(\Os))} \norm{\hcs^2}_{\Lq{q}(0,T; \Lq{\infty}(\Os))} 
		\leq C(R) \TQ \norm{w^1 - w^2}_{\YT}.
	\end{align*}
	Proceeding the same estimate as $ \tFf $ above, we have
	\begin{align*}
		\mathfrak{F}_3^g + \mathfrak{F}_4^g
		\leq C(R) \TQ \norm{w^1 - w^2}_{\YT},
	\end{align*}
	by $ \hg \geq \onehalf $ and Lemma \ref{F}.
	Collecting $ \mathfrak{F}_i^g $, $ i = 1,...,4 $ together, one concludes
	\begin{equation*}
		\norm{F_s^1(w^1) - F_s^1(w^2)}_{\ZT^5}
		\leq C(R) \TQ \norm{w^1 - w^2}_{\YT},
	\end{equation*}
	which derives the desired regularity.
	
	\textbf{Estimate of $ \norm{F^2(w^1) - F^2(w^2)}_{\ZT^6} $, $ \norm{F^3(w^1) - F^3(w^2)}_{\ZT^7} $.}
	Since the key ingredient here is to estimate $ \tF(w^1) - \tF(w^2) $ in space $ \W{1 - \frac{1}{q}, \onehalf ( 1 - \frac{1}{q} )}(\Gamma \times (0,T)) $, we only give the details to handle this term. By definition,
	\begin{align*}
		& \norm{\tF(w^1) - \tF(w^2)}_{\W{1 - \frac{1}{q}, \onehalf ( 1 - \frac{1}{q} )}(\Gamma \times (0,T))^n} \\
		& = \norm{\tF(w^1) - \tF(w^2)}_{\Lq{q}(0,T; \W{1 - \frac{1}{q}}(\Gamma))^n}
		+ \norm{\tF(w^1) - \tF(w^2)}_{\W{\onehalf \left( 1 - \frac{1}{q} \right)}(0,T; \Lq{q}(\Gamma))^n}.
	\end{align*}
	The first term can be controlled easily by trace Theorem for $ q > n $ and the estimates of $ \tF $ in $ \tO $ above. Namely,
	\begin{align*}
		& \norm{\tF(w^1) - \tF(w^2)}_{\Lq{q}(0,T; \W{1 - \frac{1}{q}}(\Gamma))^n} \\
		& \qquad \qquad \leq C \norm{\tF(w^1) - \tF(w^2)}_{\Lq{q}(0,T; \W{1}(\tO))^n}
		\leq C(R) \TQ \norm{w^1 - w^2}_{\YT}.
	\end{align*}
	For the second term, again by the definition of vector-valued Sobolev-Slobodeckij space, we have
	\begin{align*}
		\norm{\tF(w^1) - \tF(w^2)}_{\W{\onehalf \left( 1 - \frac{1}{q} \right)}(0,T; \Lq{q}(\Gamma))^n} 
		\leq C(R) \TQQ \norm{w^1 - w^2}_{\YT},
	\end{align*}
	following the argument of estimating $ \tr_{\Gamma}( G(w^1) - G(w^2) ) $. Then,
	\begin{equation*}
		 \norm{F^2(w^1) - F^2(w^2)}_{\ZT^6} + \norm{F^3(w^1) - F^3(w^2)}_{\ZT^7}
		 \leq C(R) \TD \norm{w^1 - w^2}_{\YT}.
	\end{equation*}
	
	\textbf{Estimate of $ \norm{F^4(w^1) - F^4(w^2)}_{\ZT^8} $, $ \norm{F^5(w^1) - F^5(w^2)}_{\ZT^8} $.} Observing that the nonlinearities in $ F^4 $ and $ F^5 $ are $ \hcs \hcss $ and $ \hcs \hg $, which are all quadratic, we control them under the assumptions $ \rv{\hc^i}_{t = 0} = \co $, $ \rv{\hcss}_{t = 0} = 0 $, $ \rv{\hg^i}_{t = 0} = 1 $, $ i = 1, 2 $, and by 
	\begin{align*}
		& \norm{uv}_{\Lq{q}(0,T; \W{1}(\Os))}
		\leq M_q \norm{u}_{\Lq{\infty}(0,T; \W{1}(\Os))} \norm{v}_{\Lq{q}(0,T; \W{1}(\Os))},
	\end{align*}
	for $ u,v \in \W{1}(0,T; \W{1}(\Os)) $.
	Hence,
	\begin{align*}
		\norm{F^4(w^1) - F^4(w^2)}_{\ZT^8} + \norm{F^5(w^1) - F^5(w^2)}_{\ZT^8}
		\leq C(R) \TQ \norm{w^1 - w^2}_{\YT}.
	\end{align*}
	
	Consequently, we derive \eqref{MMw}. Now, let $ w^1 = w $ and $ w^2 = (0,0,0,0,1) $ in \eqref{MMw}, \eqref{Mw} follows immediately from the fact that $ \sM(0,0,0,0,1) = 0 $.
\end{proof}

\begin{proof}[\textbf{Proof of Theorem \ref{main}}]
	Since $ \sL: \YT \rightarrow \ZT \times \Dq $ is an isomorphism as showed in Proposition \ref{isomorphism}, from the estimates in Theorem \ref{twophase: theorem}, we set a well-defined constant 
	\begin{equation*}
		C_{\sL} := \sup_{0 \leq T \leq 1} \norm{\sL^{-1}}_{\cL(\ZT \times \Dq , \YT)}.
	\end{equation*}
	Choose $ R > 0 $ large such that $ R \geq 2 C_\sL \norm{(\vo, \co)}_{\Dq} $. Then 
	\begin{equation}
		\label{L0}
		\norm{\sL^{-1} \sN(0, \wo)}_{\YT} 
		\leq C_\sL \norm{(\vo, \co)}_{\Dq}
		\leq \frac{R}{2}.
	\end{equation}
	For $ \norm{w^i}_{\YT} \leq R $, $ i = 1, 2 $, we take $ T_R > 0 $ small enough such that 
	\begin{align*}
		C_\sL C(R) \TD_R \leq \onehalf,
	\end{align*}
	where $ C(R) $ is the constant in \eqref{MMw}. Then for $ 0 < T < T_R $, we infer from Theorem \ref{M} that
	\begin{equation}
		\begin{aligned}
			\label{L12}
			& \norm{\sL^{-1}\sN(w^1, w_0) - \sL^{-1}\sN(w^2, w_0)}_{\YT}  \\
			& \qquad \qquad \leq C_\sL C(R) \TD \norm{w^1 - w^2}_{\YT} 
			\leq \onehalf \norm{w^1 - w^2}_{\YT},
		\end{aligned}
	\end{equation}
	which implies the contraction property.
	From \eqref{L0} and \eqref{L12}, we have
	\begin{align*}
		& \norm{\sL^{-1}\sN(w, \wo)}_{\YT} \\
		& \qquad \qquad \leq \norm{\sL^{-1}\sN(0, \wo)}_{\YT} + \norm{\sL^{-1}\sN(w, \wo) - \sL^{-1}\sN(0, \wo)}_{\YT}
		\leq R.
	\end{align*}
	Define $ \cM_{R,T} $ by
	\begin{equation*}
		\cM_{R,T} := \left\{ w \in \overline{B_{\YT}(0,R)}: w = (\hv, \hpi, \hc, \hcs, \hg), \quad \rv{\hg}_{t = 0} = 1, \quad \rv{\hc}_{t = 0} = \co \right\},
	\end{equation*}
	a closed subset of $ \YT $. Hence, $ \sL^{-1}\sN : \cM_{R,T} \rightarrow \cM_{R,T} $ is well-defined for all $ 0 < T < T_R $ and a strict contraction. Since $ \YT $ is a Banach space, the Banach fixed-point Theorem implies the existence of a unique fixed-point of $ \sL^{-1}\sN $ in $ \cM_{R,T} $, i.e., \NLFSI\  admits a unique strong solution in $ \cM_{R,T} $ for small time $ 0 < T < T_R $.
	
	In the following, we prove the uniqueness of solutions in $ \YT $ by the continuity argument. Let $ w^1, w^2 \in \YT $ be two different solutions of \NLFSI\ and $ \widetilde{R} := \max\{ \norm{w^1}_{\YT}, \norm{w^2}_{\YT} \} $, then there is a time $ T_{\widetilde{R}} \leq T $ such that $ \sL^{-1}\sN : \cM_{\widetilde{R},T_{\widetilde{R}}} \rightarrow \cM_{\widetilde{R},T_{\widetilde{R}}} $ is a contraction and therefore $ \rv{w^1}_{[0,T_{\widetilde{R}}]} = \rv{w^2}_{[0,T_{\widetilde{R}}]} $. Now we argue by contradiction. Define $ \widetilde{T} $ as
	\begin{equation*}
		 \widetilde{T} := 
		 \sup \left\{
		 	T' \in (0,T]: \rv{w^1}_{[0,T']} = \rv{w^2}_{[0,T']}
		 \right\}.
	\end{equation*}
	Assume $ \widetilde{T} < T $. Since $ \rv{w^1}_{[0,\widetilde{T}]} = \rv{w^2}_{[0,\widetilde{T}]} $, we consider $ \rv{w^1}_{t = \widetilde{T}} = \rv{w^2}_{t = \widetilde{T}} $ as the initial value for \NLFSI. Repeating the argument above, we see that there is a time $ \widehat{T} \in (\widetilde{T}, T) $ such that $ \rv{w^1}_{[\widetilde{T},\widehat{T}]} = \rv{w^2}_{[\widetilde{T},\widehat{T}]} $, which contradicts to the definition of $ \widetilde{T} $.
	
	In conclusion, \NLFSI\ admits a unique solution in $ \YT $.
	
	For the nonnegativity of $ \hc $, we show it in Eulerian coordinate. Let $ U_T = (\OM^t \backslash \Gamma^t) \times (0,T) $, $ U_{f,T} = \Oft \times (0,T) $, $ U_{s,T} = \Ost \times (0,T) $, define the parabolic boundary $ \partial_P U_{f,T} := (\Bar{\OM}_f^0 \times \{0\}) \cup (\Gt \times [0,T]) $, $ \partial_P U_{s,T} := (\Bar{\OM}_s^0 \times \{0\}) \cup ((\Gt \cup \Gst) \times [0,T]) $ and $ \partial_P U_T := \partial_P U_{f,T} \cup \partial_P U_{s,T} $. First of all, we claim that $ c \in C^{2,1}_{loc}(U_T) \cap C(\Bar{U_T}) $, where
	\begin{equation*}
		C^{2s,s}(U_T) := \left\{
			c(\cdot, t) \in C^{2s}(\OM^t \backslash \Gamma^t),
			\ 
			c(x, \cdot) \in C^s(0,T), \ \forall x \in \OM^t \backslash \Gamma^t, t \in (0,T)
		\right\},
	\end{equation*}
	for $ s > 0 $. As shown above, assume $ c \in \YT^3 $ is the solution of 
	\begin{equation}
		\label{concentration-Eulerian}
		\pt c - D \Delta c = - (\bv \cdot \nabla c + (\Div \bv + \beta) c) =: f.
	\end{equation}
	With the regularity of $ \bv, c $ and the embedding theorem, we know that $ f \in C_{loc}^{\alpha, \alpha/2}(U_T) $ for some $ 0 < \alpha < 1 $. By the local regularity theory for parabolic equations, one obtains
	\begin{equation*}
		c \in C^{2 + \alpha,1 + \frac{\alpha}{2}}_{loc}(U_T) \hookrightarrow C^{2,1}_{loc}(U_T).
	\end{equation*}
	The continuous of $ c $ can be derived directly from the Lemma \ref{embedding}, especially \eqref{embedding: W21 W1} with 
	\begin{equation*}
		\W{1} \hookrightarrow C^{1 - \frac{n}{q}} \hookrightarrow C^0, \text{ for } q > n.
	\end{equation*}
	
	Now, given a nonnegative initial value $ c^0(x) \geq 0 $, $ x \in \OM^0 $. Define $ c_\lambda := e^{- \lambda t} c $ where $ \lambda > 0 $ is a constant which will be assigned later. Adding $ c c_\lambda $ to the both sides of \eqref{concentration-Eulerian}, we have the equation for $ c_\lambda $
	\begin{equation*}
		\pt c_\lambda - D \Delta c_\lambda + \bv \cdot \nabla c_\lambda + (\Div \bv + c + \beta + \lambda) c_\lambda  = c^2 e^{- \lambda t} \geq 0.
	\end{equation*}
	Taking $ \lambda $ sufficiently large such that
	\begin{equation*}
		\beta + \lambda \geq \sup_{0 \leq t \leq T, x \in \OM^t \backslash \Sigma^t} \abs{\Div \bv} + \abs{c},
	\end{equation*}
	one obtains
	\begin{equation*}
		\Div \bv + c + \beta + \lambda \geq 0.
	\end{equation*}
	By the weak maximum principle for parabolic equations, we have
	\begin{gather*}
		\min_{ \Bar{U}_{f,T} }
		c_f(x, t) \geq 
		- \max_{ \partial_P U_{f,T} }
		c_f^-(x, t), \quad
		\min_{ \Bar{U}_{s,T} }
		c_s(x, t) \geq
		- \max_{  \partial_P U_{s,T} }
		c_s^-(x, \tau),
	\end{gather*}
	namely,
	\begin{equation*}
		\min_{ \Bar{U}_T }
		c(x, t) \geq 
		- \max_{ \partial_P U_T }
		c^-(x, t),
	\end{equation*}
	where $ c^-(x, t) := - \min\{ c(x, t), 0\} $.
	
	Since $ c^0(x) \geq 0 $, now we claim that $ c(x,t) \geq 0 $ for all $ (x,t) \in (\Gt \cup \Gst) \times [0,T] $. To this end, we argue by contradiction. Assume that for some $ t_0 \in (0, T] $, there exist a point $ x_0 \in \Gamma^{t_0} \cup \Gs^{t_0} $, such that
	\begin{equation*}
		c(x_0, t_0) = - \max_{ x \in \Gamma^{t_0} \cup \Gs^{t_0} } c^-(x, t_0) < 0,
	\end{equation*}
	that is, 
	\begin{equation*}
		\min_{ x \in \Gamma^{t_0} \cup \Gs^{t_0} } \min\{ c(x, t_0), 0\} < 0,
	\end{equation*}
	which implies that $ x \mapsto \min\{ c(x, t_0), 0\} $ attains a negative minimum at $ x_0 $, i.e., $ x \mapsto c(x, t_0) $ attains a negative minimum at $ x_0 $. 
	
	
	\textbf{Case 1: $ x_0 \in \Gamma^{t_0} $.} For both $ \Of^{t_0} $ and $ \Os^{t_0} $, since $ \Gamma^{t_0} $ is assumed to be a $ C^{3-} $ interface, we infer from Hopf's Lemma that
	\begin{equation*}
		\Df \nabla c_f \cdot \bn_{\Gamma^{t_0}}(x_0) < 0, \quad
		\Ds \nabla c_s \cdot \bn_{\Gamma^{t_0}}(x_0) > 0, \text{ on } \Gamma^{t_0}.
	\end{equation*}
	Hence,
	\begin{equation*}
		\jump{D \nabla c} \cdot \bn_{\Gamma^{t_0}}(x_0) < 0,
	\end{equation*}
	which contradicts to \eqref{TC1}.
	
	\textbf{Case 2: $ x_0 \in \Gs^{t_0} $.} Again by Hopf's Lemma, one obtains 
	\begin{equation*}
		D \nabla c \cdot \bn_{\Gs^{t_0}}(x_0) < 0, \text{ on } \Gs^{t_0},
	\end{equation*}
	which contradicts to \eqref{boundsc}.
	
	In summary, $ c(x,t) \geq 0 $ for all $ (x,t) \in \overline{\OM}^t \times [0,T] $.
	
	For $ \hcss $ and $ \hg $, notice that the equations for them in Lagrangian coordinate are ordinary differential equations with suitable $ \hcs \geq 0 $. Then
	\begin{gather*}
		\hcss = \int_{0}^t e^{\int_{t}^\sigma \frac{\gamma \beta}{\hrs} \hcs(x,\tau) \d \tau} \beta \hcs(x,\sigma) \d \sigma > 0, \quad
		\hg = e^{\int_{0}^t \frac{\gamma \beta}{n \hrs} \hcs(x,\tau) \d \tau} > 0,
	\end{gather*}
	which completes the proof.
\end{proof}

\appendix
\section{Some results on linear systems}\label{results: linear}
In this section, we give several maximal $ L^q $-regularity results of different problems, which are needed for the whole system. 
\subsection{Two-phase stokes problems with Dirichlet boundary condition} \label{twophaseD}
In this section, we focus on the following nonstationary two-phase Stokes problem.
\begin{equation}
	\label{twophase: Dirichlet boundary}
	\begin{aligned}
		\vr \pt v - \Div (2 \mu D v) + \nabla p & = \vr f_u, && \quad \text{in}\  \OM \backslash \Sigma \times (0, T), \\
		\Div v & = g_d, && \quad \text{in}\  \OM \backslash \Sigma \times (0, T), \\
		v & = g_b, && \quad \text{on}\  \partial \OM \times (0, T), \\
		\jump{v} & = g_u, && \quad \text{on}\  \Sigma \times (0, T), \\
		\jump{- 2 \mu D v + p \bbI} \nusigma & = g, && \quad \text{on}\  \Sigma \times (0, T), \\
		\rv{v}_{t = 0} & = v_0, && \quad \text{in}\  \OM,
	\end{aligned}
\end{equation}
where $ \OM \subset \bbr^n $, $ n \geq 2 $, is a bounded domain with $ \partial \OM \in C^3 $, $ \Sigma \subset \OM $ a closed hypersurface of class $ C^3 $. $ \vr_j $ are positive constants, $ j = 1,2 $. $ v: \OM \times (0,T) \rightarrow \bbr^n $ is the velocity of the fluid, $ p: \OM \times (0,T) \rightarrow \bbr $ denotes the pressure. $ \mu > 0 $ is the constant viscosity and $ Dv = \onehalf \left( \nabla v + \nabla v^\top \right) $. $ \nusigma $ represents the unit outer normal vector on $ \Sigma $. $ f_u, g_d, g_b, g_u, g $ are given functions and $ v_0 $ is the prescribed initial value. System \eqref{twophase: Dirichlet boundary} was been investigated by many scholars in various aspects. We refer for the maximal $ L_q $ regularity results of such kind of two-phase Stokes problem to Pr\"{u}ss and Simonett \cite{PS2016}. Readers can also find similar results in Abels and Moser \cite{AM2018} for $ (g_b, g_u) = 0 $.
\begin{proposition}
	\label{twophaseD: proposition}
	Let $ q > n + 2 $, $ \OM \subset \bbr^n $ be a bounded domain with $ \partial \OM \in C^3 $, $ \Sigma \subset \OM $ a closed hypersurface of class $ C^3 $. Assume that $ (f_u, g_d, g_b, g_u, g) \in \ZT $ where
	\begin{equation*}
		\ZT := \left\{
			\begin{aligned}
				& f_u \in \Lq{q}\left( 0,T; \Lq{q}(\OM) \right)^n, \quad 
				g_d \in \Lq{q}\left( 0,T; \W{1}(\OM \backslash \Sigma) \right), \\
				& g_b \in \W{2 - \frac{1}{q}, 1 - \frac{1}{2q}}(\partial \OM \times (0,T))^n, \quad
				g_u \in \W{2 - \frac{1}{q}, 1 - \frac{1}{2q}}(\Sigma \times (0,T))^n,  \\
				& g \in \W{1 - \frac{1}{q}, \onehalf \left( 1 - \frac{1}{q} \right)}(\Sigma \times (0,T))^n : 
				(g_d, g_b \cdot \nu_{\partial \OM}, g_u \cdot \nusigma) \in \W{1}\left( 0,T; \widehat{W}_{q}^{-1}(\OM) \right)
			\end{aligned}
		\right\}
	\end{equation*}
	and $ v_0 \in \W{2 - \frac{2}{q}}(\OM \backslash \Sigma)^n $ satisfying the compatibility conditions
	\begin{equation}
		\label{twophaseD: compatibility}
		\Div v_0 = \rv{g_d}_{t = 0}, \  \rv{v_0}_{\partial \OM} = \rv{g_b}_{t = 0}, \  \rv{\jump{v_0}}_{\Sigma} = \rv{g_u}_{t = 0}, \ \rv{\jump{\left( 2 \mu D v_0 \nusigma \right)_{\tau}}}_{\Sigma} = \rv{g_{\tau}}_{t = 0}.
	\end{equation}
	Then two-phase Stokes problem \eqref{twophase: Dirichlet boundary} admits a unique solution $ (v, p) $ with regularity
	\begin{gather*}
		v \in \Lq{q}\left( 0,T; \W{2}(\OM \backslash \Sigma) \right)^n \cap \W{1}\left( 0,T; \Lq{q}(\OM) \right)^n, \\
		p \in \Lq{q}\left( 0,T; \WA{1}(\OM \backslash \Sigma) \right), \quad \jump{p} \in \W{1 - \frac{1}{q}, \onehalf \left( 1 - \frac{1}{q} \right)}(\Sigma \times (0,T)).
	\end{gather*}
	Moreover, for any fixed $ 0 < T_0 < \infty $, there is a constant $ C $, independent of $ T \in (0,T_0] $, such that
	\begin{align}
		& \norm{v}_{\Lq{q}\left( 0,T; \W{2}(\OM \backslash \Sigma) \right)^n} 
		+ \norm{v}_{\W{1}\left( 0,T; \Lq{q}(\OM) \right)^n} \nonumber \\ 
		& \qquad \quad + \norm{p}_{\Lq{q}\left( 0,T; \WA{1}(\OM \backslash \Sigma) \right)} 
		+ \norm{\jump{p}}_{\W{1 - \frac{1}{q}, \onehalf \left( 1 - \frac{1}{q} \right)}(\Sigma \times (0,T))} \nonumber \\
		& \quad \leq C \left( 
			\norm{f_u}_{\Lq{q}\left( 0,T; \Lq{q}(\OM) \right)^n} + \norm{g_d}_{\Lq{q}\left( 0,T; \W{1}(\OM \backslash \Sigma) \right)} + \norm{g_b}_{\W{2 - \frac{1}{q}, 1 - \frac{1}{2q}}(\partial \OM \times (0,T))^n} \right. \label{propA.2: estimates} \\
			& \left. \quad \quad \quad + \norm{g_u}_{\W{2 - \frac{1}{q}, 1 - \frac{1}{2q}}(\Sigma \times (0,T))^n} + \norm{\pt (g_d, g_b \cdot \nu_{\partial \OM}, g_u \cdot \nusigma)}_{\Lq{q}\left( 0,T; \widehat{W}_q^{-1}(\OM) \right)} \right. \nonumber \\
			& \left. \qquad \qquad\qquad \qquad\qquad \qquad \quad + \norm{g}_{\W{1-\frac{1}{q}, \onehalf \left( 1 - \frac{1}{q} \right)}(\Sigma \times (0,T))^n} + \norm{v_0}_{\W{2 - \frac{2}{q}}(\OM \backslash \Sigma)^n}
		\right). \nonumber 
	\end{align}
	Here, $ \widehat{W}_{q}^{-1}(\OM) $ is the space of all triples $ (\vp , \psi, \chi) \in \Lq{q}(\OM) \times \W{2 - 1/q}(\partial \OM)^n \times \W{2 - 1/q}(\Sigma)^n $, which enjoy the regularity property $ (\vp, \psi \cdot \nu_{\partial \OM}, \chi \cdot \nusigma) \in \dot{W}_q^{-1}(\OM) = (\dot{W}_{q'}^{1}(\OM))' $, where
	\begin{equation}
		\label{regularity property}
		\inner{(\vp, \psi \cdot \nu_{\partial \OM}, \chi \cdot \nusigma)}{\phi} := - \inner{\vp}{\phi}_{\OM} + \inner{\psi \cdot \nu_{\partial \OM}}{\phi}_{\partial \OM} + \inner{\chi \cdot \nusigma}{\phi}_{\Sigma},
	\end{equation}
	for all $ \phi \in \dot{W}_{q'}^{1}(\OM) $.
\end{proposition}
\begin{proof}
	We proceed to prove this theorem with Theorem 8.1.4 in \cite{PS2016}, by which we need some special treatments for \eqref{twophase: Dirichlet boundary}. The first one is to extend the quintuple $ (f_u, g_d, g_b, g_u, g) $ from $ \ZT $ to $ Z_{\infty} $. Notice that $ f_u \in \Lq{q}( 0,T; \Lq{q}(\OM) )^n $ without time derivatives, then we extend it by zero simply as a new function $ \bar{f}_u = \chi_{[0,T]} f_u \in \Lq{q}( 0,\infty; \Lq{q}(\OM) )^n $. Since $ g_d \in \Lq{q}( 0,T; \W{1}(\OM \backslash \Sigma) ) \cap \W{1}( 0,T; \W{-1}(\OM) )  $, by Theorem \ref{extension: general} with $ X_1 = \W{1}(\OM \backslash \Sigma) $, $ X_0 = \W{-1}(\OM) $, we obtain a new function $ \bar{g}_d := \cE(g_d) \in \Lq{q}( 0,\infty; \W{1}(\OM \backslash \Sigma) ) \cap \W{1}( 0,\infty; \W{-1}(\OM) ) $, which is uniformly bounded for $ T \leq T_0 $. For $ (g_b, g_u, g) \in \W{2 - 1/q, 1 - 1/2q}(\partial \OM \times (0,T))^n \times \W{2 - 1/q, 1 - 1/2q}(\Sigma \times (0,T))^n \times \W{1 - 1/q, ( 1 - 1/q )/2}(\Sigma \times (0,T))^n $, Theorem \ref{extension: general anisotropic} with $ \alpha = 1 - 1/2q > 1/q $ and $ (1 - 1/q)/2 > 1/q $ respectively imply that they can be extended as $ (\bar{g}_b, \bar{g}_u, \bar{g}) := \cE(g_b, g_u, g) \in \W{2 - 1/q, 1 - 1/2q}(\partial \OM \times (0,\infty))^n \times \W{2 - 1/q, 1 - 1/2q}(\Sigma \times (0,\infty))^n \times \W{1 - 1/q, ( 1 - 1/q )/2}(\Sigma \times (0,\infty))^n $, which are uniformly bounded for $ T \leq T_0 $. In summary,
	\begin{equation*}
		\rv{(\bar{f}_u, \bar{g}_d, \bar{g}_b, \bar{g}_u, \bar{g})}_{[0,T]} = (f_u, g_d, g_b, g_u, g)
	\end{equation*}
	and
	\begin{equation*}
		(\bar{f}_u, \bar{g}_d, \bar{g}_b, \bar{g}_u, \bar{g})
		\in
		Z_{\infty}.
	\end{equation*}
	
	Now, for a constant $ \omega > \omega_0 \geq 0 $, define 
	$$ (\tilde{f}_u, \tilde{g}_d, \tilde{g}_b, \tilde{g}_u, \tilde{g})(t) = e^{- \omega t} (\bar{f}_u, \bar{g}_d, \bar{g}_b, \bar{g}_u, \bar{g})(t), $$ 
	then it is easy to verify $ (\tilde{f}_u, \tilde{g}_d, \tilde{g}_b, \tilde{g}_u, \tilde{g}) $ is also contained in $ Z_{\infty} $, since $ e^{-\omega t} $ is smooth with respect to time $ t $. 
	
	Let $ (u, \pi) $ be the solution of (8.4) in \cite{PS2016} with $ (f_u, g_d, g_b, g_u, g) = (\tilde{f}_u, \tilde{g}_d, \tilde{g}_b, \tilde{g}_u, \tilde{g}) $ given above, as well as the constant viscosity $ \mu > 0 $ in \eqref{twophase: Dirichlet boundary}. For all $ t \in \bbr_+ $, define 
	\begin{equation*}
		v(t) = e^{\omega t} u(t), \quad p(t) = e^{\omega t} \pi(t),
	\end{equation*}
	then $ (v, p) $ solves \eqref{twophase: Dirichlet boundary} for $ t \in [0,T] $. Consequently, existence and regularity of $ (u, \pi) $, which are given by Theorem 8.1.4 in \cite{PS2016}, imply those of $ (v, p) $. Additionally, \eqref{propA.2: estimates} holds under our construction of $ (v, p) $.
	
	Finally, we need to show that our solution is unique. To this end, let $ (v_1, p_1) \neq (v_2, p_2) $ be two solutions of \eqref{twophase: Dirichlet boundary} in $ (0,T) $ with same source terms and initial value. Define $ (v, p) = (v_1 - v_2, p_1 - p_2) $. Since \eqref{twophase: Dirichlet boundary} is linear, then $ (v, p) $ satisfies
	\begin{equation}
		\label{twophase: Dirichlet boundary difference}
		\begin{aligned}
			\vr \pt v - \Div (2 \mu D v) + \nabla p & = 0, && \quad \text{in}\  \OM \backslash \Sigma \times (0, T), \\
			\Div v & = 0, && \quad \text{in}\  \OM \backslash \Sigma \times (0, T), \\
			v & = 0, && \quad \text{on}\  \partial \OM \times (0, T), \\
			\jump{v} & = 0, && \quad \text{on}\  \Sigma \times (0, T), \\
			\jump{- 2 \mu D v + p \bbI} \nusigma & = 0, && \quad \text{on}\  \Sigma \times (0, T), \\
			\rv{v}_{t = 0} & = 0, && \quad \text{in}\  \OM.
		\end{aligned}
	\end{equation}
	Multiplying the first equation of \eqref{twophase: Dirichlet boundary difference} by $ v $ and integrating by parts over $ \OM \backslash \Sigma \times (0,t) $, one obtains
	\begin{align*}
		& \int_{\OM \backslash \Sigma} \vr \abs{v(t)}^2 \d x 
		+ \int_0^t \int_{\OM \backslash \Sigma} 2 \mu \abs{Dv(x,t)}^2 \d x \d t
		= 0, \quad \text{for a.e. } t \in (0,T),
	\end{align*}
	which implies the uniqueness and completes the proof.
\end{proof}
\begin{remark}
	\label{twophase: hidden condition}
	For $ (g_d, g_b \cdot \nu_{\partial \OM}, g_u \cdot \nusigma) \in \W{1}( 0,T; \widehat{W}_{q}^{-1}(\OM) ) $, we notice that
	\begin{equation*}
		\int_{\OM} g_d \d x = \int_{\partial \OM} g_b \cdot \nu_{\partial \OM} \d (\partial \OM) - \int_{\Sigma} g_u \cdot \nusigma \d \Sigma,
	\end{equation*}
	when $ \phi = 1 $ in \eqref{regularity property}, the regularity property of $ \widehat{W}_{q}^{-1}(\OM) $. Thus, for the zero-Dirichlet problem, which means $ g_b = g_u = 0 $ in \eqref{twophase: Dirichlet boundary}, one has an hidden compatibility condition
	\begin{equation*}
		\int_{\OM} g_d \d x = 0.
	\end{equation*}
	This is an important condition when we solve the Stokes type problems with homogeneous Dirichlet boundary conditions.
\end{remark}
\subsection{Parabolic equations with Neumann boundary conditions}
\label{parabolic}
Thanks to general maximal regularity theory of parabolic problem, for example, Pr\"{u}ss and Simonett \cite[Section 6.3]{PS2016}, we give the solvability of parabolic systems with Neumann boundary conditions.
Let $ \OM \subset \bbr^n $, $ n \geq 1 $, be a domain with compact boundary $ \partial \OM $ of class $ C^2 $, we consider the following system
\begin{equation}
\label{parabolic: Neumann boundary}
	\begin{aligned}
		\pt u - D \Delta u & = f, && \quad \text{in}\  \OM \times (0, T), \\
		D \nabla u \cdot \nu & = g, && \quad \text{on}\  \partial \OM \times (0, T), \\
		\rv{u}_{t = 0} & = u_0, && \quad \text{in}\  \OM,
	\end{aligned}
\end{equation}
where $ u $ represents some physical property, for example, temperature or concentration. $ D $ is the diffusion coefficient. $ \nu $ denotes the unit outer normal vector on $ \partial \OM $. $ f $ and $ g $ are give functions standing for the source or reaction term. Now we state the proposition for \eqref{parabolic: Neumann boundary}.
\begin{proposition}
	\label{parabolic: proposition}
	Let $ \OM \subset \bbr^n $, $ n \geq 1 $, be a domain with compact boundary $ \partial \OM $ of class $ C^2 $, $ q > 3 $. Assume that 
	\begin{equation*}
		f \in \Lq{q}( 0,T; \Lq{q}(\OM) ), \quad
		g \in \W{1 - \frac{1}{q}, \onehalf \left( 1 - \frac{1}{q} \right)}(\OM \times (0,T)),
	\end{equation*}
	and $ u_0 \in \W{2 - \frac{2}{q}}(\OM) $ satisfying the compatibility condition
	\begin{equation*}
		\rv{D \nabla u_0 \cdot \nu}_{\partial \OM} = \rv{g}_{t = 0}.
	\end{equation*}
	Then there exists a unique solution $ u \in \W{2,1}(\OM \times (0,T)) $ of \eqref{parabolic: Neumann boundary}. Moreover, 
	\begin{align*}
		\norm{u}_{\W{2,1}(\OM \times (0,T))}
		\leq C \left( \norm{f}_{\Lq{q}( 0,T; \Lq{q}(\OM) )}
		+ \norm{g}_{\W{1 - \frac{1}{q}, \onehalf \left( 1 - \frac{1}{q} \right)}(\OM \times (0,T))}
		+ \norm{u_0}_{\W{2 - \frac{2}{q}}(\OM)} \right),
	\end{align*}
	where $ C $ does not depend on $ T \in (0,T_0] $ for any fixed $ 0 < T_0 < \infty $.
\end{proposition}
\begin{proof}
	This proposition can be easily shown by means of Pr\"{u}ss and Simonett \cite[Theorem 6.3.2]{PS2016}, for which we need to extend the right-hand sides just as in the proof of Proposition \ref{twophaseD: proposition} and construct a solution solving (6.45) in \cite{PS2016}. This can be done since we established general extension theorems in Appendix \ref{appendix:extension}.
\end{proof}
\subsection{Laplacian transmission problems with Dirichlet boundary}
In this section, we investigate a transmission problem for Laplacian equation with Dirichlet boundary condition, which reads
\begin{equation}
	\label{laplace transmission system}
	\begin{aligned}
		- \Delta \psi & = f && \quad \mathrm{in}\  \OM \backslash \Sigma, \\
		\jump{\partial_{\nu} \psi} & = g && \quad \mathrm{on}\  \Sigma, \\
		\jump{\psi} & = h && \quad \mathrm{on}\  \Sigma, \\
		\psi & = g_b && \quad \mathrm{on}\  \partial \OM.
	\end{aligned}
\end{equation}
Here, we denote the inner domain by $ \OM^- $, resp. outer domain by $ \OM^+ $ and the unit normal vector on $ \Sigma = \partial \OM^- $ by $ \nu $. 

The second result concerns the strong solutions.
\begin{proposition}
	\label{transmission laplace: strong proposition}
	Let $ 1 < q < \infty $, $ \OM \subset \bbr^n $, $ n \geq 2 $, with boundary $ \partial \OM $ of class $ C^{3-} $, and let $ \Sigma \subset \OM $ be a closed hypersurface of class $ C^{3-} $, $ s \in \{0,1\} $. Given functions $ f \in \W{s}(\OM \backslash \Sigma) $, $ g \in \W{1 + s - 1/q}(\Sigma) $, $ h \in \W{2 + s - 1/q}(\Sigma) $, $ g_b \in \W{2 + s - 1/q}(\partial \OM) $. Then the problem \eqref{laplace transmission system} admits a unique solution $ \psi \in \W{2 + s}(\OM \backslash \Sigma) $. Moreover, there is a constant $ C > 0 $ such that
	\begin{align*}
		\norm{\psi}_{\W{2 + s}}
		\leq C \left(
			\norm{f}_{\W{s}}
			+ \norm{g}_{\W{1 + s - \frac{1}{q}}}
			+ \norm{h}_{\W{2 + s - \frac{1}{q}}}
			+ \norm{g_b}_{\W{2 + s - \frac{1}{q}}}
		\right).
	\end{align*}
\end{proposition}
\begin{proof}
	\textbf{\textit{Step 1: Reduction.}} We firstly reduce to the case $ (h, g_b) = 0 $. To this end, we find a $ \vp $ solving
	\begin{equation*}
	\begin{aligned}
	- \Delta \vp & = 0 && \quad \mathrm{in}\  \OM^-, \\
	\vp & = h && \quad \mathrm{on}\  \Sigma,
	\end{aligned}
	\end{equation*}
	and
	\begin{equation*}
	\begin{aligned}
	- \Delta \vp & = 0 && \quad \mathrm{in}\  \OM^+, \\
	\vp & = 0 && \quad \mathrm{on}\  \Sigma, \\
	\vp & = g_b && \quad \mathrm{on}\  \partial \OM.
	\end{aligned}
	\end{equation*}
	The existence and uniqueness of these two systems are clear due to the elliptic theory. Thanks to the trace theorem, the extra outer normal derivatives terms on $ \Sigma $ enjoys the same regularities as $ g $. Substracting $ \varphi $ from $ \psi $, we can investigate the reduce system \eqref{laplace transmission system} with $ (h, g_b) = 0 $.
	\\\textbf{\textit{Step 2: Weak solution with $ L^2 $-setting.}} Now, let $ H^k = W^k_2 $ and $ H^k_0 = W^k_{2,0} $ for $ k \in \bbn $. Testing \eqref{laplace transmission system} by a function $ \phi \in H^1_0(\OM) $ and integrating by parts, one obtains
	\begin{equation*}
		\int_{\OM \backslash \Sigma} \nabla \psi \cdot \nabla \phi \d x
		= \int_{\OM \backslash \Sigma} f \phi \d x
		- \int_{\Sigma} g \phi \d \Sigma 
		=: \inner{F}{\phi}_{H^{-1} \times H^1_0},
	\end{equation*}
	as a result of regularities of $ f $ and $ g $. Lax-Milgram Lemma implies a unique weak solution $ \psi \in H^1_0(\OM) $ to \eqref{laplace transmission system} with $ (h, g_b) = 0 $. 
	\\\textbf{\textit{Step 3: Truncation.}} Since the problem \eqref{laplace transmission system} with the Neumann boundary condition on $ \partial \OM $ has been uniquely solved, see e.g. Pr\"{u}ss and Simonett \cite[Proposition 8.6.1]{PS2016}, we show the propostion by truncation methods. More specifically, define a cutoff function $ \eta \in C_0^\infty(\OM) $ such that
	\begin{equation*}
		\eta(x) = 
		\left\{
			\begin{aligned}
				& 1, \quad \text{in a neighborhood of}\  \OM^-,\\
				& 0, \quad \text{in a neighborhood of}\  \OM^+,
			\end{aligned}
		\right.
	\end{equation*}
	Decompose $ \psi = \eta \psi + (1 - \eta) \psi =: u_1 + u_2 $, where $ u_1 $ solves
	\begin{equation*}
		\begin{aligned}
			- \Delta u_1 & = \eta f - 2 \nabla \eta \cdot \nabla \psi + \psi \Delta \eta =: f^1 && \quad \mathrm{in}\  \OM \backslash \Sigma, \\
			\jump{\partial_{\nu} u_1} & = \jump{\partial_{\nu} \psi} = g && \quad \mathrm{on}\  \Sigma, \\
			\jump{u} & = \jump{\psi} = 0 && \quad \mathrm{on}\  \Sigma, \\
			\partial_{\nu} u_1 & = 0 && \quad \mathrm{on}\  \partial \OM,
		\end{aligned}
	\end{equation*}
	weakly, respectively $ u_2 $ solves
	\begin{equation*}
		\begin{aligned}
			- \Delta u_2 & = (1 - \eta) f + 2 \nabla \eta \cdot \nabla \psi - \psi \Delta \eta =: f^2 && \quad \mathrm{in}\  \OM, \\
			u_2 & = 0 && \quad \mathrm{on}\  \partial \OM.
		\end{aligned}
	\end{equation*}
	\\\textbf{\textit{Step 4: Improving the regularity.}} From Step 2, we already know that \eqref{laplace transmission system} admits a unique weak solution $ \psi $ enjoying the regularity $ \nabla \psi \in  \Lq{2}(\OM) $, which means $ f^i \in \Lq{2}(\OM) $ in Step 3. By classical elliptic theory and \cite{PS2016}, one obtains $ u_1 \in H^2(\OM \backslash \Sigma) $, $ u_2 \in H^1_0(\OM) \cap H^2(\OM) $. Then $ \psi \in H^1_0(\OM) \cap H^2(\OM \backslash \Sigma) $. Moreover,
	\begin{align*}
		\nabla \psi \in H^1(\OM \backslash \Sigma) 
		\hookrightarrow 
		\left\{
			\begin{aligned}
				& \Lq{p}(\OM \backslash \Sigma), && 1 \leq p < \infty, && n = 2, \\
				& \Lq{p}(\OM \backslash \Sigma), && 1 \leq p \leq p^* := \frac{2n}{n - 2}, && n > 2,
			\end{aligned}
		\right. 
	\end{align*}
	due to the Sobolev embedding Theorem.
	For $ n = 2 $, the right-hand side terms $ f^1 $ and $ f^2 $ in Step 3 are contained in $ \Lq{p}(\OM \backslash \Sigma) $, $ 1 \leq p < \infty $. Consequently with $ p = q $, Proposition 8.6.1 and Corollary 7.4.5 in Pr\"{u}ss and Simonett \cite{PS2016} indicate that $ u_1 \in \W{2}(\OM \backslash \Sigma) $ and $ u_2 \in \W{2}(\OM) $, which implies $ \psi \in \W{2}(\OM \backslash \Sigma) $.
	For $ n > 2 $, we have $ f^i \in \Lq{p^*} $, $ i = 1,2 $. Again by regularity results in \cite{PS2016}, we have $ u_1 \in W_{p^*}^2(\OM \backslash \Sigma) $ and $ u_2 \in W_{p^*}^2(\OM) $ and hence 
	\begin{align*}
		\nabla \psi \in W_{p^*}^1(\OM \backslash \Sigma)
		\hookrightarrow 
		\left\{
			\begin{aligned}
				& \Lq{p}(\OM \backslash \Sigma), && 1 \leq p < \infty, && n = q^*, \\
				& \Lq{p}(\OM \backslash \Sigma), && 1 \leq p \leq p^{**} := \frac{np^*}{n - p^*}, && n > p^*, \\
				& C^\alpha(\overline{\OM \backslash \Sigma}), && 0 < \alpha \leq 1 - \frac{n}{p^*} && 2 < n < p^*.
			\end{aligned}
		\right. 
	\end{align*}
	For the first and third cases, we find $ f^i \in \Lq{p}(\OM \backslash \Sigma) $, $ i = 1,2 $, $ 1 \leq p < \infty $, and then get the regularity of $ \psi $. For the second case, we know $ p^{**} = \frac{np^*}{n - p^*} > p^* $. Therefore, by bootstrapping argument, we can always increase the space index until $ \Lq{q} $. Thus, by Proposition 8.6.1 and Corollary 7.4.5 in Pr\"{u}ss and Simonett \cite{PS2016}, one obtains $ u_1 \in \W{2}(\OM \backslash \Sigma) $ and $ u_2 \in \W{2}(\OM) $, i.e., $ \psi \in \W{2}(\OM \backslash \Sigma) $ with the estimate
	\begin{gather*}
		\norm{\psi}_{\W{2}(\OM \backslash \Sigma)} 
		\leq C \left( \norm{f}_{\Lq{q}(\OM \backslash \Sigma)}
		+ \norm{g}_{\W{1 - \frac{1}{q}}(\Sigma)}
		+ \norm{h}_{\W{2 - \frac{1}{q}}(\Sigma)}
		+ \norm{g_b}_{\W{2 - \frac{1}{q}}(\partial \OM)} \right),
	\end{gather*}
	for some constant $ C > 0 $. Then as above, one gets $ f^i \in \W{1}(\OM \backslash \Sigma) $, $ i = 1,2 $. With the help of Proposition 8.6.1 and Corollary 7.4.5 in Pr\"{u}ss and Simonett \cite{PS2016}, we have the desired regularity and estimate with $ s = 1 $.
\end{proof}

\section{Extension of Sobolev-Slobodeckij space} \label{appendix:extension}
In this section, we are intended to construct an extension operator mapping from $ \W{s}(0,T; X) $ to $ \W{s}(0,\infty; X) $, where $ s \in (\frac{1}{q}, 1] $ and $ X $ is a Banach space. The main feature is that the extension constant does not depend on $ T > 0 $, compared to the extension theorem for general Sobolev-Slobodeckij spaces. The reason we made such modification here is that if the constant depends on $ T $, then the extended norm may blow up for small $ T $, which is the case we addressed in this paper. For example, in the proof of Theorem 5.4 in \cite{DPV2012}, the extension from $ \W{s}(\Omega) $ to $ \W{s}(\bbr) $ with $ 0 < s < 1 $, several smooth functions $ \psi_j $ satisfying $ 0 \leq \psi_j \leq 1 $ and $ \sum_{j = 0}^{k} \psi_j = 1 $ are chosen to construct the extension operator. In the case $ \abs{\Omega} \rightarrow 0 $, we have $ \nabla \psi_j \approx \frac{1}{\abs{\Omega}} \rightarrow \infty $, which means that the extension is not valid. To avoid such problem, we employ an the \textit{even} extension and make use of the embedding results in Simon \cite{Simon1990}. Now, we give the extension theorem.
\begin{theorem}
	\label{extesion: zero initial}
	Let $ q \geq 1 $, $ s = 0 $, or $ s \in (\frac{1}{q}, 1] $, $ T > 0 $ and $ X $ be a Banach space. Then there exists an extension operator $ E_T : {_0\W{s}}(0,T; X) \rightarrow \W{s}(0, \infty; X) $, where $ {_0\W{s}}(0,T; X) = \{ u \in \W{s}(0,T; X): \rv{u}_{t = 0} = 0, \text{ if } s > \frac{1}{q} \} $, such that $ \rv{E_T(u)}_{[0,T]} = u $ and
	\begin{equation*}
		\norm{E_T(u)}_{\W{s}(0,\infty; X)}
		\leq C \norm{u}_{{_0\W{s}}(0,T; X)},
	\end{equation*}
	where $ C > 0 $ depends on $ s $, $ q $ and does not depend on $ T $.
\end{theorem}
\begin{proof}
	The proof is divided into three cases, namely, $ s = 0 $, $ \frac{1}{q} < s < 1 $ and $ s = 1 $.
	\\\textbf{Case 1: $ s = 0 $.} In this situation, $ \W{s}(0,T; X) $ is just the Lebesgue space $ \Lq{q}(0,T; X) $, which does not contain any time regularity. Hence for any function $ u \in \Lq{q}(0,T; X) $, we can take the zero extension.
	\\\textbf{Case 2: $ s = 1 $.} With $ \rv{u}_{t = 0} = 0 $, we apply an even extension to $ u $ in $ [0,T] $ around $ T $ to $ [0, 2T] $ and zero extension for $ T > 2T $ such that the extended function $ \bar{u} $ is weakly differentiable with
	\begin{equation*}
		\pt \bar{u}(t) = 
		\left\{
			\begin{aligned}
			& \pt u(t), && \text{ if } 0 \leq t \leq T, \\
			& - \pt u(2T - t), && \text{ if } T < t \leq 2T,\\
			& 0, && \text{ if } t > 2T.
			\end{aligned}
		\right.
	\end{equation*}
	Then we have
	\begin{align*}
		\norm{\bar{u}}_{\W{1}(0, \infty; X)}
		= 2^{\frac{1}{q}} \norm{u}_{\W{1}(0, T; X)}.
	\end{align*}
	\\\textbf{Case 3: $ \frac{1}{q} < s < 1 $.} With the same extension as in \textbf{Case 2}, we define the same function $ \tilde{u} $. Now we are in the position to show $ \tilde{u} \in \W{s}(0, \infty; X) $, for which we only need to prove $ \seminorm{\tilde{u}}_{\W{s}(0, \infty; X)} \leq C \seminorm{u}_{\W{s}(0, T; X)} $, where $ C $ is independent of $ T $. 
	From the definition of Sobolev-Slobodeckij space,
	\begin{align*}
		& \seminorm{\tilde{u}}_{\W{s}(0, \infty; X)}^q \\
		& = \int_0^T \int_0^T \frac{\norm{u(t) - u(\tau)}_X^q}{\abs{t-\tau}^{1 + sq}} \d t \d \tau
		+ \int_T^{2T} \int_T^{2T} \frac{\norm{u(2T - t) - u(2T - \tau)}_X^q}{\abs{t-\tau}^{1 + sq}} \d t \d \tau \\
		& \quad + 2 \int_0^T \int_T^{2T} \frac{\norm{u(t) - u(2T - \tau)}_X^q}{\abs{t-\tau}^{1 + sq}} \d \tau \d t 
		+ 2 \int_0^{2T} \int_{2T}^\infty \frac{\norm{\tilde{u}(t)}_X^q}{\abs{t-\tau}^{1 + sq}} \d \tau \d t 
		=: \sum_{i = 1}^4 Q_i.
	\end{align*}
	It is clear that 
	\begin{align*}
		Q_1 + Q_2 = 2 \seminorm{u}_{\W{s}(0, T; X)}^q.
	\end{align*}
	Since $ \abs{t - \tau} \geq \abs{t - \left(2T - \tau\right)} $ with $ t \in [0, T] $ and $ \tau \in [T, 2T] $, we have
	\begin{align*}
		Q_3 \leq 2 \int_0^T \int_0^{T} \frac{\norm{u(t) - u(h)}_X^q}{\abs{t-h}^{1 + sq}} \d h \d t 
		= 2 \seminorm{u}_{\W{s}(0, T; X)}^q.
	\end{align*}
	Noticing that $ \rv{\tilde{u}}_{t = 2T} = 0 $ due to the even extension, we get 
	\begin{align*}
		Q_4 
		& = \frac{2}{sq} \int_0^{2T} \frac{\norm{\tilde{u}(2T - h) - \tilde{u}(2T)}_X^q}{h^{sq}} \d h \\
		& \leq \frac{2}{sq} \int_0^{2T} \left( \frac{\norm{\tilde{u}(\cdot - h) - \tilde{u}(\cdot)}_{\Lq{\infty}(h, 2T; X)}}{h^{s - \frac{1}{q}}} \right)^q \frac{\d h}{h} 
		= \frac{2}{sq} \seminorm{\tilde{u}}_{B^{s - \frac{1}{q}}_{\infty, q}(0, 2T; X)}^q,
	\end{align*}
	where the seminorm of $ B^s_{p, q}(0, T; X) $ is given by 
	\begin{equation*}
		\seminorm{f}_{B^s_{p, q}(0, T; X)}
		= \left( \int_0^T \left( \frac{\norm{\Dh f(t)}_{L^p(h,T; X)}}{h^{s}} \right)^q \frac{\d h}{h} \right)^{\frac{1}{q}}
	\end{equation*}
	for $ 0 < s <1 $ and $ 1 \leq p, q \leq \infty $. From Theorem 10 in Simon \cite{Simon1990}, we know that for $ \frac{1}{q} < s < 1 $ and $ q \geq 1 $,
	\begin{align*}
		\seminorm{f}_{B^{s - \frac{1}{q}}_{\infty, q}(0, T; X)}
		\leq \frac{3 \theta}{s - \frac{1}{q}} \seminorm{f}_{B^s_{q, q}(0, T; X)}
		= \frac{3 \theta}{s - \frac{1}{q}} \seminorm{f}_{\W{s}(0, T; X)}, \quad \forall f \in \W{s}(0, T; X),
	\end{align*}
	where $ \theta = 3^{1 - \left( s - 1/q \right)} $.
	Hence,
	\begin{align*}
		& Q_4 
		\leq \frac{6 \theta}{sq(sq - 1)} \seminorm{\tilde{u}}_{\W{s}(0,2T; X)}^q 
		\leq \frac{24 \theta}{sq(sq - 1)} \seminorm{u}_{\W{s}(0, T; X)}^q.
	\end{align*}
	Combining estimates of $ Q_i $, $ i = 1,\dots,4 $, one obtains
	\begin{align*}
		\seminorm{\tilde{u}}_{\W{s}(0, \infty; X)} \leq C \seminorm{u}_{\W{s}(0, T; X)},
	\end{align*}
	where $ C = \left( 4 + \frac{24 \theta}{sq(sq - 1)} \right)^{1/q} $.
	
	Now, let $ E_T(u) = \tilde{u} $, then $ E_T(u) $ is well-defined from $ {_0\W{s}}(0, T; X) $ to $ \W{s}(0, T; X) $ as well as $ \rv{E_T(u)}_{[0,T]} = u $ and
	\begin{equation*}
		\norm{E_T(u)}_{\W{s}(0,\infty; X)}
		\leq C \norm{u}_{{_0\W{s}}(0,T; X)},
	\end{equation*}
	where $ C > 0 $ depends on $ s $, $ q $ and does not depend on $ T $.
\end{proof}

Next, we give an extension theorem for general functions.
\begin{theorem}
	\label{extension: general}
	Let $ X_1, X_0 $ be two Banach spaces and $ X_1 \hookrightarrow X_0 $. For $ 1 < q < \infty $ and $  0 < T < \infty $, define $ X_T := \Lq{q}(0,T; X_1) \cap \W{1}(0,T; X_0) $ endowed with the norm
	\begin{equation*}
	\norm{u}_{X_T} := \norm{u}_{\Lq{q}(0,T; X_1)} + \norm{u}_{\W{1}(0,T; X_0)} + \norm{\rv{u}_{t = 0}}_{X_\gamma},
	\end{equation*}
	where $ X_\gamma = (X_0, X_1)_{1- 1/q, q} $. Then there exists an extension operator $ \cE \in \cL(X_T, X_\infty) $ satisfying $ \rv{\cE(u)}_{[0,T]} = u $, for all $ u \in X_T $. Moreover, there is a constant $ C > 0 $, independent of $  0 < T < \infty $, such that
	\begin{equation}
		\label{extension inequality: general}
		\norm{\cE(u)}_{X_\infty} \leq C \norm{u}_{X_T},
	\end{equation}
	for all $ u \in X_T $.
\end{theorem}
\begin{proof}
	First of all, we consider the case $ \rv{u}_{t = 0} = 0 $. Let $ E $ be the extension operator as in Theorem \ref{extesion: zero initial}. Define $ \tilde{u} = E(u) $, we have
	$ \rv{\tilde{u}}_{[0,T]} = u $ and
	\begin{equation*}
		\norm{\tilde{u}}_{X_\infty} \leq C \norm{u}_{X_T},
	\end{equation*}
	where $ C $ does not depend on $ T $.
	
	Let $ u_0 := \rv{u}_{t = 0} \in X_\gamma $. Since $ X_\gamma = (X_0, X_1)_{1- 1/q, q} $, the trace method of interpolation implies that there exists a function $ v \in X_\infty $ such that $ \rv{v}_{t = 0} = u_0 $. Moreover, it follows from the norm of $ X_T $ that there is a constant $ C > 0 $ such that
	\begin{equation*}
		\norm{v}_{X_\infty} \leq C \norm{\rv{u}_{t = 0}}_{X_\gamma} \leq C \norm{u}_{X_T}.
	\end{equation*}
	Now for general $ u \in X_T $, define $ w := u - v $, then $ w $ is reduced to the case $ \rv{w}_{t = 0} = 0 $ and can be extended to $ E(w) $ in $ X_\infty $ like $ \tilde{u} $. Now we define the extension operator as $ \cE(u) := w + v $. Then one obtains $ \rv{\cE(u)}_{[0,T]} = u $ and there is a constant, independent of $ T $, such that
	\begin{equation*}
		\norm{\cE(u)}_{X_\infty} \leq C \norm{w}_{X_\infty} + C \norm{v}_{X_\infty} \leq C \norm{u}_{X_T},
	\end{equation*}
	for all $ u \in X_T $, which completes the proof.
\end{proof}
With a similar argument, we have the following extension theorem for functions in $ \W{2 \alpha, \alpha} $.
\begin{theorem}
	\label{extension: general anisotropic}
	Let $ \Sigma $ be a compact sufficiently smooth hypersurface. For $ 1 < q < \infty $, $ 1/q < \alpha \leq 1 $ and $  0 < T < \infty $, if $ \W{2\alpha, \alpha}(\Sigma \times (0,T)) := \Lq{q}(0,T; \W{2 \alpha}(\Sigma)) \cap \W{\alpha}(0,T; \Lq{q}(\Sigma)) $ is endowed with norm
	\begin{equation*}
		\norm{g}_{\W{2\alpha, \alpha}(\Sigma \times (0,T))} := \norm{g}_{\Lq{q}(0,T; \W{2 \alpha}(\Sigma))} + \norm{g}_{\W{\alpha}(0,T; \Lq{q}(\Sigma))} + \norm{\rv{g}_{t = 0}}_{\W{2\alpha - \frac{2}{q}}(\Sigma)},
	\end{equation*}
	then for $ g \in \W{2\alpha, \alpha}(\Sigma \times (0,T)) $, there exists an extension operator $ \cE \in \cL(\W{2\alpha, \alpha}(\Sigma \times (0,T)), \W{2\alpha, \alpha}(\Sigma \times (0,\infty))) $ satisfying $ \rv{\cE(g)}_{[0,T]} = g $. Moreover, there is a constant $ C > 0 $, independent of $  0 < T < \infty $, such that
	\begin{equation}
		\label{extension inequality: general anisotropic}
		\norm{\cE(g)}_{\W{2\alpha, \alpha}(\Sigma \times (0,\infty))} \leq C \norm{g}_{\W{2\alpha, \alpha}(\Sigma \times (0,T))}.
	\end{equation}
\end{theorem}
\begin{remark}
	The proof is similar to what in Theorem \ref{extension: general}, for which it relies on Theorem \ref{extesion: zero initial} for $ 1/q < \alpha < 1 $ and the trace method interpolation, namely,
	\begin{equation*}
		\W{2\alpha - \frac{2}{q}}(\Sigma) 
		= \left\{
			g(0) : g \in \Lq{q}(0,T; \W{2 \alpha}(\Sigma)) \cap \W{\alpha}(0,T; \Lq{q}(\Sigma))
		\right\},
	\end{equation*}
	see e.g., Lemma \ref{embedding: W alpha} or \cite[Example 3.4.9(i)]{PS2016}. These results can also be extended to more general anisotropic Sobolev-Slobodeckij spaces with general trace theorem, see e.g., \cite[Theorem 3.4.8]{PS2016}.
\end{remark}

\end{document}